\documentclass[a4paper,reqno,11pt]{amsart}

\usepackage{fullpage}

\usepackage{amssymb}
\usepackage{dsfont}

\usepackage{tikz}
\usetikzlibrary{arrows.meta,calc,decorations.markings}
\tikzset{
    >=Stealth,
    every path/.style={shorten >=1.2pt, shorten <=1.2pt},
    midarrow/.style={
        postaction={decorate},
        decoration={
            markings,
            mark=at position 0.5 with {\arrow{Stealth}}
        }
    }
}

\usepackage[utf8]{inputenc} 
\usepackage[T1]{fontenc} 

\usepackage{hyperref}

\usepackage[nobysame,alphabetic,initials,msc-links]{amsrefs}

\DefineSimpleKey{bib}{how}

\renewcommand{\eprint}[1]{#1}
\BibSpec{misc}{%
  +{}{\PrintAuthors}  {author}
  +{,}{ \textit}      {title}
  +{,}{ }             {how}
  +{}{ \parenthesize} {date}
  +{,} { available at \eprint}        {eprint}
  +{,}{ available at \url}{url}
  +{,}{ }             {note}
  +{.}{}              {transition}
}

\numberwithin{equation}{section}

\theoremstyle{plain}
\newtheorem{thm}{Theorem}[section]
\newtheorem{prop}[thm]{Proposition}
\newtheorem{lemma}[thm]{Lemma}

\newtheorem{cor}[thm]{Corollary}
\theoremstyle{definition}

\newtheorem{defn}[thm]{Definition}
\theoremstyle{remark}

\newtheorem{question}[thm]{Question}
\newtheorem{remark}[thm]{Remark}
\newtheorem{example}[thm]{Example}

\theoremstyle{plain}

\newcommand\bp{\begin{proof}}
\newcommand\ep{\end{proof}}

\newcommand{\un}{\mathds{1}}
\newcommand\dach{{\!\widehat{\ \ }}}

\newcommand\C{\mathbb{C}}
\newcommand\N{\mathbb{N}}

\newcommand\R{\mathbb{R}}
\newcommand\T{\mathbb{T}}
\newcommand\Z{\mathbb{Z}}

\newcommand{\BV}{\mathcal{BV}}

\newcommand{\G}{\mathcal{G}}
\newcommand{\HH}{\mathcal{H}}
\newcommand{\LL}{\mathcal{L}}
\newcommand{\M}{\mathcal{M}}
\newcommand{\OO}{\mathcal{O}}
\newcommand{\QQ}{\mathcal{Q}}

\newcommand{\TT}{\mathcal{T}}

\newcommand\Ad{\operatorname{Ad}}
\newcommand\Aut{\operatorname{Aut}}

\newcommand{\Gu}{{\mathcal{G}^{(0)}}}
\newcommand{\Gxx}{\mathcal{G}^x_x}

\newcommand\hull{\operatorname{hull}}

\newcommand\Ind{\operatorname{Ind}}

\newcommand\MT{\operatorname{M}}
\newcommand\Stab{\operatorname{Stab}}

\newcommand\PerG{\operatorname{Per}_{G}}
\newcommand\prim{\mathrm{prim}}
\newcommand\Prim{\operatorname{Prim}}

\newcommand\sing{\mathrm{sing}}

\newcommand\eps{\varepsilon}
\newcommand\Iso{\operatorname{Iso}(\mathcal G)^{\circ}}
\newcommand\IsoG{\operatorname{Iso}(\mathcal G)}
\newcommand\IsoGx[1]{\operatorname{Iso}(\mathcal G_{\overline{[#1]}})}

\newcommand\ee{\nopagebreak\mbox{\ }\hfill$\diamond$}

\begin{document}

\title{The ideal structure of Exel--Pardo algebras and their higher rank analogues}

\date{May 22, 2026}

\author{Johannes Christensen}
\address{Department of Mathematics, Aarhus University, Denmark}
\email{johannes@math.au.dk}

\author{Sergey Neshveyev}
\address{Department of Mathematics, University of Oslo, Norway}
\email{sergeyn@math.uio.no}

\thanks{The authors would like to thank the Isaac Newton Institute for Mathematical Sciences, Cambridge, and the School of Mathematics and Physics of Queen's University Belfast  for support and hospitality during the programme \emph{Topological groupoids and their C$^*$-algebras}, where a part of the work on this paper was undertaken. This work was supported by EPSRC grant EP/V521929/1.}

\thanks{J.C. is supported by the research grant (VIL72080) from Villum Fonden.}

\begin{abstract}
Given a pseudo-free self-similar action of a countable group $G$ on a countable directed graph~$E$ with amenable stabilizers of the vertices, we identify the exact conditions under which these stabilizers do not contribute to the ideal structure of the corresponding Exel--Pardo algebra~$\OO_{G,E}$. Under these conditions, we give a complete description of the primitive ideal space of $\OO_{G,E}$ in graph-theoretic terms. Our results apply in particular to certain crossed products $\OO_E\rtimes G$, where~$G$ acts on~$E$ by graph automorphisms. When $G$ is trivial, this recovers Hong--Szyma\'nski's description of the ideal structure of the Cuntz--Krieger algebras~$\OO_E$. Similar results are then obtained for self-similar actions of groups on row-finite higher rank graphs without sources.

In order to obtain these results we formalize the notion of a graded groupoid with essentially central isotropy, which generalizes essentially principal groupoids and groupoids injectively graded by abelian groups. Under the amenability and second countability assumptions, we describe the primitive ideal spaces of the corresponding C$^*$-algebras as topological spaces.
\end{abstract}

\maketitle

\section*{Introduction}
Understanding the ideal structure is a fundamental problem in the analysis of nonsimple C$^*$-algebras. The research on this problem has a long history. In the context of C$^*$-algebras of \'etale groupoids, which this paper deals with, the first major result was obtained by Effros and Hahn~\cite{EH}, who showed that if $\Gamma$ is a countable amenable group acting freely on a second countable locally compact space $X$, then the primitive ideal space of $C_0(X)\rtimes\Gamma$ is homeomorphic to the quasi-orbit space $(\Gamma\backslash X)^\sim$. This implies that the lattice of ideals in $C_0(X)\rtimes\Gamma$ is isomorphic to the lattice of $\Gamma$-invariant open subsets of $X$. Using different methods, Renault showed in his thesis~\cite{Rbook} that  a similar isomorphism holds for all amenable second countable \'etale groupoids~$\G$ and, moreover, instead of freeness of the action $\G\curvearrowright\Gu$, that is, principality of~$\G$, it suffices to assume that~$\G$ is  \emph{essentially principal}, meaning that for every invariant closed subset $X\subset\Gu$ the set of units $x\in X$ with trivial isotropy groups $\Gxx$ is dense in $X$. This result was essentially rediscovered by Kawamura and Tomiyama~\cite{MR1088230} for transformation groupoids. Their paper implies that a version of Renault's theorem for such groupoids is true beyond the second countable case and, arguably more importantly, the assumption of essential principality is not only sufficient but is also necessary for the result to be true. For \'etale groupoid C$^*$-algebras this necessity was later also proved by Brown, Clark, Farthing and Sims~\cite{MR3189105}. A short elegant proof of these results for transformation groupoids has been given by Archbold and Spielberg~\cite{MR1258035}, their arguments readily generalize to \'etale groupoids~\cite{CN4}. Let us also mention that in the recent years the progress in noncommutative boundary theory has led to a similar description of the ideal structure of $C^*_r(\G)$ for some nonamenable groupoids that are not essentially principal~\cite{KKLRU}, but in this paper we are primarily interested in the amenable case.

Over the last almost fifty years Renault's theorem has been the main tool 
to understand the ideal structure, and in particular simplicity, of~$C^*(\G)$ for amenable second countable \'etale groupoids~$\G$. In the same period there have emerged a number of important examples of amenable groupoids that are not always essentially principal, yet their C$^*$-algebras have a tractable ideal structure, such as groupoids underlying Cuntz--Krieger algebras and their various generalizations. In the absence of a clear strategy how to deal with the non-essentially-principal case, the corresponding C$^*$-algebras have been often analyzed by ad hoc methods. In particular, a complete description of the ideal structure of C$^*$-algebras of countable directed graphs was obtained without using groupoids by Hong and Szyma\'nski~\cite{HS}, culminating around twenty years of research on the topic.

The lack of a strategy does not mean that there have been no general results on the ideal structure of $C^*(\G)$. In fact, the origins of one of the most fundamental results on this topic predate Renault's theorem. Specifically, it was conjectured by Effros and Hahn~\cite{EH} that in the amenable second countable case every primitive ideal of $C_0(X)\rtimes\Gamma$ is induced from an isotropy group. This was proved by Sauvageot~\cite{Sau}, and later the result was extended to amenable groupoids by Ionescu and Williams~\cite{IW}. Therefore it has been known for quite some time that the space $\Prim C^*(\G)$ of primitive ideals is a quotient of the set $\Stab(\G)^\prim$ of pairs $(x,J)$ such that $x\in\Gu$ and $J\in\Prim C^*(\Gxx)$. What is, however, still missing is a description of the topology on $\Prim C^*(\G)$ in terms of $\Stab(\G)^\prim$. In general, the set $\Stab(\G)^\prim$ does not carry any obvious useful topology and the measure-theoretic methods of the proof of the Effros--Hahn conjecture do not give a clue how to describe the topology on $\Prim C^*(\G)$ in terms of this set.

Until recently the main two cases in which it was known how to describe the topology on $\Prim C^*(\G)$ in terms of $\Stab(\G)^\prim$ were amenable transformation groupoids $\G=\Gamma\ltimes X$ such that either all stabilizers of the action $\Gamma\curvearrowright X$ are contained in one abelian subgroup of~$\Gamma$~\cite{MR0617538}, or the action is proper~\cite{EE} (and so in particular the stabilizers are finite); see also the introduction to~\cite{CN3} for a discussion of other related results, in particular, on certain classes of Deaconu--Renault groupoids~\cites{Kat,MR4887755}. In \cites{CN3,CN4} we have shown how to describe the topology on $\Prim C^*(\G)$ for the amenable \'etale groupoids $\G$ such that every isotropy group~$\Gxx$ belongs to a certain class~$\M$ that includes all countable groups of local polynomial growth. Although Renault's theorem can be proved by similar methods (see the discussion at the end of~\cite{CN4}*{Section~4.5}), it is not formally covered by our results, since they require conditions on \emph{every} isotropy group~$\Gxx$. It is then natural to ask whether it is possible to relax our conditions and allow a certain amount of ``bad'' units $x$ where there are no restrictions on $\Gxx$. 

We do not know an answer to this question in a generality we would be satisfied with, but our goal in this paper is to give at least a partial answer and illustrate it with some nontrivial examples. For this we introduce a class of \emph{graded groupoids with essentially central isotropy}. For the purpose of this introduction, and since this will be enough for our main examples, let us give the definition of a more restricted class of groupoids that are graded by abelian groups such that the grading is \emph{essentially injective}. Recall that for a discrete group $\Gamma$ by a $\Gamma$-grading of~$\G$  one means a continuous homomorphism $\Phi\colon\G\to\Gamma$. Denote by $\Omega\subset\Gu$ the subset of units $x$ such that $\Phi$ is injective on $\Gxx$. We say that the grading $\Phi$ is essentially injective if $\Omega\cap X$ is dense in $X$ for every invariant closed subset $X\subset\Gu$. When~$\Gamma$ is trivial, this gives back the definition of essential principality of $\G$.

Assuming that $\Gamma$ is abelian, we can induce the representations $\chi\circ\Phi|_{\Gxx}$ for~$\chi\in\widehat\Gamma$ to irreducible representations of~$C^*(\G)$, so that we get a map
\begin{equation}\tag{$*$}\label{eq:i}
\Gu\times\widehat\Gamma\to\Prim C^*(\G).
\end{equation}
Strengthening a result from~\cite{CN2}, we show that if the grading is essentially injective and~$\G$~is amenable and second countable, then the restriction of  the map~\eqref{eq:i} to $\Omega\times\widehat\Gamma$ is surjective, equivalently, the restriction of the induction map $\Stab(\G)^\prim\to\Prim C^*(\G)$ to the subset of pairs $(x,J)$ with $x\in\Omega$ is surjective. Once this is proved, results of \cites{CN3,CN4} apply and we get a description of the topological space $\Prim C^*(\G)$ as a quotient of the space $(\G\backslash\Omega)\times\widehat\Gamma$ equipped with a topology that is in general weaker than the product-topology. Moreover, in the end we do not even have to distinguish between good and bad units and can also describe $\Prim C^*(\G)$ either as a quotient of $(\G\backslash\Gu)\times\widehat\Gamma$ or as a quotient of $\G\backslash \Stab(\G)^\prim$ with respect to suitable topologies on these spaces. When $\Gamma$ is trivial, this implies Renault's theorem. Furthermore, similarly to how essential principality is a necessary condition in Renault's theorem, we show that essential injectivity of the grading is necessary for surjectivity of the map~\eqref{eq:i} and hence for our description of $\Prim C^*(\G)$ as a quotient of $(\G\backslash\Gu)\times\widehat\Gamma$.

We were led to the essentially injectively graded groupoids 
by our attempt to understand the ideal structure of Exel--Pardo algebras $\OO_{G,E}$, which are defined from self-similar actions $G\curvearrowright E$ of groups on directed graphs~\cite{MR3581326}. Under suitable conditions these algebras admit Hausdorff groupoid models $\G_{G,E}$, which can be thought of as semidirect products of the groupoids~$\G_E$ underlying the graph C$^*$-algebras $\OO_E$ and transformation groupoids $G\ltimes\partial E$, where $\partial E=\G_E^{(0)}$ is the space of \emph{boundary paths} in $E$. An interesting feature of the construction is that purely algebraically each groupoid $\G_{G,E}$ is injectively graded by the so-called lag group, but this grading is in general discontinuous and what makes sense topologically is only a grading $\Phi\colon\G_{G,E}\to\Z$ that is trivial on $G\ltimes\partial E$, while on $\G_E$ it coincides with the standard grading defining the gauge action on $\OO_E$. More generally, Li and Yang~\cite{MR4294118} defined algebras $\OO_{G,\Lambda}$ and the corresponding groupoids~$\G_{G,\Lambda}$ from self-similar actions $G\curvearrowright\Lambda$ of groups on $k$-graphs. In this case the standard grading $\Phi$ on~$\G_{G,\Lambda}$ is $\Z^k$-valued.

Simplicity criteria for the C$^*$-algebras $\OO_{G,\Lambda}$ have been obtained in~\cites{MR3581326,MR4294118,MR4283280}, but apart from this not much is known about their ideal structure even when $G$ acts on $\Lambda$ by graph automorphisms, in which case $\OO_{G,\Lambda}\cong\OO_\Lambda\rtimes G$. A theorem of Kumjian and Pask~\cite{MR1738948} states that if a countable group $G$ acts freely by graph automorphisms on a locally finite graph $E$, then $\OO_E\rtimes G$ is strongly Morita equivalent to~$\OO_{E/G}$. As was argued by Marelli and Raeburn~\cite{MR2495260}, the structure of $\OO_E\rtimes G$ becomes much more mysterious when one allows even a small degree of nonfreeness of the action. Under some assumptions, the ideal structure of self-similar $k$-graph C$^*$-algebras $\OO_{G,\Lambda}$ has been described by Li and Yang~\cite{MR4283280}. Although the conditions in~\cite{MR4283280}, which include the requirement that $G$ acts trivially on the vertices of $\Lambda$, are rather restrictive, they still allow situations where the groupoid~$\G_{G,\Lambda}$ is not essentially principal and the $\Z^k$-grading is not injective, so that the results of~\cite{MR4283280} cannot be deduced from Renault's theorem or~\cites{CN3}.

By construction, the kernels of $\Phi$ restricted to the isotropy groups of $\G_{G,\Lambda}$ are built out of subgroups of the $G$-stabilizers of vertices of $\Lambda$. It follows that in principle our results in~\cite{CN4} should allow one to understand the ideal structure of  $\OO_{G,\Lambda}$ for a large class of self-similar $k$-graphs, namely, it suffices to assume that the $G$-stabilizer of every vertex has local polynomial growth. In practice this might quickly become a daunting task involving a careful analysis of the isotropy groups and their representation theory for all boundary paths. On the other hand, the results of~\cite{MR4283280} imply that such an analysis is not always necessary, and indeed we can conclude now that, under amenability and second countability assumptions on $\G_{G,\Lambda}$, the vertex stabilizers and their contributions to the isotropy groups can be ignored in the analysis of the primitive spectrum exactly when the grading $\Phi\colon\G_{G,\Lambda}\to\Z^k$ is essentially injective. In this paper we describe  when this happens in terms of the self-similar action $G\curvearrowright\Lambda$ and then give a description of the primitive spectrum in graph-theoretic terms.

\smallskip

The paper consists of five sections. Section~\ref{sec:prelim0} explains our conventions and collects various auxiliary results on groupoids. Among other things, here we clarify the connections between different classes of ideals that have appeared in the literature, such as dynamical, diagonal-invariant and gauge-invariant ones. In particular, we show that in the amenable second countable case all gauge-invariant ideals of a graded groupoid C$^*$-algebra are dynamical if and only if the grading is essentially injective, which can be viewed as a graded version of Renault's theorem.

In Section~\ref{sec:ess-isotropy} we introduce our class of graded groupoids with essentially central isotropy. 
After establishing some basic properties of such groupoids $\G$, we explain how to adjust the arguments in~\cite{CN2} to classify the primitive ideals of $C^*(\G)$ when $\G$ is amenable and second countable. We then apply results of~\cite{CN3} to ``good'' units of $\G$ to describe the topology on $\Prim C^*(\G)$.

In Section~\ref{sec:prelim} we discuss the Exel--Pardo algebras and their higher rank analogues. There are different ways of introducing these algebras, which in the higher rank case have been shown to be equivalent only under certain conditions. We discuss exclusively the groupoid approach but make no restrictions on the graphs (apart from the standard assumption of finite alignment in the higher rank case), leaving a comparison with other approaches to the interested reader or possibly another occasion. For a number of results in the subsequent sections, in order to deal with Hausdorff groupoids, we need to assume that the self-similar actions that we consider are \emph{pseudo-free}, as defined in~\cite{MR3581326}. This assumption is trivially satisfied for the actions by graph automorphisms, but it does exclude some interesting examples of self-similar actions.

The main new result of this section is a characterization of amenability of the groupoids $\G_{G,\Lambda}$ for pseudo-free self-similar actions $G\curvearrowright\Lambda$: we show that if $\Lambda$ is an ordinary directed graph or a row-finite higher rank graph, then the groupoid $\G_{G,\Lambda}$ is amenable if and only if the $G$-stabilizer of every vertex of $\Lambda$ is amenable.

In Section~\ref{sec:1-graphs} we study the groupoids $\G_{G,E}$ for directed graphs $E$. Overall the approach here is similar to our analysis of the groupoids $\G_E$ in~\cite{CN3}. The key result is a description of the quasi-orbit space of $\G_{G,E}$ in terms of $G$-equivariant versions of notions that have played an important role in the analysis of graph C$^*$-algebras~\cite{MR1988256}, such as maximal tails and breaking vertices. This result does not need $\G_{G,E}$ to be Hausdorff and is valid without pseudo-freeness. The action of $G$ does create additional difficulties, so that we had to rethink the strategy  in~\cite{CN3}. An outcome of this effort is that even when $G$ is trivial, our description of the quasi-orbit space improves on that given in~\cite{CN3}, which used Baire category arguments to avoid analyzing periodic paths with nondiscrete orbits.

Once the quasi-orbit space of $\G_{G,E}$ is described, it is not difficult to understand when the $\Z$-grading is essentially injective. Informally, this happens if and only if there are no substantial parts of $E$ on which some element $g\ne1_G$ acts trivially. In such cases we can apply results of Section~\ref{sec:ess-isotropy} and obtain a complete description of $\Prim C^*(\G_{G,E})$ for amenable $\G_{G,E}$. When~$G$ is trivial, we get the same result as in~\cite{CN3}, which is equivalent to Hong--Szyma\'nski's description of $\Prim \OO_E$.

In Section~\ref{sec:k-graphs} we extend results of the previous section to row-finite higher rank graphs without sources. The main new ingredient that we need here is a $G$-equivariant version of certain vertex sets introduced by Carlsen, Kang, Shotwell and Sims in~\cite{MR3150171}, which should be thought of as higher rank substitutes for \emph{$G$-circuits without an entry} analyzed in~\cite{MR3581326}.

\bigskip

\section{Preliminaries}\label{sec:prelim0}

\subsection{\'Etale groupoids, their \texorpdfstring{C$^*$}{C*}-algebras, and induction}\label{ssec:groupoids}
In the following we will give a very brief introduction to the results and terminology on \'etale groupoids and induction that we will need in the paper. For a more detailed account of \'etale groupoids we refer the reader to~\cite{Rbook}, and for more on induction to the preliminary sections of ~\cites{CN2,CN3}.

We use the same conventions and terminology as in \cite{CN3} and work with Hausdorff locally compact \'etale groupoids~$\G$, with range and source maps denoted by~$r$ and~$s$. \'Etaleness means that these maps are local homeomorphisms. Additional assumptions on $\G$ such as second countability are formulated explicitly when they are needed. The full groupoid C$^*$-algebra $C^*(\G)$ is by definition a universal C$^*$-completion of the $*$-algebra $C_c(\G)$, where the product is defined by
\begin{equation*} \label{eprod}
(f_{1}*f_{2})(g) := \sum_{h \in \G^{r(g)}} f_{1}(h) f_{2}(h^{-1}g)
\end{equation*}
and involution by $f^{*}(g):=\overline{f(g^{-1})}$.

\smallskip

Given a unit $x\in\Gu$, a subgroup $S\subset\Gxx$ and a unitary representation $\pi\colon S\to U(H)$ on a Hilbert space $H$, we define the induced representation $\Ind\pi=\Ind^\G_S\pi$ of $C^*(\G)$ as follows. The underlying Hilbert space $\Ind H$ of $\Ind\pi$ consists of the functions $\xi \colon  \G_{x} \to H$ such that
\begin{equation*}
\xi(gh)=\pi(h)^{*}\xi(g)\quad (g \in \G_{x},\ h\in S)\quad\text{and}\quad  \sum_{g\in \G_{x}/S}\lVert \xi(g)\rVert^{2}<\infty,
\end{equation*}
and the inner product on $\Ind H$ is given by
\begin{equation*}
(\xi_{1}, \xi_{2} ):=\sum_{g\in \G_{x}/S}( \xi_{1}(g), \xi_{2}(g)).
\end{equation*}
The induced representation is then defined by
\begin{equation}\label{eq:ind-rep}
\big((\Ind\pi)(f)\xi\big)(g) :=
\sum_{h\in \G^{r(g)}}f(h) \xi(h^{-1}g)\quad (g\in\G_x,\ \xi\in\Ind H,\ f\in C_c(\G)).
\end{equation}

Recall that given two representations $\pi$ and $\rho$ of a C$^*$-algebra, one says that $\pi$ is weakly contained in $\rho$, and writes $\pi\prec\rho$, if $\ker\rho\subset\ker\pi$. If $\pi$ is a unitary representation of a discrete group~$S$, then by $\ker\pi$ we mean the kernel of the extension of $\pi$ to $C^*(S)$. Induction respects weak containment of representations, which implies that the kernel of $\Ind\pi$ depends only on the kernel of $\pi$.

We thus get a well-defined procedure of inducing ideals of $C^*(S)$: if $\ker\pi=J\subset C^*(S)$, then $\Ind^{\G}_SJ:=\ker \Ind^\G_S\pi$ is an ideal in $C^*(\G)$. Note that by an ideal in a C$^{*}$-algebra we always mean a closed two-sided ideal.

If $\pi$ is an irreducible unitary representation of $\Gxx$, then  $\Ind^{\G}_{\Gxx}\pi$ is irreducible as well. Denote by $\Stab(\G)^\prim$ the set of pairs $(x,J)$ such that $x\in\Gu$ and $J$ is a primitive ideal in $C^*(\Gxx)$. We then get a well-defined map
\begin{equation}\label{eq:Ind}
\Ind\colon\Stab(\G)^\prim\to\Prim C^*(\G),\quad \Ind(x,J):=\Ind^\G_{\Gxx}J=\ker\Ind^\G_{\Gxx}\pi_J,
\end{equation}
where $\pi_J$ denotes any irreducible representation of $C^*(\Gxx)$ with kernel $J$. We have an action of~$\G$ on $\Stab(\G)^\prim$ defined by
$$
g(x,J):=(r(g),(\Ad g)(J))\quad (g\in\G_x).
$$
The induction map $\Ind$ is constant on every $\G$-orbit.

\smallskip

For every $x\in\Gu$, define a regular representation $\rho_x:=\Ind^\G_{\{x\}}\epsilon_x$, where $\epsilon_x$ denotes the trivial representation.
For subgroups $H\subset S \subset \Gxx$ and any unitary representation $\pi$ of $H$, there is a unitary equivalence between $\Ind_{S}^{\G}(\Ind_{H}^{S} \pi)$ and $\Ind_{H}^{\G} \pi$. Since $\Ind^S_{\{x\}}\epsilon_x$ is unitarily equivalent to the regular representation $\lambda_S$ of $S$, it follows that up to unitary equivalence the representation~$\rho_x$ can also be defined as $\Ind^\G_S\lambda_S$ for any subgroup $S\subset\Gxx$. The reduced C$^*$-algebra~$C^*_r(\G)$ is defined as the completion of $C_c(\G)$ with respect to the norm
$$
\|f\|_r:=\sup_{x\in\Gu}\|\rho_x(f)\|.
$$
The map $C_c(\G)\ni f\mapsto f|_{\Gu}$ extends to a faithful conditional expectation $E\colon C^*_r(\G)\to C_0(\Gu)$.

\subsection{Amenable groupoids}\label{ssec:amenable-groupoids}
A Hausdorff locally compact \'etale groupoid $\G$ is called \emph{amenable} if there is a net $(f_i)_i$ of positive continuous functions on $\G$ such that
$\sum_{g\in\G_x}f_i(g)=1$ for all $i$ and $x\in\Gu$ and
\begin{equation*}\label{eq:amenab-def}
\sum_{g\in\G_{r(h)}}|f_i(gh)-f_i(g)|\xrightarrow[i]{}0
\end{equation*}
for all $h\in\G$, uniformly on compact sets. In this case $(f_i)_i$ is called an \emph{approximate invariant density}.
An \'etale groupoid $\G$ is amenable if and only if $C^*_r(\G)$ is nuclear, in which case $C^*(\G)=C^*_r(\G)$~\cite{MR2391387}*{Section~5.6}.

\begin{remark}
There are several other equivalent ways to define amenability. In particular, the definition used in~\cite{MR2391387} requires existence of a net of positive functions $(g_{i})_{i}$ in $C_{c}(\G)$ such that $\sum_{h\in \G_{x}} g_{i}(h)\to1$ uniformly in $x$ on compact subsets of $\Gu$ and
$$
\sum_{\gamma\in \G_{r(h)}} |g_{i}(\gamma h) - g_{i}(\gamma )|\to0
$$
uniformly in $h$ on compact subsets of $\G$. Given an approximate invariant density $(f_i)_i$, as a net satisfying these properties we can take $g_{\theta,i}(h):=\theta(h) f_i(h)$ for $(\theta,i)\in\Lambda\times I$, where $\Lambda:=\{\theta\in C_c(\G):0\le\theta\le 1\}$ is considered with its standard order. Conversely, given a net $(g_i)_i$ as above, put $\theta_i(x):=\sum_{h\in\G_x}g_i(h)$ and
$$
f_{i}(h):=
\begin{cases}
\theta_i(s(h))^{-1}g_{i}(h), &\text{if}\ \ \theta_i(s(h))\ge 1,\\
g_{i}(h), &\text{if}\ \ \theta_i(s(h))<1\ \ \text{and}\ \ h\notin\Gu,\\
g_{i}(x)+1-\theta_i(x), &\text{if}\ \ \theta_i(s(h))<1\ \ \text{and}\ \ h=x \in \Gu.
\end{cases}
$$
It is not difficult to check that $(f_i)_i$ is an approximate invariant density. \ee
\end{remark}

An action $\Gamma\curvearrowright X$ of a discrete group $\Gamma$ on a locally compact space $X$ is called amenable if the transformation groupoid $\Gamma\ltimes X$ is amenable.
Recall that $\Gamma\ltimes X$ is defined as the topological space $\Gamma\times X$ with product $(g,hx)(h,x)=(gh,x)$.

If $\Gamma$ is amenable, then any action $\Gamma\curvearrowright X$ is amenable. In general amenability of an action is a much weaker property than amenability of the acting group, but there is the following well-known observation, cf.~\cite{MR1799683}*{Remark~5.3.29(2)}.

\begin{lemma}\label{lem:inv-measure-amenable}
If an action $\Gamma\curvearrowright X$ is amenable and there is a $\Gamma$-invariant regular Borel probability measure~$\mu$ on $X$, then the group $\Gamma$ is amenable.
\end{lemma}

\bp
If $(f_i)_i$ is an approximate invariant density on $\Gamma\ltimes X$, then the positive functions~$h_i$ on~$\Gamma$ defined by
$
h_i(\gamma):=\int_X f_i(\gamma,x)d\mu(x)
$
have the property $\|h_i\|_1=1$, since $\mu$ is a probability measure. By $\Gamma$-invariance of $\mu$ we also have
\begin{align*}
\|h_i(\cdot\,\gamma)-h_i\|_1 &=\sum_{g\in \Gamma} \Big|\int_{X}f_i(g \gamma,x)d\mu(x)-\int_{X}f_i(g ,\gamma x)d\mu(x)\Big|\\
&\le\int_{X}\sum_{g\in \Gamma} |f_i((g ,\gamma x)(\gamma, x))-f_i(g ,\gamma x)|d\mu(x) \to 0
\end{align*}
for all $\gamma\in\Gamma$. Therefore $\Gamma$ is amenable.
\ep

A fundamental result that we need in this paper is the following version of the Effros--Hahn conjecture.

\begin{thm}[\cite{IW}]\label{thm:EH}
If $\G$ is an amenable second countable Hausdorff locally compact \'etale groupoid, then the induction map $\Ind\colon\Stab(\G)^\prim\to\Prim C^*(\G)$ is surjective.
\end{thm}

In general there is no obvious useful topology on $\Stab(\G)^\prim$ and the above result says nothing about the topology on $\Prim C^*(\G)$. Recall that given a C$^*$-algebra $A$, the Jacobson, or hull-kernel, topology on the primitive ideal space $\Prim A$ is defined by declaring the closure of a set $\Omega\subset\Prim A$ to be $\{J\in\Prim A: \bigcap_{I\in\Omega}I\subset J\}$. Then the map $J\mapsto \hull(J)^c$, where
$$
\hull(J):=\{I\in\Prim A:J\subset I\},
$$
defines an isomorphism of the ideal lattice in $A$ onto the lattice of open subsets of $\Prim A$, see, e.g., \cite{Ped}*{Theorem 4.1.3}.

\smallskip

Let us now discuss a few more or less known (especially in the second countable case) ways to check amenability. For this discussion, let us remind that by a bisection one means any subset of $\G$ on which the range and source maps~$r$ and~$s$ are injective. By a \emph{$\Gamma$-grading} of a locally compact \'etale groupoid~$\G$, where $\Gamma$ is a discrete group, one means a continuous homomorphism $\Phi\colon\G\to\Gamma$. Given such a grading, we can define a $*$-homomorphism
$$
\delta\colon C^*(\G)\to C^*(\G)\otimes_{\mathrm{max}}C^*(\Gamma)
$$
such that if $f\in C_c(W)$ for an open bisection $W\subset\G$ and $\Phi(W)=\{\gamma\}$ for some $\gamma\in\Gamma$, then
$$
\delta(f)=f\otimes u_\gamma,
$$
where $u_{\gamma}$ is the canonical unitary generator of $C^*(\Gamma)$ associated to $\gamma$. It is known that $\delta$ passes to an injective homomorphism
$$
\delta_r\colon C^*_r(\G)\to C^*_r(\G)\otimes C^*_r(\Gamma),
$$
see, e.g., \cite{CN1}*{Lemma~4.6}.

\begin{prop}[cf.~\citelist{\cite{MR3256184}*{Proposition~9.3}\cite{RW}*{Corollary~4.5}}]\label{prop:graded-amen}
Assume $\G$ is a Hausdorff locally compact \'etale groupoid graded by an amenable discrete group $\Gamma$, with grading $\Phi\colon\G\to\Gamma$. Then $\G$ is amenable if and only if the subgroupoid $\Phi^{-1}(e)\subset\G$ is amenable.
\end{prop}

\bp
Let $\tau$ be the canonical trace on $C^*_r(\Gamma)$. Then $(\iota\otimes\tau)\circ\delta_r\colon C^*_r(\G)\to C^*_r(\G)$ maps $C_{c}(\G)$ onto $C_{c}(\Phi^{-1}(e))$. Since $\Phi^{-1}(e)\subset \mathcal{G}$ is an open subgroupoid, the closure of $C_{c}(\Phi^{-1}(e))$ in $C^*_r(\G)$ is identical to $C^*_r(\Phi^{-1}(e))$, so the map $(\iota\otimes\tau)\circ\delta_r$ is a conditional expectation $C^*_r(\G)\to C^*_r(\Phi^{-1}(e))$. It follows that if $\G$ is amenable and therefore $C^*_r(\G)$ is nuclear, then $C^*_r(\Phi^{-1}(e))$ is nuclear, equivalently, $\Phi^{-1}(e)$ is amenable.

Conversely, if $\Phi^{-1}(e)$ is amenable, so that $C^*(\Phi^{-1}(e))=C^*_r(\Phi^{-1}(e))$ is nuclear, then one can deduce nuclearity of $C^*(\G)$, and hence of $C^*_r(\G)$, using the coaction $\delta\colon C^*(\G)\to C^*(\G)\otimes_{\mathrm{max}}C^*(\Gamma)=C^*(\G)\otimes C^*_r(\Gamma)$ along the same lines as in the proof of~\cite{MR3256184}*{Proposition~9.3}. Alternatively, one can apply the main result of~\cite{MR1932666} to the coaction $\delta_r\colon C^*_r(\G)\to C^*_r(\G)\otimes C^*_r(\Gamma)$ to immediately conclude that the fixed point algebra $C^*_r(\Phi^{-1}(e))=C^*_r(\G)^{\delta_r}:=\{a\in C^*_r(\G):\delta_r(a)=a\otimes1\}$ is nuclear.
\ep

\begin{lemma}\label{lem:amen-sequence}
Assume $\G$ is a Hausdorff locally compact \'etale groupoid that is the union of an increasing net of amenable open subgroupoids. Then $\G$ is amenable.
\end{lemma}

\bp
For any Hausdorff locally compact \'etale groupoid $\G$ and any open subgroupoid $\mathcal{H} \subset \mathcal{G}$, the natural inclusion $C_{c}(\mathcal{H}) \subset C_{c}(\mathcal{G})$ gives rise to an embedding $C_{r}^{*}(\mathcal{H}) \hookrightarrow C_{r}^{*}(\mathcal{G})$. It follows from this that $C_{r}^{*}(\mathcal{G})$ is the inductive limit of an increasing net of the reduced C$^{*}$-algebras of amenable open subgroupoids. Since inductive limits of nuclear C$^*$-algebras are nuclear, this proves the lemma.
\ep

Recall that a subset $X \subset\Gu$ is called locally closed if it is an intersection of an open and a closed set in $\Gu$. A locally closed subset of a Hausdorff  locally compact space is again a Hausdorff  locally compact space. It follows that the reduced groupoid $\G_X=\{g\in \G\mid r(g)\in X\}$ for an invariant locally closed subset $X\subset \Gu$ is again a locally compact Hausdorff \'etale groupoid. It is clear that if $\G$ is amenable, then $\G_X$ is amenable as well, with an approximate invariant density obtained by restricting such a density on~$\G$.

\begin{lemma}[{cf.~\cite{RW}*{Lemma 3.3}}]\label{lem:amen-decompose}
Assume $\G$ is a Hausdorff locally compact \'etale groupoid. Assume the unit space~$\Gu$ is the union of finitely many invariant locally closed subsets $X_1,\dots,X_n$ such that each groupoid $\G_{X_i}$ is amenable. Then $\G$ is amenable.
\end{lemma}

\bp
When $\G$ is second countable, this follows immediately from the fact that for such groupoids amenability is a Borel property~\cite{MR3403785}, see~\cite{RW}*{Lemma 3.3}. In the general case we will argue as follows.

First of all observe that for every invariant locally closed subset $X\subset\Gu$ there are an invariant open set $U$ and an invariant closed set $F$ such that $X=U\cap F$. Indeed, by assumption we have $X=V\cap\bar X$ for an open set $V$. We claim that the open set $U:=r(\G_V)$ and the closed set $F:=\bar X$ have the right properties. It is clear that they are invariant and $X\subset U\cap F$. To see that the equality holds, take any point $x\in U\cap F$ and choose a net of elements $x_i\in X$ converging to~$x$. There is an open bisection $W$ of $\G$ such that $x\in r(W)$ and $s(W)\subset V$. Let $g\in W$ and, for all $i$ large enough, $g_i\in W$ be such that $r(g)=x$ and $r(g_i)=x_i$. Then $g_i\to g$ and hence $s(g_i)\to s(g)$. Since $s(g_i)\in X$ it follows that $s(g)\in V\cap\bar X=X$. Then $x\in X$, proving our claim.

Therefore we can find invariant open sets $U_1,\dots,U_n$ and invariant closed sets $F_1,\dots,F_n$ such that $X_i=U_i\cap F_i$ for all $i$. Consider the finite algebra of sets generated by $U_1,\dots,U_n$ and $F_1,\dots,F_n$. Each atom, that is, a minimal set, $D$ in this algebra has the form $B_{1}\cap \dots \cap B_{2n}$ for some $B_{i}\in \{U_{j}, F_{j}, U_{j}^{c}, F_{j}^{c}: j=1, \dots, n\}$. It follows that $D$ is invariant, locally closed and entirely contained in some $X_i$. Hence $\G_D$ is amenable. Since we can consider our algebra of sets as generated by $U_1,\dots,U_n, F_1^{c},\dots,F_n^{c}$, it follows that it suffices to prove the following version of the lemma. Assume we are given finitely many invariant open sets $U_1,\dots,U_N$ such that $\G_D$ is amenable for every atom $D$ in the algebra generated by these sets. Then $\G$ is amenable. We will prove this by induction on $N$.

Consider the case $N=1$, so we are given one open set $U$ such that both $\G_U$ and $\G_{U^c}$ are amenable. Then $C^*(\G_{U^c})=C^*_r(\G_{U^c})$. Hence we get a short exact sequence
$$
0\to C^*_r(\G_U)\to C^*_r(\G)\to C^*_r(\G_{U^c})\to0,
$$
see, e.g.~\cite{CN2}*{Proposition~1.2}. Since the class of nuclear C$^*$-algebras is closed under extensions, we conclude that $C^*_r(\G)$ is nuclear, so $\G$ is amenable.

For the induction step we consider separately the groupoids $\G_{U_N}$ and $\G_{U_N^c}$ together with the open subsets $U_1\cap U_N,\dots, U_{N-1}\cap U_N\subset U_N$ and $U_1\cap U_N^c,\dots,U_{N-1}\cap U_N^c\subset U_N^c$ to conclude that these groupoids are amenable by the inductive assumption. Then $\G$ is amenable by the case $N=1$.
\ep

\subsection{Quasi-orbit map and dense sets of primitive ideals}\label{ssec:quasi-orbits}
Recall that a subset $Y$ of a topological space $X$ is called \emph{very dense} if $Y\cap F$ is dense in $F$ for every closed subset $F\subset X$. The role of this notion is explained by the following property.

\begin{lemma} \label{lem:verydense}
If $Y\subset X$ is a very dense subset, then the map $F\mapsto F\cap Y$ defines an isomorphism between the lattice of closed subsets of~$X$ and the lattice of closed (in the relative topology) subsets of $Y$, with the inverse map given by $F\cap Y\mapsto \overline{F\cap Y}$.
\end{lemma}

It follows that in order to describe all ideals of a C$^*$-algebra $A$ it suffices to understand some very dense subset of $\Prim A$.

Consider an amenable second countable Hausdorff locally compact \'etale groupoid $\G$. Then we know that every primitive ideal has the form $\Ind(x,J)$ for some $(x,J)\in\Stab(\G)^\prim$. It is not difficult to understand when a set of such ideals where $x$ runs through a subset of $\Gu$ is dense. Before we formulate the result, let us recall a few notions.

For a topological space $X$, denote by $X^\sim$ its \emph{$T_0$-ization}, also known as the Kolmogorov quotient, obtained by identifying points with identical closures. Any locally compact \'etale groupoid $\G$ acts on $\Gu$ by $ g.s(g)=r(g)$. Denote by $[x]$ the $\G$-orbit $r(\G_x)$ of $x$, and by $\G\backslash\Gu$ the orbit space. The $T_0$-ization of the orbit space coincides with the space of \emph{quasi-orbits} of the action $\G\curvearrowright\Gu$, that is, with the quotient of $\Gu$ by the equivalence relation
$$
x\sim y\quad\text{iff}\quad\overline{[x]}=\overline{[y]},
$$
see~\cite{CN3}*{Corollary~1.6}. Denote by $\QQ$ the quotient map $\Gu\to(\G\backslash\Gu)^\sim$, so $\QQ(x)$ is the quasi-orbit of $x\in\Gu$. By ~\cite{CN3}*{Corollary~1.5}, the map $\QQ$ is open and it establishes a one-to-one correspondence between the open $\G$-invariant subsets of $\Gu$ and the open subsets of $(\G\backslash\Gu)^\sim$, hence also a one-to-one correspondence between the closed $\G$-invariant subsets of~$\Gu$ and the closed subsets of $(\G\backslash\Gu)^\sim$. Using this, the following observation becomes straightforward.

\begin{lemma} \label{lem:recap}
For any subset $\Omega \subset \Gu$ and any closed $\G$-invariant subset $F \subset \Gu$, the set $r(\G_{\Omega})\cap F$ is dense in $F$ if and only if $\QQ(\Omega)\cap \QQ(F)$ is dense in $\QQ(F)$.
\end{lemma}

\bp
Since $\QQ^{-1}(\QQ(F))=F$, we have $\QQ(\Omega)\cap\QQ(F)=\QQ(\Omega\cap F)=\QQ(r(\G_\Omega)\cap F)$. Since the set $r(\G_\Omega)\cap F$ and its closure are invariant, we also have $\overline{\QQ(r(\G_{\Omega})\cap F)}=\QQ\big(\,\overline{r(\G_{\Omega})\cap F}\,\big)$. Therefore $\overline{\QQ(\Omega)\cap\QQ(F)}=\QQ\big(\,\overline{r(\G_{\Omega})\cap F}\,\big)$. Hence $\overline{\QQ(\Omega)\cap\QQ(F)}=\QQ(F)$ if and only if $\overline{r(\G_{\Omega})\cap F}=F$.
\ep

Taking $F=\Gu$, we conclude that a subset $\Omega\subset\Gu$ has a dense image in $(\G\backslash\Gu)^\sim$ if and only if $r(\G_\Omega)$ is dense in $\Gu$. Taking arbitrary $\G$-invariant closed sets, we conclude that a subset $\Omega\subset\Gu$ has a very dense image in $(\G\backslash\Gu)^\sim$ if and only if $r(\G_\Omega)\cap F$ is dense in $F$ for every $\G$-invariant closed subset $F\subset\Gu$.

Returning to the second countable amenable case, there is a well-defined continuous map $p\colon\Prim C^*(\G)\to (\G\backslash\Gu)^\sim$ such that $p(\Ind(x,J))=\QQ(x)$, see~\cite{AHW}*{Definition~3.3}. Moreover, by~\cite{AHW}*{Theorem~3.5} this map is open. The following proposition gives a closely related result.

\begin{prop}\label{prop:dense}
Assume $\G$ is an amenable second countable Hausdorff locally compact \'etale groupoid and $\Omega\subset\Gu$ is a subset. Then the set
$$
\{\Ind(x,J):x\in\Omega,\ J\in\Prim C^*(\Gxx)\}
$$
is dense in $\Prim C^*(\G)$ if and only if the set $r(\G_\Omega)$ is dense in $\Gu$.
\end{prop}

\bp
Since $p\colon\Prim C^*(\G)\to (\G\backslash\Gu)^\sim$ is surjective and continuous, density of $\{\Ind(x,J):x\in\Omega,\ J\in\Prim C^*(\Gxx)\}$ in $\Prim C^*(\G)$ implies density of $\QQ(\Omega)$ in $(\G\backslash\Gu)^\sim$, which is in turn equivalent to density of $r(\G_\Omega)$ in $\Gu$.

Conversely, assume $r(\G_\Omega)$ is dense. Since the induction map is constant on the $\G$-orbits, without loss of generality we may assume that $\Omega$ is $\G$-invariant, so we assume that $\Omega$ itself is dense. Take a primitive ideal $\Ind(y,I)$. It is known that density of $\Omega$ implies that the representation $\rho_y$ is weakly contained in $\bigoplus_{x\in\Omega}\rho_x$, see \cite{NS}*{Proposition~1.4} for a short proof. Since the regular representation of an amenable group is weakly equivalent to the direct sum of all irreducible representations, it follows first that $\pi_{I} \prec \lambda_{\G_{y}^{y}}$ and then that
$$
\Ind^\G_{\G^y_y}\pi_I\prec \Ind^\G_{\G^y_y}\lambda_{\G_{y}^{y}} \prec \rho_y\prec\bigoplus_{x\in\Omega}\rho_x\sim\bigoplus_{x\in\Omega}\bigoplus_{J\in\Prim C^*(\Gxx)}\Ind^\G_{\Gxx}\pi_J.
$$
This means exactly that $\Ind(y,I)$ lies in the closure of $\{\Ind(x,J): x\in\Omega,\ J\in\Prim C^*(\Gxx)\}$.
\ep

As we discussed above, a much more interesting question is when a subset is very dense.

\begin{question} \label{ques}
Given a subset $\Omega\subset\Gu$ with a very dense image in $(\G\backslash\Gu)^\sim$, is the set
$$
\{\Ind(x,J):x\in\Omega,\ J\in\Prim C^*(\Gxx)\}
$$
very dense in $\Prim C^*(\G)$? Does the above set coincide with $\Prim C^*(\G)$ when $\Omega$ contains a representative of every quasi-orbit, that is, when $\QQ(\Omega)=(\G\backslash\Gu)^\sim$?
\end{question}

Note that by continuity of $p\colon\Prim C^*(\G)\to (\G\backslash\Gu)^\sim$ the image of any very dense subset of $\Prim C^*(\G)$ in $(\G\backslash\Gu)^\sim$ is very dense. Hence very density of the image of $\Omega\subset\Gu$ in $(\G\backslash\Gu)^\sim$ is necessary for very density of $\{\Ind(x,J):x\in\Omega,\ J\in\Prim C^*(\Gxx)\}$ in $\Prim C^*(\G)$.

\smallskip

As we will see, for the main examples of groupoids analyzed in this paper both parts of Question~\ref{ques} have positive answers.

\subsection{Dynamical ideals}

As we saw above, in the amenable second countable case the space $\Prim C^*(\G)$ is fibered over $(\G\backslash\Gu)^\sim$. 
The open subsets of $(\G\backslash\Gu)^\sim$, or equivalently, the open $\G$-invariant subsets of~$\Gu$, correspond to certain ideals of $C^*(\G)$ that make sense without any assumptions on $\G$.

Namely, let $\G$ be any Hausdorff locally compact \'etale groupoid. Given a nonempty $\G$-invariant open subset $U\subset\Gu$, it is well known that the C$^*$-algebra $C^*_r(\G_U)$ can be viewed as a C$^*$-subalgebra of $C^*_r(\G)$ and in fact becomes an ideal in it; we also let $C^*_r(\G_U):=0$ for $U=\emptyset$. In~\cite{BCS} it is proposed to call such ideals \emph{dynamical}, as they are determined by the orbit structure of the dynamical system $\G\curvearrowright\Gu$. We thus have an isomorphism $U\mapsto C^*_r(\G_U)$ of the lattice of invariant open subsets of $\Gu$ onto the lattice of dynamical ideals in $C^*_r(\G)$. The goal of this subsection is to give a couple of (more or less known) alternative characterizations of such ideals.

\smallskip

Given an invariant open set $U\subset\Gu$, the restriction map $C_c(\G)\to C_c(\G_{U^c})$, $f\mapsto f|_{\G_{U^c}}$, extends to a surjective $*$-homomorphism $C^*_r(\G)\to C^*_r(\G_{U^c})$ whose kernel contains $C^*_r(\G_U)$. A groupoid $\G$ is called \emph{inner exact} if the sequence
\begin{equation}\label{eq:inner}
0\to C^*_r(\G_U)\to C^*_r(\G)\to C^*_r(\G_{U^c})\to0
\end{equation}
is exact for every invariant open set $U\subset\Gu$. If $\G$ is amenable, then it is inner exact.

Recall that $E\colon C^*_r(\G)\to C_0(\Gu)$ denotes the canonical conditional expectation. An ideal $I\subset C^*_r(\G)$ is called \emph{diagonal-invariant} if $E(I)\subset I$. It is clear that every ideal $C^*_r(\G_U)$ is diagonal-invariant.

\begin{prop}[cf.~\cite{BL}*{Lemma~3.6}]\label{prop:diagonal}
Assume $\G$ is a Hausdorff locally compact \'etale groupoid. Then the following conditions are equivalent:
\begin{enumerate}
  \item the groupoid $\G$ is inner exact;
  \item every diagonal-invariant ideal of $C^*_r(\G)$ is dynamical.
\end{enumerate}
\end{prop}

\bp
Assume (1) holds. Let $I\subset C^*_r(\G)$ be a diagonal-invariant ideal and $U\subset \Gu$ be the invariant open set such that $I\cap C_0(\Gu)=C_0(U)$. Then $C^*_r(\G_U)\subset I$ and we want to prove that in fact the equality holds. Since the map $C^*_r(\G)\to C^*_r(\G_{U^c})$ respects the canonical conditional expectations, by using inner exactness and passing to the quotient $C^*_r(\G)/C^*_r(\G_U)\cong C^*_r(\G_{U^c})$, we may assume that $I\cap C_0(\Gu)=0$, and then we need to prove that $I=0$. The condition $E(I)\subset I$ implies that $E=0$ on $I$. As the conditional expectation $E$ is faithful, it follows that $I=0$. Therefore (1)$\Rightarrow$(2).

For the opposite implication, assume that $\G$ is not inner exact, so that there exists an invariant open set $U\subset\Gu$ such that~\eqref{eq:inner} is not exact. Consider the kernel $I$ of the map $C^*_r(\G)\to C^*_r(\G_{U^c})$. Then $I\cap C_0(\Gu)=C_0(U)$ and $I\ne C^*_r(\G_U)$, so the ideal $I$ is not dynamical. This ideal is diagonal-invariant, since the map $C^*_r(\G)\to C^*_r(\G_{U^c})$ respects the canonical conditional expectations.
\ep

Next, recall that a groupoid $\G$ is called \emph{effective} if the interior $\Iso$ of its isotropy bundle coincides with $\Gu$. It is called \emph{strongly effective} if the reduced groupoid $\G_X$ is effective for every nonempty invariant  closed  subset $X\subset\Gu$.

Given a $\Gamma$-grading of $\G$, consider the corresponding coaction $\delta_r\colon C^*_r(\G)\to C^*_r(\G)\otimes C^*_r(\Gamma)$.
An ideal $I\subset C^*_r(\G)$ is called \emph{gauge-invariant} if $\delta_r(I)\subset I\otimes C^*_r(\Gamma)$. Clearly, every dynamical ideal is gauge-invariant.

\begin{prop}\label{prop:gauge}
Assume $\G$ is a Hausdorff locally compact \'etale groupoid graded by a discrete group $\Gamma$, with grading $\Phi\colon\G\to\Gamma$. Consider the following conditions:
\begin{enumerate}
  \item the groupoid $\G$ is inner exact and, for every nonempty $\G$-invariant closed  subset $X\subset\Gu$, the reduced groupoid $\Phi^{-1}(e)_X$ is effective;
  \item every gauge-invariant ideal of $C^*_r(\G)$ is dynamical.
\end{enumerate}
Then $(1)\Rightarrow(2)$. If the units $x$ with amenable isotropy groups~$\Gxx$ are dense in every $\G$-invariant closed subset $X\subset\Gu$ (for example, if $\G$ is amenable) and $\Gamma$ is exact, then $(2)\Rightarrow(1)$.
\end{prop}

When $\Phi$ is trivial, the proposition implies that if $\G$ is inner exact and strongly effective, then the map $U\mapsto C^*_r(\G_U)$ defines an isomorphism between the lattice of invariant open subsets of~$\Gu$ and the ideal lattice in $C^*_r(\G)$. This result essentially goes back to Renault~\cite{Rbook}*{Proposition II.4.6} (but note that the formulation in~\cite{Rbook} misses the assumption of inner exactness), see the discussion in~\cite{CN4}*{Section~4.5}. In the general case the implication $(1)\Rightarrow(2)$ is deduced from this by analyzing spectral subspaces for the coaction of $\Gamma$. In the context of Steinberg algebras of ample groupoids this has been done by Clark, Exel and Pardo~\cite{MR3794898}*{Theorem~5.3} under the formally stronger assumption of strong effectiveness of the groupoid $\Phi^{-1}(e)$. In order to deal with our C$^*$-algebraic case we need to recall some facts about group coactions.

Assume $I\subset C^*_r(\G)$ is a gauge-invariant ideal. Denote the generators of $C^*_r(\Gamma)$ by $\lambda_\gamma$, and recall that $\tau$ denotes the canonical trace on $C^*_r(\Gamma)$. For $\gamma\in\Gamma$, the spectral subspace $I_\gamma\subset I$ is defined by
$$
I_\gamma:=\{a\in I:\delta_r(a)=a\otimes\lambda_\gamma\}.
$$

\begin{lemma}\label{lem:spandense}
Assume $\G$ is a Hausdorff locally compact \'etale groupoid graded by a discrete group $\Gamma$, and $I\subset C^*_r(\G)$ is a gauge-invariant ideal. Then the spaces $I_\gamma$ ($\gamma\in\Gamma$) span a dense subspace of $I$.
\end{lemma}
\bp
By \cite{MR1010978}*{Corollaire 7.15}, injectivity of~$\delta_r$ implies that the \emph{Podle\'{s} condition} is satisfied: the space $(1\otimes C^*_r(\Gamma))\delta_r(I)$, spanned by the elements $(1\otimes a)\delta_r(b)$ ($a\in C^*_r(\Gamma)$, $b\in I$), is dense in $I\otimes C^*_r(\Gamma)$. This implies the lemma by the following formula for the spectral subspaces:
\begin{equation}\label{eq:spectral}
I_\gamma=(\iota\otimes\tau)\big((1\otimes\lambda_\gamma^*)\delta_r(I)\big).
\end{equation}
This is a standard formula in the theory of group coactions, we include a short proof for the reader's convenience.

Since $(\iota\otimes\tau)\big((1\otimes\lambda_\gamma^*)\delta_r(\cdot)\big)$ acts as the identity operator on $I_\gamma$, the inclusion $\subset$ is clear. For the opposite inclusion, take $a\in I$ and let $b:=(\iota\otimes\tau)\big((1\otimes\lambda_\gamma^*)\delta_r(a)$. Since $(\delta_r\otimes\iota)\circ\delta_r=(\iota\otimes\Delta)\circ\delta_r$, where $\Delta\colon C^*_r(\Gamma)\to C^*_r(\Gamma)\otimes C^*_r(\Gamma)$ is the comultiplication defined by $\Delta(\lambda_\gamma)=\lambda_\gamma\otimes\lambda_\gamma$, we get that
\begin{equation}\label{eeeq}
\delta_{r}(b)=(\iota\otimes\iota\otimes\tau)\big((1\otimes 1\otimes\lambda_\gamma^*)( \iota \otimes \Delta)(\delta_{r}(a))\big) .
\end{equation}
Since $(\iota\otimes\tau)\big((1\otimes\lambda_\gamma^*)\Delta(c)\big)=\tau(\lambda_\gamma^*c)\lambda_\gamma$ for $c\in C^*_r(\Gamma)$, which can be checked on the generators of $C^*_r(\Gamma)$, we have the following identity for all $c\in C^*_r(\G)\otimes C^*_r(\Gamma)\otimes C^*_r(\Gamma)$:
$$
(\iota\otimes\iota\otimes\tau)\big((1\otimes 1\otimes\lambda_\gamma^*)( \iota \otimes \Delta)(c)\big)=(\iota\otimes \tau)((1\otimes \lambda_{\gamma}^{*} )c) \otimes \lambda_{\gamma}.
$$
Combining this with \eqref{eeeq}, we see that $\delta_{r}(b)=b\otimes \lambda_{\gamma}$, which finishes the proof of~\eqref{eq:spectral}.
\ep

\bp[Proof of Proposition~\ref{prop:gauge}]
Assume (1) holds. Consider a gauge-invariant ideal $I\subset C^*_r(\G)$. Let $U\subset\Gu$ be the $\G$-invariant open set such that $I\cap C_0(\Gu)=C_0(U)$. We want to show that $I=C^*_r(\G_U)$. Letting $\pi$ denote the canonical homomorphism $C_{r}^{*}(\G)\to C_{r}^{*}(\G_{U^{c}})$ and $\delta_{r}'$ be the coaction defined by the grading of $\G_{U^{c}}$, we see that $(\pi\otimes \iota) \circ \delta_{r} = \delta_{r}' \circ \pi$. Hence $\pi(I)$ is a gauge-invariant ideal in $C^*_r(\G_{U^c})$. As in the proof of Proposition~\ref{prop:diagonal}, by passing to the quotient $C^*_r(\G_{U^c})\cong C^*_r(\G)/C^*_r(\G_U)$, it therefore suffices to consider the case $U=\emptyset$, that is, $I\cap C_0(\Gu)=0$.

The spectral subspace $I_e$ is an ideal in $C^*_r(\G)_e=C^*_r(\Phi^{-1}(e))$. By assumption, the groupoid $\Phi^{-1}(e)$ is effective. Since $I_e\cap C_0(\Gu)=0$, this is known to imply that $I_e=0$, see, e.g., \cite{CN4}*{Proposition~4.23}. Any approximate unit $(e_i)_i$ in $I_e$ is an approximate unit in $I$ by Lemma \ref{lem:spandense}, since if $a\in I_\gamma$ for some~$\gamma$, then $a^*a\in I_e$ and
$$
\|a-ae_i\|^{2}=\|(1-e_i)a^*a(1-e_i)\|\le\|a^*a-a^*ae_i\|.
$$
It follows that the ideal $I$ is generated by $I_e$, so we conclude that also $I=0$. Thus, (1)$\Rightarrow$(2).

\smallskip

Assume now that the units with amenable isotropy groups are dense in every $\G$-invariant closed subset $X\subset\Gu$, the group $\Gamma$ is exact, but (1) does not hold. We need to show that then there is a nondynamical gauge-invariant ideal. Assume first that $\G$ is not inner exact, so that there exists an invariant open set $U\subset\Gu$ such that~\eqref{eq:inner} is not exact. As in the proof of Proposition~\ref{prop:diagonal}, consider the kernel $I$ of the quotient map $q\colon C^*_r(\G)\to C^*_r(\G_{U^c})$. Then $I$ is nondynamical and gauge-invariant, since $q$ is equivariant with respect to the coaction of $\Gamma$ and $I\otimes C^*_r(\Gamma)$ is the kernel of the map $q\otimes\iota\colon C^*_r(\G)\otimes C^*_r(\Gamma)\to C^*_r(\G_{U^c})\otimes C^*_r(\Gamma)$ by exactness of $\Gamma$.

Assume next that $\G$ is inner exact and there is a nonempty $\G$-invariant closed subset $X\subset\Gu$ such that the reduced groupoid $\Phi^{-1}(e)_X$ is not effective. By passing to a quotient of~$C^*_r(\G)$ we may assume that $X=\Gu$. We can then argue similarly to the case of trivial grading, cf.~\citelist{\cite{MR1088230}*{Theorem~4.1}\cite{MR1258035}*{Theorem~2}\cite{MR3189105}*{Theorem~5.1}}. Namely, by assumption there is an open bisection~$W$ of~$\G$ contained in $\operatorname{Iso}(\Phi^{-1}(e))^\circ\setminus\Gu$. Take any nonzero function $f\in C_c(W)$. Define a function $f_0\in C_c(r(W))$ by $f_0(r(g)):=f(g)$ for $g\in W$. Then the ideal $I\subset C^*_r(\G)$ generated  by $f-f_0$ is nonzero and $I\cap C_0(\Gu)=0$, see the proof of the implication $(2)\Rightarrow(1)$ in~\cite{CN4}*{Proposition~4.23}. Therefore $I$ is nondynamical. As $\delta_r(f-f_0)=(f-f_0)\otimes1$, this ideal is obviously gauge-invariant.
\ep

\begin{remark}
When $\Gamma$ is abelian, then $C^*_r(\Gamma)\cong C(\widehat\Gamma)$ and every character $\chi\in\widehat\Gamma$ defines a gauge automorphism $\alpha_\chi$ of $C^*_r(\G)$ such that for all $f\in C_c(\G)$ and $g\in\G$ we have
$$
\alpha_\chi(f)(g)=\chi(\Phi(g))f(g).
$$
Under the identification $C_r^{*}(\G)\otimes C^*_r(\Gamma)\simeq C(\widehat{\Gamma}; C_r^{*}(\G))$, the automorphism $\alpha_\chi$ is obtained by applying $\delta_{r}\colon C^*_r(\G)\to C(\widehat{\Gamma}; C_r^{*}(\G))$ and then evaluating at $\chi$. It follows that the condition $\delta_r(I)\subset I\otimes C^*_r(\Gamma)$ is equivalent to $\alpha_\chi(I)=I$ for all $\chi\in\widehat\Gamma$.\ee
\end{remark}

It is sometimes convenient to think about the condition on $\Phi$ in Proposition~\ref{prop:gauge} in a different way.

\begin{lemma}\label{lem:ess-injectivity}
Assume $\G$ is a Hausdorff locally compact \'etale groupoid graded by a discrete group~$\Gamma$, with grading $\Phi\colon\G\to\Gamma$. Consider the following conditions:
\begin{enumerate}
  \item[(1)] the groupoid $\Phi^{-1}(e)$ is effective;
  \item[(2)] the set of units $x$ such that $\Phi$ is injective on $\Gxx$ is dense in $\Gu$;
  \item[(3)] the set of units $x$ such that $\Phi$ is injective on $\Gxx$ is residual in $\Gu$.
\end{enumerate}
Then $(3)\Rightarrow(2)\Rightarrow(1)$. If $\G$ can be covered by countably many open bisections, then all three conditions are equivalent.
\end{lemma}

\bp
Obviously (3) implies (2). Condition (2) is equivalent to saying that the set of units $x$ such that $\Phi^{-1}(e)^x_x=\{x\}$ is dense in $\Gu$. This is clearly stronger than the effectiveness condition $\operatorname{Iso}(\Phi^{-1}(e))^\circ=\Gu$. If~$\G$, and hence $\Phi^{-1}(e)$, can be covered by countably many bisections, then the three conditions are equivalent, since the set of units $x$ such that $\operatorname{Iso}(\Phi^{-1}(e))^\circ_x=\Phi^{-1}(e)^x_x$ is residual in $\Gu$, see, e.g., \cite{CN2}*{Lemma~5.2}.
\ep

\bigskip

\section{The ideal structure of graded groupoids with essentially central isotropy}\label{sec:ess-isotropy}

\subsection{Groupoids with essentially central isotropy}

The class of groupoids we are interested in can be introduced in several ways. As a main definition we take the following.

\begin{defn}
Assume $\G$ is a Hausdorff locally compact \'etale groupoid graded by a discrete group $\Gamma$, with grading $\Phi\colon\G\to\Gamma$. We say that the graded groupoid $\G$ has \emph{essentially central isotropy} if, for every $x\in\Gu$, the map $\Phi$ is injective on $\IsoGx{x}^\circ_x$ and $\Phi(\IsoGx{x}^\circ_x)\subset Z(\Gamma)$.
\end{defn}

Here $\IsoGx{x}^\circ$ denotes the interior, relative to the reduced groupoid~$\G_{\overline{[x]}}$, of the isotropy bundle of~$\G_{\overline{[x]}}$. Its fiber $\IsoGx{x}^\circ_x$ at $x$ is called the \emph{essential isotropy group} of~$x$. We recall that for any \'etale groupoid $\HH$ the groupoid $\operatorname{Iso}(\HH)^\circ$ is normal in $\HH$ in the sense that $h \operatorname{Iso}(\HH)^\circ_{s(h)}h^{-1}=\operatorname{Iso}(\HH)^\circ_{r(h)}$ for all $h\in\HH$. The closure $\overline{\operatorname{Iso}(\HH)^\circ}$ is a normal subgroupoid of $\HH$ as well.

The following lemma implies among other things that it suffices to check the conditions on~$\Phi$ on representatives of quasi-orbits of the action $\G\curvearrowright\Gu$.

\begin{lemma}\label{lem:iso}
Assume $\G$ is a Hausdorff locally compact \'etale groupoid graded by a discrete group~$\Gamma$, with grading $\Phi\colon\G\to\Gamma$.
Assume $x\in\Gu$ is a unit such that $\overline{[x]}=\Gu$, $\Phi$ is injective on $\Iso_x$ and $\Phi(\Iso_x)\subset Z(\Gamma)$. Then, for any $y\in\Gu$, we have:
\begin{enumerate}
  \item $\Phi$ is injective on $\overline{\IsoG^\circ}\cap\G^y_y$;
  \item $\Phi(\overline{\IsoG^\circ}\cap\G^y_y)\subset\Phi(\Iso_x)$;
  \item $\overline{\IsoG^\circ}\cap\G^y_y\subset Z(\G^y_y)$.
\end{enumerate}
In particular, if $\overline{[y]}=\Gu$, then
$$
\overline{\IsoG^\circ}\cap\G^y_y=\Iso_y\quad\text{and}\quad \Phi(\Iso_y)=\Phi(\Iso_x).
$$
\end{lemma}

\bp
Take $g\in \overline{\IsoG^\circ}\cap\G^y_y$. Since the orbit of $x$ is dense in $\Gu$ and hence $\bigcup_{z\in[x]}\Iso_z$ is dense in~$\IsoG^\circ$, we can find nets of elements $h_i\in\G_x$ and $g_i\in\Iso_x$ such that $h_ig_ih_i^{-1}\to g$. Then $\Phi(g_i)\to\Phi(g)$, hence $\Phi(g_i)=\Phi(g)$ for $i$ large enough. This already proves (2). If $\Phi(g)=e$, then we also get $\Phi(g_i)=e$ for all $i$ large enough, hence $g_i=x$ by the injectivity of $\Phi$ on $\Iso_x$. It follows that $h_ig_ih_i^{-1}=r(h_i)$, hence $g=y$, proving (1).

To prove (3), take $g\in \overline{\IsoG^\circ}\cap\G^y_y$ and $h\in\G^y_y$. Then by (2) and the centrality of $\Phi(\Iso_x)$ we have $\Phi(hgh^{-1})=\Phi(g)$. Since $\overline{\IsoG^\circ}\cap\G^y_y$ is a normal subgroup of~$\G^y_y$ and $\Phi$ is injective on $\overline{\IsoG^\circ}\cap\G^y_y$ by (1), it follows that $hgh^{-1}=g$, proving (3).

Finally, if $\overline{[y]}=\Gu$, then by (1) and (2) the unit $y$ satisfies the same assumptions as $x$, hence by symmetry we have $\Phi(\Iso_x)\subset\Phi(\Iso_y)$. This implies the last two identities in the statement of the lemma.
\ep

It can be a nontrivial task to understand the interior of an isotropy bundle, so we will next characterize essential centrality only in terms of isotropy groups. For this we need the following lemma, which will also be useful later.

\begin{lemma}\label{lem:nice-points}
Assume $\G$ is a second countable Hausdorff locally compact \'etale groupoid. Then, for every $x\in \Gu$, the set of points $y\in\overline{[x]}$ such that $\overline{[y]}=\overline{[x]}$ and $\IsoGx{x}^\circ_y=\G^y_y$ is residual, in particular dense, in $\overline{[x]}$.
\end{lemma}

\bp
This follows from \cite{CN2}*{Lemmas~5.2 and~6.2}.
\ep

\begin{prop}\label{prop:ess-central-alt}
Assume $\G$ is a second countable Hausdorff locally compact \'etale groupoid  graded by a discrete group~$\Gamma$, with grading $\Phi\colon\G\to\Gamma$. Consider the sets
$$
\Omega_1:=\{x\in\Gu: \Phi\ \text{is injective on}\ \Gxx\}\quad\text{and}\quad\Omega_2:=\{x\in\Gu: \Phi(\Gxx)\subset Z(\Gamma)\}.
$$
Then $\G$ has essentially central isotropy if and only if, for every $\G$-invariant closed  subset $X\subset\Gu$, the sets $\Omega_i\cap X$ ($i=1,2$) are dense in $X$.
\end{prop}

\bp
The ``only if'' direction is immediate from the definitions and the previous lemma. For the opposite direction we need to consider the reductions of $\G$ by orbit closures, so without loss of generality we may assume in addition that $\Gu=\overline{[x_0]}$ for some $x_0$ and need to prove that~$\Phi$ is injective on $\Iso_x$ and~$\Phi(\Iso_x)\subset Z(\Gamma)$ for all $x\in\Gu$ with $\overline{[x]}=\Gu$.

By assumption, the set $\Omega_1$ is dense in $\Gu$. By Lemma~\ref{lem:ess-injectivity} if follows that it is in fact residual. By Lemma~\ref{lem:nice-points} we can then find $x\in\Omega_1$ such that $\overline{[x]}=\Gu$ and $\Iso_x=\Gxx$. Then, by Lemma~\ref{lem:iso}, in order to finish the proof it suffices to show that $\Phi(\Iso_x)\subset Z(\Gamma)$. Choose a sequence of elements $x_n\in\Omega_2$ converging to $x$. For every $g\in\Iso_x$ we can then find elements $g_n\in\G^{x_n}_{x_n}$ such that $g_n\to g$. Then $\Phi(g_n)=\Phi(g)$ for all $n$ sufficiently large, and since $\Phi(g_n)\in Z(\Gamma)$, we conclude that $\Phi(g)\in Z(\Gamma)$.
\ep

\begin{remark}
If $\G$ is second countable and has essentially central isotropy, then the sets $\Omega_1\cap X$ and $\Omega_2\cap X$, and hence also $\Omega_1\cap\Omega_2\cap X$, are in fact residual in $X$ for every $\G$-invariant closed subset $X\subset\Gu$. In order to see this it suffices to consider $X=\Gu$. Then $\Omega_1$ is residual by Lemma~\ref{lem:ess-injectivity}. As for $\Omega_2$, the same reasoning as above shows that density of $\Omega_2$ implies that $\Phi(\Iso_x)\subset Z(\Gamma)$ for all $x$. Since the set $\Omega_0:=\{x: \Iso_x=\Gxx\}$ is residual and $\Omega_0\subset\Omega_2$, we conclude that $\Omega_2$ is residual.\ee
\end{remark}

\begin{cor}
Assume $\G$ is a second countable Hausdorff locally compact \'etale groupoid  graded by a discrete group~$\Gamma$, with grading $\Phi\colon\G\to\Gamma$. Assume also that $\G$ has essentially central isotropy and is inner exact. Then an ideal $I\subset C^*_r(\G)$ is dynamical if and only if it is gauge-invariant.
\end{cor}

\bp
By Proposition~\ref{prop:ess-central-alt} and Lemma~\ref{lem:ess-injectivity}, the groupoid $\Phi^{-1}(e)_X$ is effective for every nonempty $\G$-invariant closed subset $X\subset\Gu$. Hence the corollary follows from Proposition~\ref{prop:gauge}.
\ep

In our applications we will mostly be interested in the case where $\Gamma$ is abelian. In this case, using again Proposition~\ref{prop:ess-central-alt} and Lemma~\ref{lem:ess-injectivity}, we get the following result.

\begin{cor}\label{cor:effective-vs-central}
Assume $\G$ is a second countable Hausdorff locally compact \'etale groupoid  graded by a discrete abelian group~$\Gamma$, with grading $\Phi\colon\G\to\Gamma$. Then $\G$ has essentially central isotropy if and only if the groupoid $\Phi^{-1}(e)_X$ is effective for every nonempty $\G$-invariant closed subset $X\subset\Gu$.
\end{cor}

\subsection{A parametrization of the primitive ideals}\label{ssec:parameterization}
Assume $\G$ is a graded groupoid with essentially central isotropy. For $x\in\Gu$ and $J\in\Prim C^*(\Gxx)$, choose an irreducible representation~$\pi_J$ with kernel $J$.
Since $\IsoGx{x}^\circ_x\subset Z(\Gxx)$ by Lemma~\ref{lem:iso}(3), the restriction of~$\pi_J$ to $\IsoGx{x}^\circ_x$ defines a character $\chi_J$. Since a character of an abelian C$^{*}$-algebra is uniquely determined by its kernel and $\ker\chi_J=J\cap C^{*}(\IsoGx{x}^\circ_x)$, it follows that  $\chi_J$ depends only on~$J$, not on the choice of~$\pi_J$.

Given a character $\chi$ of $Z(\Gamma)$, consider the induced representation
$$
\omega^\chi_x:=\Ind^\G_{\IsoGx{x}^\circ_x}(\chi\circ\Phi|_{\IsoGx{x}^\circ_x}).
$$
If $\IsoGx{x}^\circ_x\ne\Gxx$, then the representation $\omega^\chi_x$ is not irreducible, since by centrality of $\IsoGx{x}^\circ_x$ in $\Gxx$ the action of $\Gxx$ on the right on $\G_x$ defines a nontrivial representation of $\Gxx$ into the commutant of $\omega^\chi_x(C^*(\G))$. But as the following result shows, at least for amenable $\G$ the kernel of $\omega^\chi_x$ is a primitive ideal.

\begin{thm}\label{thm:essentially-central}
Let $\G$ be an amenable second countable Hausdorff locally compact \'etale groupoid graded by a discrete group $\Gamma$, with grading $\Phi\colon\G\to\Gamma$, and assume $\G$ has essentially central isotropy. Then
\begin{enumerate}
  \item the map $\Gu\times\widehat{Z(\Gamma)}\to \Prim C^*(\G)$, $(x,\chi)\mapsto\ker\omega^\chi_x$, is well-defined and surjective; specifically, if $J\in \Prim C^*(\Gxx)$ and $\chi_J=\chi\circ\Phi|_{\IsoGx{x}^\circ_x}$, then $\ker\omega^\chi_x=\Ind(x,J)$;
  \item we have $\ker\omega^{\chi_1}_{x_1}=\ker\omega^{\chi_2}_{x_2}$ if and only if  $\overline{[x_1]}=\overline{[x_2]}$ and $\chi_1=\chi_2$ on $\Phi(\IsoGx{x_1}^\circ_{x_1})=\Phi(\IsoGx{x_2}^\circ_{x_2})$.
\end{enumerate}
\end{thm}

Note that by Lemma~\ref{lem:iso} the equality $\Phi(\IsoGx{x_1}^\circ_{x_1})=\Phi(\IsoGx{x_2}^\circ_{x_2})$ here is not an extra requirement, but a consequence of the assumption $\overline{[x_1]}=\overline{[x_2]}$.

\smallskip

Theorem~\ref{thm:essentially-central} is an extension of results in \cite{CN2}*{Section~6}, which were obtained under slightly stronger assumptions. We will sketch a proof following~\cite{CN2} and providing details where changes (fortunately, minor ones) are needed.

\smallskip

We claim that, similarly to~\cite{CN2}, it suffices to establish the following.

\begin{prop}\label{prop:essentially-central}
Assume $\G$ is an amenable second countable Hausdorff locally compact \'etale groupoid graded by a discrete group $\Gamma$, with grading $\Phi\colon\G\to\Gamma$.
Assume there is a unit $x\in\Gu$ with dense orbit such that $\Phi$ is injective on $\Iso_x$ and $\Phi(\Iso_x)\subset Z(\Gamma)$, hence (by Lemma~\ref{lem:iso}) the same properties hold for any unit with dense orbit. Denote by $\mathcal I\subset\Prim C^*(\G)$ the set of primitive ideals $I$ such that $C_0(\Gu)\cap I=0$. Then
\begin{enumerate}
  \item the map $\Prim C^*(\Gxx)\to\mathcal I$, $J\mapsto\Ind(x,J)$, is surjective;
  \item the map $\widehat{Z(\Gamma)}\to \mathcal I$, $\chi\mapsto\ker\omega^\chi_x$, is well-defined and surjective; specifically, if $J\in \Prim C^*(\Gxx)$ and $\chi_J=\chi\circ\Phi|_{\Iso_x}$, then $\ker\omega^\chi_x=\Ind(x,J)$;
  \item if $y\in\Gu$ is a unit with dense orbit, then $\ker\omega^{\chi}_{x}=\ker\omega^{\kappa}_y$ if and only if $\chi=\kappa$ on $\Phi(\Iso_x)=\Phi(\Iso_y)$.
\end{enumerate}
\end{prop}

Indeed, assuming that this proposition is true and working in the setting of Theorem~\ref{thm:essentially-central}, denote by
$\pi_x\colon C^{*}(\G)\to C^{*}(\G_{\overline{[x]}})$ the canonical surjection for $x\in\Gu$. Then, by the definition of induction, for any subgroup $S\subset\Gxx$ and any unitary representation $\rho$ of $S$ we have
\begin{equation}\label{eqpiker}
\big(\Ind_{S}^{\G_{\overline{[x]}}} \rho\big)\circ\pi_x=\Ind_{S}^{\G} \rho.
\end{equation}
For every ideal $I\in \Prim C^{*}(\G)$ there exists a unit $x\in \Gu$ such that $C_{0}(\Gu)\cap I=C_{0}\big(\overline{[x]}^{c}\big)$, implying that
$I=\pi_x^{-1}(J)$ for some $J\in\Prim C^*(\G_{\overline{[x]}})$ such that $C_0\big(\G_{\overline{[x]}}^{(0)}\big)\cap J=0$. Using~\eqref{eqpiker} for $S=\Gxx$ and $S=\IsoGx{x}^\circ_x$ we then conclude that Theorem~\ref{thm:essentially-central} follows from parts~(2) and~(3) of Proposition~\ref{prop:essentially-central} applied to the groupoids $\G_{\overline{[x]}}$.

\smallskip

Proposition~\ref{prop:essentially-central} is an extension of~\cite{CN2}*{Proposition~6.3}. The part of the proof in \cite{CN2} that requires some changes is contained in the following lemma. In the setting of the proposition, for $z\in\Gu$ and $\chi\in\widehat{Z(\Gamma)}$, let $\rho^\chi_z:=\Ind^\G_{\Iso_z}(\chi\circ\Phi|_{\Iso_z})$. Therefore $\rho^\chi_z=\omega^\chi_z$ if the orbit of $z$ is dense.

\begin{lemma}\label{lem:equiv}
For each $\chi\in\widehat{Z(\Gamma)}$, the weak equivalence class of $\rho^\chi_z$ is independent of $z$ such that $\overline{[z]}=\Gu$.
\end{lemma}

\bp This corresponds to~\cite{CN2}*{Lemma~6.6}. What needs a modification in that lemma is the proof of the claim that if $w_n\to z$, then $\rho^\chi_z\prec\bigoplus_n\rho^\chi_{w_n}$. Specifically, we need to show that if $W\subset\G$ is an open bisection such that $z\in s(W)$, then $W\cap \Iso_z\ne\emptyset$ if and only if $W\cap  \Iso_{w_n}\ne\emptyset$ for all $n$ large enough. The ``only if'' part is obvious. For the ``if'' part, let $g_n$ be the unique element of $W\cap  \Iso_{w_n}$. Then $g_n\to g$ for some $g\in W\cap\G^z_z$. Since $g\in\overline{\IsoG^\circ}$, by Lemma~\ref{lem:iso} it follows that $g\in\Iso_z$.
\ep

\bp[Sketch of the proof of Proposition~\ref{prop:essentially-central}]
Fix a unit $x_0$ with dense orbit such that $\Iso_{x_0}=\G^{x_0}_{x_0}$, which is possible by Lemma~\ref{lem:nice-points}. We will first prove the proposition for $x=x_0$. In this case there is no difference between statements (1) and (2).

Take $I\in\mathcal I$. By Theorem~\ref{thm:EH} we have $I=\Ind(y,J)$ for some $(y,J)\in\Stab(\G)^\prim$. As $C_0(\Gu)\cap I=0$, the orbit of $y$ must be dense. Choose $\chi\in\widehat{Z(\Gamma)}$ such that $\chi_J=\chi\circ\Phi|_{\Iso_y}$.
By~\cite{CN2}*{Lemma 6.5} then $\pi_J\prec\Ind^{\G_{y}^{y}}_{\Iso_y}\chi_J$, so Lemma~\ref{lem:equiv} implies that
$$
\ker\rho^\chi_{x_{0}}=\ker\rho^\chi_y\subset I,
$$
see the arguments in~\cite{CN2} between Lemma~6.5 and Lemma~6.7. On the other hand, by \cite{CN2}*{Theorem~2.14} there exists $\omega\in\widehat{Z(\Gamma)}$ such that $I\subset \ker\rho^\omega_{x_{0}}$. By~\cite{CN2}*{Lemma~6.7}, whose proof goes through without any changes (with $A:=\Phi(\Iso_{x_{0}})=\Phi(\G_{x_{0}}^{x_{0}})$), the inclusion $\ker\rho^\chi_{x_{0}}\subset \ker\rho^\omega_{x_{0}}$ implies that $\chi=\omega$ on $\Phi(\Iso_{x_{0}})$. Hence $\ker\rho^\chi_{x_{0}}=I$. Since $\rho^\chi_z=\omega^\chi_z$ for the units with dense orbits, this proves~(1) and~(2) for $x=x_{0}$, as well as (3) for $y=x=x_{0}$.

Now, take any unit $y$ with dense orbit. For every $\chi\in\widehat{Z(\Gamma)}$ there is $J\in\Prim C^*(\G^y_y)$ such that $\chi_J=\chi\circ\Phi|_{\Iso_y}$, see
the last paragraph of the proof of \cite{CN2}*{Proposition~6.3}. Then we see from the previous paragraph that $\Ind(y,J)=\ker\rho^\chi_{x_{0}}=\ker\rho^\chi_y$. This shows that~(1) and~(2) hold for $x=y$. We also conclude that (3) is true for $x=x_{0}$, but then it is true for any~$x$ with dense orbit.
\ep

As an immediate application we see that the second part of Question~\ref{ques} has a positive answer for $\G$.

\begin{cor}\label{cor:very-dense}
In the setting of Theorem~\ref{thm:essentially-central}, assume $\Omega\subset\Gu$ is a subset that contains a representative of every quasi-orbit. Then every primitive ideal of $C^*(\G)$ has the form $\Ind(x,J)$, or equivalently the form $\ker\omega^\chi_x$, for some $x\in\Omega$, $J\in\Prim C^*(\Gxx)$ and $\chi\in\widehat{Z(\Gamma)}$.
\end{cor}

Another consequence of the theorem is that instead of the representations $\omega^\chi_x$ we can take many other representations to parameterize the primitive ideals.

\begin{cor}\label{cor:differen-reps}
In the setting of Theorem~\ref{thm:essentially-central}, take $x\in \Gu$ and $\chi\in\widehat{Z(\Gamma)}$. Assume $S\subset\Gxx$ is a subgroup containing $\IsoGx{x}^\circ_x$ and  $\pi$ is a unitary representation of $S$ such that $\pi|_{\IsoGx{x}^\circ_x}=(\chi\circ\Phi)(\cdot)1$. Then the representations $\omega^\chi_x$ and $\Ind^\G_S\pi$ of $C^*(\G)$ are weakly equivalent.
\end{cor}

\bp
The representation $\Ind^{\Gxx}_S\pi$ is weakly equivalent to a direct sum of irreducible representations $\pi_n$. Since $\IsoGx{x}^\circ_x$ is a central subgroup of $\Gxx$, the restriction of each of these representations to this subgroup is a scalar representation defined by the character $\chi\circ\Phi|_{\IsoGx{x}^\circ_x}$. By Theorem~\ref{thm:essentially-central} we conclude that $\Ind^\G_{\Gxx}\pi_n\sim\omega^\chi_x$ for all~$n$. It follows that
$$
\Ind^\G_S\pi\sim \Ind^\G_{\Gxx}\Ind^{\Gxx}_S\pi\sim\bigoplus_n \Ind^\G_{\Gxx}\pi_n\sim\omega^\chi_x,
$$
establishing the required equivalence.
\ep

Let us now consider the case where $\Gamma$ is abelian. Then by the previous corollary instead of the representations $\omega^\chi_x$  in Theorem~\ref{thm:essentially-central} we can take the irreducible representations
$$
\pi_{(x,\chi)}:=\Ind^\G_{\Gxx}(\chi\circ\Phi|_{\Gxx}).
$$
By definition, the underlying space of $\pi_{(x,\chi)}$ is the Hilbert space of functions $f\colon\G_x\to\C$ such that $f(gh)=\overline{\chi(\Phi(h))}f(g)$ for $g\in\G_x$ and $h\in \Gxx$, and $\sum_{g\in\G_x/\Gxx}|f(g)|^2<\infty$. We can identify this space with $\ell^2([x])$ by associating to every $\xi\in\ell^2([x])$ the function $f_\xi\colon\G_x\to\C$ by
$f_\xi(g):=\overline{\chi(\Phi(g))}\xi(r(g))$. Then $\pi_{(x,\chi)}$ becomes a representation $C^*(\G)\to B(\ell^2([x]))$ such that
\begin{equation}\label{eq:pi(x,z)}
\pi_{(x,\chi)}(f)\delta_y=\sum_{g\in\G_y}\chi(\Phi(g))f(g)\delta_{r(g)},\quad y\in[x],\ f\in C_c(\G).
\end{equation}

The following result shows that for abelian $\Gamma$ the assumption of essential centrality in Theorem~\ref{thm:essentially-central} is in a certain sense optimal.

\begin{prop}\label{prop:central-necessity}
Let $\G$ be a Hausdorff locally compact \'etale groupoid graded by a discrete abelian group $\Gamma$, with grading $\Phi\colon\G\to\Gamma$.
Assume there exists a nonempty $\G$-invariant closed subset $X\subset\Gu$ such that the groupoid $\Phi^{-1}(e)_X$ is not effective. Then the map
$$
\Gu\times\widehat{\Gamma}\to \Prim C^*(\G),\quad (x,\chi)\mapsto\ker\pi_{(x,\chi)},
$$
is not surjective. In particular, if $\G$ is amenable and second countable, then the above map is surjective if and only if $\G$ has essentially central isotropy.
\end{prop}

\bp
We argue similarly to the proof of Proposition~\ref{prop:gauge}. Note that if $x\notin X$, then $C_0(\Gu)\cap\ker\pi_{(x,\chi)}=C_0\big(\overline{[x]}^c\big)$ and therefore $C^*(\G_{X^c})\not\subset\ker\pi_{(x,\chi)}$. It follows that by passing to the quotient $C^*(\G_X)\cong C^*(\G)/C^*(\G_{X^c})$ we may assume that $X=\Gu$.  Then by assumption there is an open bisection $W$ of $\G$ contained in $\operatorname{Iso}(\Phi^{-1}(e))^\circ\setminus\Gu$. Take any nonzero function $f\in C_c(W)$. Define a function $f_0\in C_c(r(W))$ by $f_0(r(g)):=f(g)$ for $g\in W$. Then one can easily see from~\eqref{eq:pi(x,z)} that $f-f_0$ is contained in the kernel of every representation~$\pi_{(x,\chi)}$, hence the kernels of these representations cannot exhaust $\Prim C^*(\G)$.

The last statement of the proposition follows from Theorem~\ref{thm:essentially-central} and Corollary~\ref{cor:effective-vs-central}.
\ep

\subsection{The primitive spectrum as a topological space}
We will next describe the topology on $\Prim C^*(\G)$. For this it will be convenient to use the following notion~\cite{CN3}*{Definition~3.13}. Given a discrete abelian group $H$ and a sequence $(S_{n})^\infty_{n=1}$ of subsets of $H$, we say that a sequence $(\chi_{n})^\infty_{n=1}$ in $\widehat{H}$ \emph{converges to $\chi\in \widehat{H}$ along $(S_{n})_n$} if, for all $h \in H$, we have
$$
\lim_{n\to\infty}\un_{S_n}(h)|\chi_n(h)-\chi(h)|=0,
$$
where $\un_{S_n}$ is the characteristic function of $S_n$.

\begin{thm}\label{thm:convfg}
In the setting of Theorem~\ref{thm:essentially-central}, the topology on $\Prim C^*(\G)$ is described as follows. Fix $(x,\chi)\in\Gu\times\widehat{Z(\Gamma)}$ and, for every $g\in\IsoGx{x}^\circ_x$, choose an open bisection $W_g$ of~$\G$ containing~$g$. Then the following conditions on a sequence $((x_n,\chi_n))_n$ in $\Gu\times\widehat{Z(\Gamma)}$ are equivalent:
\begin{enumerate}
  \item $\ker\omega^{\chi_n}_{x_n}\to\ker\omega^\chi_x$ in $\Prim C^*(\G)$;
  \item there exist points $y_n\in[x_n]$ such that $y_n\to x$ and $\chi_n\to\chi$ along the sets
\begin{equation*}\label{eq:Rn}
R_n:=\{\Phi(g): g\in\IsoGx{x}^\circ_x,\ W_g\cap\IsoGx{x_n}^\circ_{y_n}\ne\emptyset\};
\end{equation*}
  \item there exist points $y_n\in[x_n]$ such that $y_n\to x$ and $\chi_n\to\chi$ along the sets
\begin{equation*}\label{eq:Sn}
S_n:=\{\Phi(g): g\in\IsoGx{x}^\circ_x,\ W_g\cap\G^{y_n}_{y_n}\ne\emptyset\}.
\end{equation*}
\end{enumerate}
\end{thm}

We remark that although the sets $R_n$ and $S_n$ depend on the choice of the bisections $W_g$, conditions (2) and (3) do not, since thanks to the condition $y_n\to x$, for any given $g\in\IsoGx{x}^\circ_x$, a different choice of $W_g$ changes the sets $W_g\cap\IsoGx{x_n}^\circ_{y_n}$ and $W_g\cap\G^{y_n}_{y_n}$ only for finitely many~$n$'s.

\smallskip

For the proof of the theorem we need the following lemma. We note in passing that it can be used to simplify some of the arguments in~\cite{CN3}*{Section~3.2}.

\begin{lemma}\label{lem:essential-pass}
Assume $\G$ is a Hausdorff locally compact \'etale groupoid and $x\in\Gu$ is a unit such that $\overline{[x]}=\Gu$ and $\overline{\IsoG^\circ}\cap\Gxx=\Iso_x$. Assume $U$ is an open neighbourhood of $x$ and $W_1,\dots,W_n$ are open bisections of $\G$ such that $U\subset\bigcap^n_{i=1}s(W_i)$ and $W_i\cap\Iso_x=\emptyset$ for all $i$. Then there exists $y\in [x]\cap U$ such that $W_i\cap\G^y_y=\emptyset$ for all $i$.
\end{lemma}

\bp
Let us show first that by replacing $U$ by a smaller neighbourhood we may assume that $W_i\cap\Iso_y=\emptyset$ for all $1\le i\le n$ and $y\in U$. Indeed, if this is not true, we can find an index~$i$  and a net of units $y_j$ such that $y_j\to x$ and $W_i\cap\Iso_{y_j}\ne\emptyset$ for all $j$. Then $W_i\cap\Gxx\ne\emptyset$, contradicting the assumptions  $\overline{\IsoG^\circ}\cap\Gxx=\Iso_x$ and $W_i\cap\Iso_x=\emptyset$.

We will now prove the lemma by induction on $n$. Assuming that it is proved for $n-1$, consider the map $T_n\colon s(W_n)\to r(W_n)$ defined by $W_n$. The assumption $W_n\cap\Iso_x=\emptyset$ means that $x$ does not belong to the interior of the set of points fixed by $T_n$. As $[x]$ is dense in $\Gu$, it follows that there exists $y'\in [x]\cap U$ such that $T_ny'\ne y'$. Choose an open neighbourhood $V$ of~$y'$ such that $V\subset U$ and $T_ny\ne y$ for all $y\in V$. Then, for all $y\in V$, we have $W_n\cap\G^y_y=\emptyset$. On the other hand, $W_i\cap\Iso_{y'}=\emptyset$ for all $1\le i\le n-1$, so by the inductive assumption there is $y\in [y']\cap V=[x]\cap V$ such that $W_i\cap\G^y_y=\emptyset$ for all $1\le i\le n-1$. The unit $y$ has the required properties.
\ep

\bp[Proof of Theorem~\ref{thm:convfg}]
We will first prove equivalence of (1) and (2). When $x$ and $x_n$ are such that $\IsoGx{x}^\circ_{x}=\Gxx$ and $\IsoGx{x_n}^\circ_{x_n}=\G^{x_n}_{x_n}$, then this equivalence follows from \cite{CN3}*{Corollary~2.7} (see also Corollaries~2.9 and~3.14 there for formulations of convergence closer to our current setup). As we have already used in the proof of Proposition~\ref{prop:essentially-central}, by Lemma~\ref{lem:nice-points}, for every unit $y\in\Gu$ there exists $y'\in\Gu$ such that $\overline{[y]}=\overline{[y']}$ and $\IsoGx{y}^\circ_{y'}=\G^{y'}_{y'}$. By Theorem~\ref{thm:essentially-central} we have $\ker\omega^\chi_y=\ker\omega^\chi_{y'}$. It follows that in order to prove that (1) and (2) are equivalent it suffices to show that condition (2) is stable under replacing~$x$ and~$x_n$ by units on the same quasi-orbits.

It will be convenient to replace (2) by the following equivalent condition, cf.~the proof of \cite{CN3}*{Corollary~3.14}. For every neighbourhood~$U$ of~$x$, every $\eps>0$, every finite subset $F\subset\IsoGx{x}^\circ_x$ and arbitrary open bisections $W_g$ containing $g\in F$, there exists an index~$n_0$ such that for each $n\ge n_0$ we can find $y\in [x_n]\cap U$ satisfying the following property: for every $g\in F$, we have
$$
\text{either}\quad W_g\cap\IsoGx{x_n}^\circ_y=\emptyset\quad \text{or}\quad |\chi(\Phi(g))-\chi_n(\Phi(g))|<\eps.
$$
Let us call this condition (2$'$). Assuming that this condition is satisfied, we want to show that then the same condition holds for any units on the quasi-orbits of $x$ and $x_n$.

Let us first consider $x'$ such that $\overline{[x']}=\overline{[x]}$. Take an open neighbourhood~$U$ of~$x'$, $\eps>0$, a finite subset $F\subset\IsoGx{x'}^\circ_{x'}$ and open bisections $W_g$ containing $g\in F$.  We may assume that $U$ is small enough so that $\Phi(W_g\cap s^{-1}(U))=\{\Phi(g)\}$ for all $g\in F$ and $W_g\cap \IsoGx{x'}^\circ_{y}\ne\emptyset$ for all $y\in \overline{[x']}\cap U$. It is easy to see that condition (2$'$) is stable under replacing $x$ by another point on the same orbit. Hence without loss of generality we may also assume that $x\in U$. Then applying condition (2$'$) to the neighbourhood $U$ of $x$, $\eps>0$, the set $(\bigcup_{g\in F}W_g)\cap \IsoGx{x}^\circ_{x}$ and the bisections $W_g$, we see that the same points $y\in [x_n]\cap U$ can be used to check (2$'$) for $x'$ instead of~$x$, the neighbourhood~$U$ of~$x'$, $\eps>0$, $F\subset\IsoGx{x'}^\circ_{x'}$ and the bisections $W_g$.

Next we keep $x$ intact and take units $x_n'$ such that $\overline{[x_n']}=\overline{[x_n]}$. In order to see that (2$'$) holds for $x_n'$ instead of $x_n$, it suffices to show the following: if $U$ is an open neighbourhood of $x$,  $F\subset\IsoGx{x}^\circ_{x}$ is finite, $W_g$ are bisections containing $g\in F$ such that $U\subset s(W_g)$, and $y\in [x_n]\cap U$ for some~$n$, then there exists $y'\in [x_n']\cap U$ such that for each $g\in F$ with $W_g\cap\IsoGx{x_n}^\circ_y=\emptyset$ we have $W_g\cap\IsoGx{x_n}^\circ_{y'}=\emptyset$. This is, in fact, true for any $y'\in [x_n']$ sufficiently close to~$y$. Indeed, assume this is not the case. Then there exist $g\in F$ and a sequence of elements $y_k\in[x_n']$ such that $y_k\to y$, $W_g\cap\IsoGx{x_n}^\circ_y=\emptyset$ and $W_g\cap\IsoGx{x_n}^\circ_{y_k}\ne\emptyset$ for all $k$. But this contradicts the property $\overline{\IsoGx{x_n}^\circ}\cap\G^y_y=\IsoGx{x_n}^\circ_y$ established in Lemma~\ref{lem:iso}, cf.~the proof of Lemma~\ref{lem:equiv}. This finishes the proof of equivalence of (1) and (2).

\smallskip

Condition (3) is obviously stronger than (2). It remains to show that (2) implies (3). Condition~(3) can be equivalently formulated as follows. For every neighbourhood~$U$ of~$x$, every $\eps>0$ and every finite subset $F\subset\IsoGx{x}^\circ_x$, there exists an index~$n_0$ such that for each $n\ge n_0$ we can find $y\in [x_n]\cap U$ satisfying the following property: for every $g\in F$, we have either $W_g\cap\G^y_y=\emptyset$ or $|\chi(\Phi(g))-\chi_n(\Phi(g))|<\eps$. Let us call this condition (3$'$).

We may assume that $U\subset\bigcap_{g\in F}s(W_g)$. By condition (2), in its equivalent form (2$'$), we know that we can find $n_0$ such that for each $n\ge n_0$ there is $y'\in [x_n]\cap U$ satisfying the following property: for every $g\in F$, we have either $W_g\cap\IsoGx{x_n}^\circ_{y'}=\emptyset$ or $|\chi(\Phi(g))-\chi_n(\Phi(g))|<\eps$. Note that $\overline{\operatorname{Iso}(\G_{\overline{[x_n]}})^\circ}\cap\G^{y'}_{y'}=\IsoGx{x_n}^\circ_{y'}$ by Lemma~\ref{lem:iso}. We then apply Lemma~\ref{lem:essential-pass} to the unit $y'$ and the reduced groupoid $\G_{\overline{[x_n]}}$ to find $y\in[y']\cap U=[x_n]\cap U$ such that $W_g\cap\G^y_y=\emptyset$ for all $g\in F$ with $W_g\cap\IsoGx{x_n}^\circ_{y'}=\emptyset$. The unit $y$ has the properties required in~(3$'$).
\ep

We can now show that the first part of Question~\ref{ques} also has a positive answer for $\G$.

\begin{cor}
In the setting of Theorem~\ref{thm:essentially-central}, assume $\Omega\subset\Gu$ is a subset such that its image in $(\G\backslash\Gu)^\sim$ is very dense. Then the set
$$
\{\Ind(x,J):x\in\Omega,\ J\in\Prim C^*(\Gxx)\}=\{\ker\omega^\chi_x:x\in\Omega,\ \chi\in\widehat{Z(\Gamma)}\}
$$
is very dense in $\Prim C^*(\G)$.
\end{cor}

\bp
Since the induction map is constant on the $\G$-orbits, without loss of generality we may assume that $\Omega$ is $\G$-invariant.
Take any point $\ker\omega^\chi_y$ in $\Prim C^*(\G)$. By assumption and Lemma~\ref{lem:recap}, the set $\Omega\cap\overline{[y]}$ is dense in $\overline{[y]}$. For every point $x$ in this set we can find points $y_n\in[y]$ such that $y_n\to x$. Then $\ker\omega^\chi_y=\ker\omega^\chi_{y_n}\to\ker\omega^\chi_x$, that is, $\ker\omega^\chi_x\in\overline{\{\ker\omega^\chi_y\}}$. On the other hand, we can find points $x_n\in\Omega\cap\overline{[y]}$ such that $x_n\to y$ and then get that $\ker\omega^\chi_{x_n}\to\ker\omega^\chi_y$. Therefore already the set $\{\ker\omega^\chi_x:x\in\Omega\cap\overline{[y]}\}$ is dense $\overline{\{\ker\omega^\chi_y\}}$. This implies the corollary.
\ep

\bigskip

\section{Exel--Pardo groupoids and their higher rank analogues} \label{sec:prelim}

\subsection{Higher rank graphs}\label{ssec:higher-rank}
We recall basic notions of the higher rank graph theory, see~\cite{KP} and \cite{MR2301938} for details.

Fix a natural number $k\ge1$. A \emph{$k$-graph} is a pair $(\Lambda, d)$, where $\Lambda$ is a small category, thought of as a set of morphisms, and $d\colon \Lambda \to \Z_+^{k}$ is a functor with the unique factorization property, meaning that if $d(\mu)=n+m$ for some $\mu \in \Lambda$ and $n,m\in \Z_+^{k}$, then there exist unique morphisms $\alpha, \beta \in \Lambda$ satisfying $d(\alpha)=n$, $d(\beta)=m$ and $\mu=\alpha \beta$.  For each $n\in \Z_+^{k}$, we set $\Lambda^{n}:=d^{-1}(n)$. Then~$\Lambda^0$ coincides with the set of identity morphisms in the category, so it can be identified with the set of objects, which are called the vertices of $\Lambda$. Any $1$-graph corresponds to an oriented graph in the usual sense, with the vertex set $\Lambda^0$ and the edge set $\Lambda^1$; then $\Lambda^n$ is the set of paths of length $n$ in the graph.

For elements $n,m \in \Z_+^{k}$, we define $n \vee m \in \Z_+^{k}$ by $(n \vee m)_{i} := \max(n_{i}, m_{i})$ and we define $n \wedge m \in \Z_+^{k}$  by $(n \wedge m)_{i} := \min(n_{i}, m_{i})$. A  $k$-graph is called \emph{finitely aligned} if for all $\lambda , \mu \in \Lambda$ the set
$$
\Lambda^{\text{min}}(\lambda, \mu) := \{ (\alpha, \beta)\in\Lambda\times\Lambda : \lambda \alpha =\mu \beta \; , d(\lambda \alpha)=d(\lambda) \vee d(\mu) \}
$$
is finite. For $k=1$ this property obviously always holds. We assume that
\textbf{all $k$-graphs considered in this paper are finitely aligned}.

\smallskip

Denote by $r,s\colon\Lambda\to\Lambda^0$ the codomain and domain maps, resp. For subsets $A, B \subset \Lambda$ we define
$$
AB:=\{ \alpha \beta : \alpha \in A,\ \beta \in B\ \ \text{and}\ \ s(\alpha)=r(\beta) \}.
$$
Therefore $v\Lambda:=\{v\}\Lambda$ for $v \in \Lambda^{0}$ is the set of morphisms $\alpha \in \Lambda$ with $r(\alpha)=v$. A $k$-graph $\Lambda$ is called \emph{row-finite} if $|v\Lambda^{n}|<\infty$ for all $v\in \Lambda^{0}$ and $n \in \Z_+^{k}$, and it is called \emph{source-free} if $v \Lambda^{n} \neq \emptyset$ for all $v\in \Lambda^{0}$ and $n \in \Z_+^{k}$.

A functor $x\colon \Lambda_{1} \to \Lambda_{2}$ between two $k$-graphs is called a graph morphism if it preserves the degree map, that is, $d(x(\lambda))=d(\lambda)$ for all $\lambda$.
For  each $m \in (\Z_+ \cup \{\infty\})^{k}$ we define a $k$-graph by considering the set
$$
\Omega_{m, k} := \{ (p,q) \in \Z_+^{k} \times \Z_+^{k} \mid p \leq q \leq m \}
$$
with $r(p,q)=p$, $s(p,q)=q$, $(n,p)(p, m)=(n,m)$ and $d(p,q)=q-p$.
We now define the \emph{path space} of $\Lambda$ to be the set
$$
X_{\Lambda} := \bigcup_{m \in (\Z_+ \cup \{\infty\})^{k}} \{ x: \Omega_{m,k} \to \Lambda \mid \text{$x$ is a graph morphism} \}.
$$
We write $x(n)$ for $x(n,n)\in\Lambda^0$. Define the (generalized) degree map on $X_\lambda$ by setting $d(x):=m \in (\Z_+ \cup \{\infty\})^{k}$ when $x\colon \Omega_{m,k} \to \Lambda$, and the range map $r\colon X_{\Lambda} \to \Lambda^{0}$ by setting $r(x):=x(0)$. For $m \in \Z_+ ^{k}$ the graph morphisms $x\colon \Omega_{m,k} \to \Lambda$ can be identified with the elements of $\Lambda^{m}$, because for each morphism $\lambda \in \Lambda^{m}$ there exists a unique graph morphism $x\colon \Omega_{m,k} \to \Lambda$ such that $x(0, m)=\lambda$. We therefore often view $\Lambda$ as a subset of $X_{\Lambda} $ and then refer to the elements of $\Lambda$ as finite paths.

For $x\in X_{\Lambda}$ and $m \in \Z_+^{k}$ with $m \leq d(x)$ we can define a unique path $\sigma^{m}(x) \in X_{\Lambda}$ such that $\sigma^{m}(x)(p, q)=x(m+p, m+q)$ for all $p, q \in \Z_+^{k}$ with $p \leq q \leq d(x)-m$. For any $x\in X_{\Lambda}$ and $\lambda \in \Lambda$ with $s(\lambda)=r(x)$ one can likewise define a concatenation of paths $\lambda x$ 
such that $(\lambda x)(0,d(\lambda))=\lambda$ and $\sigma^{d(\lambda)}(\lambda x)=x$.

\smallskip

For each $\lambda \in \Lambda$ define a cylinder set $Z(\lambda)$ in $X_{\Lambda}$ by
$$
Z(\lambda) := \{ x\in X_{\Lambda} : d(x) \geq d(\lambda)\ \ \text{and}\ \ x=\lambda y\ \ \text{for some}\ \ y \in X_{\Lambda}\}.
$$
For a finite subset $F \subset s(\lambda)\Lambda$ put then $Z_{F}(\lambda):=Z(\lambda) \setminus \bigcup_{\alpha \in F} Z(\lambda \alpha)$. These sets give a basis of a Hausdorff locally compact topology on $X_\Lambda$. The sets $\{x \in X_{\Lambda} : d(x) \geq m \}$ are open in~$X_{\Lambda}$ and the maps $\sigma^{m} \colon \{x \in X_{\Lambda} \mid d(x) \geq m \} \to X_{\Lambda}$ are local homeomorphisms. We remark that the assumption of finite alignment is essential to ensure local compactness of the topology, cf.~\cite{MR2301938}*{Remark 3.7}. When $\Lambda$ is in addition countable, then $X_\Lambda$ is second countable.

\smallskip

Given a vertex $v\in \Lambda^{0}$, a subset $E \subset v\Lambda$ is called \emph{exhaustive} for $v$ if for each $\lambda \in v\Lambda$ there exists an element $\mu\in E$ such that $\Lambda^{\min}(\lambda, \mu) \neq \emptyset$. A path $x\in X_{\Lambda}$ is called a \emph{boundary path} if it satisfies the following condition: for every $m \in \Z_+^{k}$ with $m \leq d(x)$ and for every \emph{finite} exhaustive set $E \subset x(m)\Lambda$, there exists an element $\lambda \in E$ such that $x(m, m+d(\lambda))=\lambda$. We let $\partial \Lambda$ denote the set of all boundary paths in $\Lambda$. The set $\partial \Lambda$ is a closed subset of $X_{\Lambda}$ invariant under the shifts~$\sigma^{m}$ and their inverses, meaning that if $x\in X_\Lambda$, $\lambda \in \Lambda$ and $s(\lambda)=r(x)$, then $\lambda x \in \partial \Lambda$ if and only if $x\in\partial\Lambda$, see \cite{MR2301938}*{Propositions~4.4 and~4.7}.

For $k=1$ the set $\partial\Lambda$ has a simple description. Namely, consider the set
$$
\Lambda^\sing:=\{v\in \Lambda^{0} : |v\Lambda^1| \in \{0, \infty\}\}
$$
of \emph{singular vertices}. Then $\partial\Lambda$ consists of all infinite paths $x$ and all finite paths $\lambda\in\Lambda$ such that $s(\lambda)\in\Lambda^\sing$, see~\cite{KL}*{Proposition 4.6}.

\subsection{Self-similar \texorpdfstring{$k$}{k}-graphs and the associated groupoids}
Let $\Lambda$ be a $k$-graph and $G$ be a discrete group with the unit element~$1_G$. For the purpose of the next definition, by an action of~$G$ on $\Lambda$ we mean an action of~$G$ on the set $\Lambda$ that preserves the degree map and commutes with the range and source maps, that is, $d(g\cdot\lambda)=d(\lambda)$, $r(g\cdot\lambda)=g\cdot r(\lambda)$ and $s(g\cdot\lambda)=g\cdot s(\lambda)$, but we do not require the maps $\lambda\mapsto g\cdot\lambda$ to be graph morphisms, that is, to respect the composition of morphisms in $\Lambda$.

\begin{defn}[{\cites{MR3581326,MR4294118}}]\label{def:selfsim}
An action $G\curvearrowright\Lambda$ is called \emph{self-similar}, and the pair $(G,\Lambda)$ is called a \emph{self-similar $k$-graph}, if we are also given a map $G\times \Lambda \to G$, $(g,\mu)\mapsto g|_\mu$ , satisfying the following properties:
\begin{enumerate}
\item[(i)] $g \cdot (\mu \nu) = (g \cdot \mu)(g|_{\mu} \cdot \nu) $ for all $g \in G$ and $\mu, \nu \in \Lambda$ with $s(\mu)=r(\nu)$;
\item[(ii)] $g|_{\mu \nu} =\big(g|_{\mu}\big)|_{\nu}$ for all $g\in G$ and $\mu, \nu \in \Lambda$ with $s(\mu)=r(\nu)$;
\item[(iii)]  $(gh)|_{\mu}=g|_{h\cdot \mu} h|_{\mu}$ for all $g,h \in G$ and $\mu \in \Lambda$;
\item[(iv)] $g|_{v}=g$ for all $g\in G$ and $v\in \Lambda^{0}$.
\end{enumerate}

We furthermore call the action $G\curvearrowright\Lambda$, or the self-similar $k$-graph $(G,\Lambda)$, \emph{pseudo-free} if for all $g\in G$ and $\mu\in \Lambda$ the two conditions $g\cdot \mu=\mu$ and $g|_{\mu}=1_{G}$ imply that $g=1_{G}$. We say that $(G,\Lambda)$ is countable if both $G$ and $\Lambda$ are countable.
\end{defn}

Note that condition (i) applied to $\nu=s(\mu)$ implies that for all $g\in G$ and $\mu\in\Lambda$ we have
$$
g\cdot s(\mu)=g|_\mu\cdot s(\mu),
$$
while condition (iii) applied to $g=h=1_G$ implies that $1_{G}|_{\mu}=1_{G}$ for all $\mu\in\Lambda$.

Any action $G\curvearrowright\Lambda$ by graph automorphisms gives an example of a pseudo-free self-similar action such that $g|_\mu=g$ for all $g\in G$ and $\mu\in\Lambda$. Conversely, if $(G,\Lambda)$ is a self-similar $k$-graph such that $g|_\mu=g$ for all $g\in G$ and $\mu\in\Lambda$, then $G$ acts on $\Lambda$ by graph automorphisms.

\smallskip

Assume $(G,\Lambda)$ is a self-similar $k$-graph. Our goal is to define certain groupoids associated with $(G,\Lambda)$. For $k=1$ these groupoids were introduced under some conditions in~\cite{MR3581326} as tight groupoids of inverse semigroups associated with self-similar graphs. They were also shown to provide models for certain Cuntz--Pimsner algebras. Without any restrictions on the graph, and even in a more general setting of self-similar groupoid actions, these constructions have been discussed in~\cites{MS,KM}. In the higher rank case different points of view haven't been fully reconciled. We will not try to complete the picture here and rather introduce the groupoids in an ad hoc manner following~\cite{MR4294118} and checking that everything works without the assumptions on the $k$-graphs made there.

It is not difficult to see that the action of~$G$ on~$\Lambda$ extends uniquely to an action on $X_\Lambda$ such that
$$
(g\cdot x) (n,m)= g|_{x(0, n)} \cdot x(n,m)
$$
for all $x\in X_{\Lambda}$, $g\in G$ and $n,m \in \Z_+^k$ with $n\le m \leq d(x)$. It is clear that this action satisfies $d(g\cdot x)=d(x)$ and $r(g\cdot x)=g\cdot r(x)$ for all $g\in G$ and $x\in X_{\Lambda}$. It is not difficult to see, cf.~Lemma~\ref{lem:calcrules}(1) below, that
$$
g\cdot \Big(Z(\lambda) \setminus \bigcup_{\alpha \in F} Z(\lambda \alpha)\Big)
=Z(g\cdot\lambda) \setminus \bigcup_{\alpha \in F} Z\big((g\cdot\lambda) (g|_{\lambda}\cdot\alpha)\big)\; .
$$
It follows that $G$ acts by homeomorphisms of $X_{\Lambda}$.

\begin{lemma} \label{lem:Ginvbound}
The subset $\partial\Lambda\subset X_\Lambda$ is $G$-invariant.
\end{lemma}

\begin{proof}
Take $g\in G$. We claim first that if $E \subset v \Lambda$ is an exhaustive set for a vertex $v\in \Lambda^{0}$, then $g\cdot E \subset (g\cdot v) \Lambda$ is an exhaustive set for $g\cdot v$. Indeed, if $\lambda\alpha=\mu\beta$ in $\Lambda$, then
$$
(g\cdot\lambda)(g|_\lambda\cdot\alpha)=(g\cdot\mu)(g|_\mu\cdot\beta),
$$
which implies that the sets $\Lambda^{\min}(\lambda, \mu)$ and $\Lambda^{\min}(g \cdot \lambda,g\cdot \mu)$ have the same cardinality. This proves our claim.

Assume now that $x \in \partial \Lambda$ is a boundary path and consider a finite exhaustive set
$$
E \subset (g\cdot x)(m) \Lambda=(g\cdot x(m))\Lambda=(g|_{x(0,m)}\cdot x(m))\Lambda
$$
for some $m \leq d(x)$. Then $(g|_{x(0, m)})^{-1} \cdot E$ is a finite exhaustive set for $x(m)\Lambda$, and hence there is a path $\lambda \in E$ such that $x(m, m+d(\lambda))= (g|_{x(0, m)})^{-1}\cdot \lambda$. Then $(g\cdot x)(m, m+d(\lambda))= g|_{x(0, m)} \cdot x(m, m+d(\lambda)) = \lambda$, and we conclude that $g\cdot x$ is again a boundary path. This finishes the proof of the lemma.
\end{proof}

Denote by $C(\Z_+^k; G)$ the group of all mappings from $\Z_+^k$ to $G$, with the group operation given by pointwise multiplication. For $m \in \Z^{k}$ and $f \in C(\Z_+^k; G)$ we define $\mathcal{T}_{m}(f) \in C(\Z_+^k; G)$ by
$$
\mathcal{T}_{m}(f) (n):=
\begin{cases}
f(n-m),&\text{if}\ \ n \geq m,\\
1_G, \ &\text{otherwise.}
\end{cases}
$$
Let $Q_k(G)$ denote the quotient group $C(\Z_+^k; G)/\sim$ with respect to the equivalence relation such that $f_{1} \sim f_{2}$ if and only there exists $m\in \Z_+^k$ such that $f_{1}(n)=f_{2}(n)$ for all $n \geq m$. It is then straightforward to check that $\mathcal{T}_{m}$ defines an automorphism on $Q_k(G)$, which we continue to denote by $\TT_m$, so $\mathcal{T}_{m}([f])=[\mathcal{T}_{m}(f) ]$. This way we get an action of $\Z^k$ on $Q_k(G)$ by automorphisms and hence a semidirect product $Q_k(G)\rtimes \Z^k$, called the \emph{lag group}.

Next, for $x\in X_{\Lambda}$ and $g\in G$, consider the element $f\in C(\Z_+^k; G)$ defined by
$$
f(n): = g|_{x(0, n \wedge d(x))}.
$$
We denote by $[g|_x]$ the class of $f$ in $Q_k(G)$. For $x\in \Lambda \subset X_{\Lambda}$ we have $f(n) = g|_{x}$ for all $n \geq d(x)$, and hence $[f] $ coincides with the equivalence class of the constant map $\Z_+^k \to G$, $n \mapsto g|_{x}$. Therefore the notation $[g|_{x}]$ should be unambiguous. Using this notation we get
the following generalizations of identities from Definition \ref{def:selfsim}.

\begin{lemma} \label{lem:calcrules}
For all $g,h\in G$, $x\in X_\Lambda$ and $\lambda \in \Lambda$  with $r(x)=s(\lambda)$, we have:
\begin{enumerate}
\item $g \cdot (\lambda x) = (g \cdot \lambda) (g|_{\lambda} \cdot x)$;
\item $[g|_{\lambda x} ]=\mathcal{T}_{d(\lambda)} ([(g|_{\lambda})|_{x}])$;
\item $[(gh)|_{x}]=[g|_{h \cdot x}][h|_{x}]$.
\end{enumerate}
\end{lemma}

\begin{proof}
The first identity holds, since $\big(g \cdot (\lambda x)\big) (0, d(\lambda) ) = g\cdot (\lambda x) (0, d(\lambda) ) = g \cdot \lambda$ and
$$
\big(g \cdot (\lambda x)\big) (d(\lambda), d(\lambda)+n ) = g|_{(\lambda x) (0,d(\lambda))} \cdot (\lambda x) (d(\lambda), d(\lambda)+n )  = g|_{\lambda} \cdot x(0, n)= (g|_{\lambda} \cdot x )(0, n).
$$
The second identity follows by observing that for all $n \in \Z_+^k$ with $n \geq d(\lambda)$  we have
$$
g|_{(\lambda x) (0, n \wedge d(\lambda x))} = g|_{(\lambda x) (0, d(\lambda)+(n- d(\lambda) )\wedge d( x))}
= g|_{\lambda x (0, (n- d(\lambda) )\wedge d( x))}
=(g|_{\lambda})|_{x (0, (n- d(\lambda) )\wedge d( x))}.
$$
Finally, the last identity follows, since for all $n \in \Z_+^k$ we have
$$
(gh)|_{x(0,n \wedge d(x))} = g|_{h \cdot x(0, n\wedge d(x))} h|_{x(0, n\wedge d(x))} =  g|_{(h \cdot x)(0, n\wedge d(x))} h|_{x(0, n\wedge d(x))}.
$$
This completes the proof of the lemma.
\end{proof}

The pseudo-freeness assumption, when it is needed, will usually be used in the following form.

\begin{lemma}[cf.~\citelist{\cite{MR3581326}*{Proposition~5.6}\cite{MR4294118}*{Corollary 5.7}}]\label{lem:cor57}
Assume $(G, \Lambda)$ is pseudo-free. Then, for any $g,h \in G$ and $x\in X_{\Lambda}$, the two conditions $g \cdot x = h \cdot x$ and $[g|_{x}]=[h|_{x}]$ imply that $g=h$.
\end{lemma}

\bp
Since $[h|_x]^{-1}=[h^{-1}|_{h\cdot x}]=[h^{-1}|_{g\cdot x}]$ and hence $1_{Q_k(G)}=[h|_x]^{-1}[g|_x]=[(h^{-1}g)|_x]$, we may assume in addition that $h=1_G$.
The assumption $[g|_{x}]=1_{Q_k(G)}$ means then that $g|_{x(0, n \wedge d(x))}=1_G$ for all~$n$ large enough. Since we also have $g\cdot x(0,n \wedge d(x))=x(0,n \wedge d(x))$, we conclude that $g=1_G$ by the pseudo-freeness assumption.
\ep

Consider the groupoid $X_\Lambda\times (Q_k(G)\rtimes\Z^k)\times X_\Lambda$ defined as the product of the group $Q_k(G)\rtimes\Z^k$ and the full equivalence relation on $X_\Lambda$, so the product on $X_\Lambda\times (Q_k(G)\rtimes\Z^k)\times X_\Lambda$ is defined~by
$$
(x,a,y)(y,b,z)=(x,ab,z)
$$
for $x,y,z\in X_\Lambda$ and $a,b\in Q_k(G)\rtimes\Z^k$. We view $X_\Lambda\times (Q_k(G)\rtimes\Z^k)\times X_\Lambda$ as a purely algebraic groupoid without attempting to define a topology on it.

\begin{prop} \label{prop:subgroupoid}
For any self-similar $k$-graph $(G, \Lambda)$, the set
\begin{align*}
  \mathcal{G}_{(G, \Lambda)}:&=\{(y,\TT_m([g|_{\sigma^n(z)}]),m-n,z): y\in\operatorname{dom}(\sigma^m),\ z\in\operatorname{dom}(\sigma^n),\ \sigma^m(y)=g\cdot\sigma^n(z)\} \\
 &= \left \{  (\mu (g\cdot x) , \mathcal{T}_{d(\mu)}([g|_{x}]), d(\mu)-d(\nu), \nu x ): s(\nu)=r(x),\ s(\mu)=g\cdot r(x)\right\}
\end{align*}
is a subgroupoid of $X_{\Lambda} \times (Q_k(G) \rtimes \Z^{k}) \times X_{\Lambda}$.
\end{prop}

\begin{proof}
Observe that if  $x=\lambda y $ for some $\lambda \in \Lambda$ and $y\in X_\Lambda$, then
\begin{multline*}
(\mu (g\cdot x) , \mathcal{T}_{d(\mu)}([g|_{x}]), d(\mu)-d(\nu), \nu x )\\
=(\mu (g\cdot \lambda) (g|_{\lambda} \cdot y), \mathcal{T}_{d(\mu (g \cdot \lambda))}([(g|_{\lambda})|_{y}]), d(\mu (g\cdot \lambda))-d(\nu\lambda), \nu \lambda y ).
\end{multline*}
It follows that when we want to prove that the product of two elements
$$
(\mu (g\cdot x) , \mathcal{T}_{d(\mu)}([g|_{x}]), d(\mu)-d(\nu), \nu x )\quad\text{and}\quad(\alpha (h\cdot y) , \mathcal{T}_{d(\alpha)}([h|_{y}]), d(\alpha)-d(\beta), \beta y ),
$$
with $\nu x=\alpha (h\cdot y)$, lies in $\mathcal{G}_{(G, \Lambda)}$, we may assume without loss of generality that $\nu = \alpha$ and $x= h \cdot y$. Then the product becomes
$$
(\mu ((gh)\cdot y) , \mathcal{T}_{d(\mu)}([g|_{x}]) \mathcal{T}_{d(\alpha)+d(\mu)-d(\nu)}([h|_{y}]), d(\mu)-d(\beta), \beta y ),
$$
but since $d(\alpha)=d(\nu)$, the entry in $Q_k(G)$ equals $\mathcal{T}_{d(\mu)}([g|_{h \cdot y}] [h|_{y}])=\mathcal{T}_{d(\mu)}([(gh)|_{y}])$, proving that $\mathcal{G}_{(G, \Lambda)}$ is closed under composition.

To see that $\mathcal{G}_{(G, \Lambda)}$ is closed under inverses, we have to show that the element
$$
(\mu (g\cdot x) , \mathcal{T}_{d(\mu)}([g|_{x}]), d(\mu)-d(\nu), \nu x )^{-1}
=(\nu x  , \mathcal{T}_{d(\nu)-d(\mu)}(\mathcal{T}_{d(\mu)}([g|_{x}]^{-1})), d(\nu)-d(\mu),  \mu (g\cdot x))
$$
again lies in $\mathcal{G}_{(G, \Lambda)}$. Setting $y:=g\cdot x$ and substituting we get that the inverse equals
$$
(\nu (g^{-1}\cdot y)  , \mathcal{T}_{d(\nu)}([g|_{g^{-1} \cdot y}]^{-1}), d(\nu)-d(\mu),  \mu y) .
$$
Since $[g|_{g^{-1} \cdot y}] [g^{-1}|_{y}]=[(gg^{-1})|_{y}]=1_{Q_k(G)}$, we have $[g|_{g^{-1} \cdot y}]^{-1} = [g^{-1}|_{y}]$, which finishes the proof.
\end{proof}

For $g\in G$, $\mu, \nu \in\Lambda$ with $g\cdot s(\nu)=s(\mu)$ and a finite set $F \subset s(\nu) \Lambda$, define
$$
Z(\mu,g,\nu, F) := \{ (\mu (g\cdot x) , \mathcal{T}_{d(\mu)}([g|_{x}]), d(\mu)-d(\nu), \nu x ) : x\in Z_{F}(s(\nu)) \},
$$
where $Z_{F}(s(\nu))= Z(s(\nu)) \setminus \bigcup_{\lambda \in F}Z(\lambda)$ are the sets introduced in Section~\ref{ssec:higher-rank}. When $F=\emptyset$, we will often write $Z(\mu,g,\nu )$ for $Z(\mu,g,\nu, \emptyset)$.

\begin{prop} \label{prop:topetale}
Assume $(G,\Lambda)$ is a pseudo-free self-similar $k$-graph. Then the sets $Z(\mu,g,\nu, F)$ constitute a basis of a topology on $\mathcal{G}_{(G, \Lambda)}$ that makes it into a Hausdorff locally compact  \'etale groupoid with unit space $X_\Lambda$. Moreover, the map $\Phi\colon \G_{(G,\Lambda)}\to\Z^k$, $(y,q,n,z)\mapsto n$, defines a $\Z^k$-grading on $\G_{(G,\Lambda)}$. If the self-similar $k$-graph $(G,\Lambda)$ is countable, then the groupoid $\G_{(G,\Lambda)}$ is second countable.
\end{prop}

For the proof we need a few lemmas.

\begin{lemma}\label{lem:basis0}
Assume that $(y, q, n, z) \in Z(\mu,g,\nu, F)$ and $z\in Z_{F'}(\nu')$. Then there exist paths $\eta,\eta'\in\Lambda$ satisfying $\nu\eta=\nu'\eta'=z(0,m)$ for some $m\le d(z)$, and for any such $\eta$ and $\eta'$ we have
$$
(y, q, n, z) \in Z(\mu (g\cdot \eta) ,g|_{\eta},\nu \eta, F'') \subset Z(\mu,g,\nu, F)\quad\text{and}\quad z\in Z_{F''}(\nu\eta)\subset Z_{F'}(\nu'),
$$
where $F''$ consists of all paths $\lambda\in s(\eta)\Lambda=s(\eta')\Lambda$ such that there exists $\rho\in\Lambda$ satisfying either $(\lambda, \rho) \in \Lambda^{\min}(\eta,\kappa)$ for some $\kappa\in F$ or $(\lambda, \rho) \in \Lambda^{\min}(\eta',\kappa)$ for some $\kappa\in F'$.
\end{lemma}

\bp
First of all observe that by the finite alignment assumption the set $F''$ is finite.

Since $z\in Z(\nu)\cap Z(\nu')$, the existence of $\eta$ and $\eta'$ such that $\nu\eta=\nu'\eta'=z(0,m)$ for some $m\le d(z)$ is obvious.
We then have $z=\nu\eta x$ for some $x\in X_\Lambda$. Since $(y, q, n, z) \in Z(\mu,g,\nu, F)$, as in the in the proof of Proposition~\ref{prop:subgroupoid}, we get
\begin{equation*}\label{eq:Haus1}
(y, q, n, z)=(\mu(g\cdot\eta)(g|_\eta\cdot x), \mathcal{T}_{d(\mu(g\cdot\eta))}([(g|_{\eta})|_x]), d(\mu)-d(\nu), \nu\eta x ).
\end{equation*}
We then see that $(y,q,n,z)\in Z(\mu (g\cdot \eta) ,g|_{\eta},\nu \eta) \subset Z(\mu,g,\nu)$, and in order to finish the proof of the lemma we need to check only that $z\in Z_{F''}(\nu\eta)\subset Z_F(\nu)\cap Z_{F'}(\nu')$. This can be easily done using the definition of $F''$.
\ep

The pseudo-freeness assumption has the following implication, which will be vital to guarantee that the topology is Hausdorff.

\begin{lemma}\label{lem:separat}
Whenever $g \neq h$ and $d(\mu)=d(\alpha)$, we have $Z(\mu,g,\nu) \cap Z(\alpha,h,\beta) = \emptyset$.
\end{lemma}

\begin{proof}
Assume there exists an element $(y, q, n, z)  \in Z(\mu,g,\nu) \cap Z(\alpha,h,\beta)$ for some $g,h\in G$ and $\mu,\alpha\in\Lambda$ with $d(\mu)=d(\alpha)$. This implies that $d(\mu)-d(\nu)=d(\alpha)-d(\beta)$, and hence $d(\nu)=d(\beta)$. We can write $y=\mu (g \cdot x)=\alpha (h \cdot t)$ and $z=\nu x=\beta t$ for some $x,t\in X_\Lambda$. Since $d(\nu)=d(\beta)$, we conclude that $x=t$, and since $d(\mu)=d(\alpha)$, we get that $g \cdot x = h \cdot t=h\cdot x$. The condition $\mathcal{T}_{d(\mu)}([g|_{x}])=q=\mathcal{T}_{d(\alpha)}([h|_{t}])$ ensures that $[g|_{x}]=[h|_t] = [h|_{x}]$, and hence $g=h$ by Lemma~\ref{lem:cor57}.
\end{proof}


A similar argument combined with Lemma~\ref{lem:basis0} gives the following result.

\begin{lemma}\label{lem:basis}
Assume that $(y, q, n, z) \in Z(\mu,g,\nu, F) \cap Z(\mu ',g',\nu ', F ')$. Then there exist paths $\eta,\eta'\in\Lambda$ satisfying $\nu\eta=\nu'\eta'=z(0,m)$ for some $m\le d(z)$, and for any such $\eta$ and $\eta'$ we have $\mu(g\cdot \eta)=\mu' (g'\cdot \eta')$, $g|_{\eta}=g'|_{\eta'}$ and
$$
(y, q, n, z) \in Z(\mu (g\cdot \eta) ,g|_{\eta},\nu \eta, F'') \subset Z(\mu,g,\nu, F) \cap Z(\mu ',g',\nu ', F '),
$$
where $F''$ consists of all paths $\lambda\in s(\eta)\Lambda=s(\eta')\Lambda$ such that there exists $\rho\in\Lambda$ satisfying either $(\lambda, \rho) \in \Lambda^{\min}(\eta,\kappa)$ for some $\kappa\in F$ or $(\lambda, \rho) \in \Lambda^{\min}(\eta',\kappa)$ for some $\kappa\in F'$.
\end{lemma}

\bp
The existence of $\eta$ and $\eta'$ is obvious, and by Lemma~\ref{lem:basis0} we have
$$
(y, q, n, z) \in Z(\mu (g\cdot \eta) ,g|_{\eta},\nu \eta, F'') \subset Z(\mu,g,\nu, F),
$$
$$
(y, q, n, z) \in Z(\mu' (g'\cdot \eta') ,g'|_{\eta'},\nu' \eta', F'') \subset Z(\mu',g',\nu', F').
$$
It remains to show that $\mu(g\cdot \eta)=\mu' (g'\cdot \eta')$ and $g|_{\eta}=g'|_{\eta'}$. Write $z$ as $z=\nu\eta x=\nu'\eta' x$. Since
$$
d(\mu)-d(\nu)=n=d(\mu')-d(\nu')\quad\text{and}\quad d(\nu)+d(\eta)=d(\nu')+d(\eta'),
$$
we have $d(\mu)+d(\eta)=d(\mu')+d(\eta')$. Since we also have
$$
\mu(g\cdot\eta)(g|_\eta\cdot x)=y=\mu'(g'\cdot\eta')(g'|_{\eta'}\cdot x),
$$
it follows that $\mu(g\cdot \eta)=\mu' (g'\cdot \eta')$ and $g|_\eta\cdot x=g'|_{\eta'}\cdot x$. Finally, the equalities
$$
 \mathcal{T}_{d(\mu(g\cdot\eta))}([(g|_{\eta})|_x])=q= \mathcal{T}_{d(\mu'(g'\cdot\eta'))}([(g'|_{\eta'})|_x])
$$
show that $[(g|_{\eta})|_x]=[(g'|_{\eta'})|_x]$, so by Lemma~\ref{lem:cor57} we conclude that $g|_{\eta}=g'|_{\eta'}$.
\ep

\begin{remark}\label{rem:non-pseudo-free}
A version of the above lemma remains true without pseudo-freeness. Namely, the claim is that although  the equality $\nu\eta=\nu'\eta'=z(0,m)$ alone is not enough to guarantee all other properties of the paths $\eta$ and $\eta'$ in the statement of the lemma, there \emph{exist} paths satisfying them. The reason is that although at the end of the proof we cannot conclude that the equality $[(g|_{\eta})|_x]=[(g'|_{\eta'})|_x]$ implies that $g|_{\eta}=g'|_{\eta'}$, it is still true that $g|_{\eta x(0,n)}=g'|_{\eta'x(0,n)}$ for sufficiently large $n\le d(x)$, so that we get the required properties by appending $x(0,n)$ to $\eta$ and $\eta'$. \ee
\end{remark}

\begin{proof}[Proof of Proposition \ref{prop:topetale}]
It follows from Lemma \ref{lem:basis} that the sets $Z(\mu,g,\nu, F)$ constitute a basis of a topology on $\G_{(G,\Lambda)}$. We identify the unit space of $\G_{(G,\Lambda)}$ with $X_\Lambda$ as a set. Then the source map $s\colon \G_{(G,\Lambda)}\to\G_{(G,\Lambda)}^{(0)}=X_\Lambda$, $(y,q,n,z)\mapsto z$, maps every set
$Z(\mu ,g,\nu, F)$ bijectively onto~$Z_{F}(\nu)$. It follows that if we consider the standard topology on $X_\Lambda$, then the source map is open. By Lemma~\ref{lem:basis0} it is also continuous. Therefore it is a local homeomorphism. Since the space $X_\Lambda$ is locally compact, we conclude that the space $\G_{(G,\Lambda)}$ is locally compact as well (but we still have to show that it is Hausdorff).

To see that the inverse is continuous on $\G_{(G,\Lambda)}$, observe first that it maps $Z(\mu ,g,\nu )$ onto $Z(\nu, g^{-1}, \mu )$, see the proof of Proposition~\ref{prop:subgroupoid}. Since the source map is a homeomorphism of $Z(\mu ,g,\nu )$ and $Z(\nu, g^{-1}, \mu )$ onto $Z(\nu)$ and $Z(\mu)$, resp., it is then enough to show that the corresponding map $Z(\nu)\to Z(\mu)$, $\nu x\mapsto \mu (g\cdot x)$, is continuous. But this we know to be true, since the maps $\sigma^m$ are local homeomorphisms and $G$ acts on $X_\Lambda$ by homeomorphisms. It follows then that the range map $r\colon \G_{(G,\Lambda)}\to\G_{(G,\Lambda)}^{(0)}=X_\Lambda$ is a local homeomorphism and the sets $Z(\mu,g,\nu)$, hence also $Z(\mu,g,\nu, F)$, are open bisections of $\G_{(G,\Lambda)}$.

To see that the product on $\G_{(G,\Lambda)}$ is continuous, recall that any pair of composable elements lies in a set of the form $Z(\mu ,g,\nu )\times Z(\nu ,h,\alpha )$ and $Z(\mu ,g,\nu ) Z(\nu ,h,\alpha ) \subset Z(\mu ,g h,\alpha )$, see again the proof of Proposition~\ref{prop:subgroupoid}. Since the source map is continuous on $Z(\nu ,h,\alpha )$ and is a homeomorphism of $Z(\mu ,g h,\alpha )$ onto its image $Z(\alpha)$, the product is then obviously continuous.

The homomorphism $\Phi\colon \G_{(G,\Lambda)}\to\Z^k$ is constant on the sets $Z(\mu,g,\nu)$, so it is continuous. We have thus proved that $\G_{(G,\Lambda)}$ is a locally compact \'etale groupoid graded by the group $\Z^k$. It is also clear that if $(G,\Lambda)$ is countable, then the groupoid $\G_{(G,\Lambda)}$ is second countable. It remains to prove that $\G_{(G,\Lambda)}$ is Hausdorff.

Since the unit space $X_\Lambda$ is Hausdorff, we just need to be able to separate points that cannot be distinguished by $\Phi$ and have the same range and source. Let us consider two such points $(y,q,n,z)$ and $(y,q',n,z)$ in $\G_{(G,\Lambda)}$. Then $(y,q,n,z)\in Z(\mu,g,\nu)$ and $(y,q',n,z)\in Z(\mu',g',\nu)$ for some~$\mu$, $\mu'$, $\nu$, $g$ and $g'$. Since $d(\mu)-d(\nu)=n=d(\mu')-d(\nu)$, we have $d(\mu)=d(\mu')$ and hence $\mu=\mu'$. If $g=g'$, then $(y,q,n,z)=(y,q',n,z)$, while if $g\neq g'$, then $Z(\mu ,g ,\nu ) \cap Z(\mu ,g',\nu ) = \emptyset$ by Lemma~\ref{lem:separat}. This proves that the topology is Hausdorff.
\end{proof}

\begin{remark}
We also have a homomorphism $\G_{(G,\Lambda)}\to Q_k(G)\rtimes\Z^k$, $(y,q,n,z)\mapsto(q,n)$, but it is not continuous in general, at least if we consider the discrete topology on $Q_k(G)\rtimes\Z^k$, since the map $X_\Lambda\to Q_k(G)$, $x\mapsto [g|_x]$, is not necessarily locally constant for a fixed $g\in G$. \ee
\end{remark}

It is not difficult to see that a subset of $X_\Lambda$ is $\G_{(G,\Lambda)}$-invariant if and only if it is invariant under the action of $G$, the shifts $\sigma^m$ and their inverses. By Lemma \ref{lem:Ginvbound} we conclude that the boundary path space $\partial \Lambda\subset X_\Lambda$ is a closed invariant subset, so we can consider the reduced groupoid
$$
\G_{G,\Lambda}:= \mathcal{G}_{(G, \Lambda)}|_{\partial \Lambda} .
$$
We call it the \emph{Exel--Pardo groupoid} of a self-similar $k$-graph $(G,\Lambda)$. When $(G,\Lambda)$ is pseudo-free, this is a Hausdorff locally compact  \'etale groupoid. Without pseudo-freeness, the above considerations (together with Remark~\ref{rem:non-pseudo-free}) show that we still get a locally compact  \'etale groupoid, but it is not Hausdorff in general.

When the group is trivial, we get the groupoid $\G_\Lambda$ defining the Cuntz--Krieger algebra $\OO_\Lambda$ of~$\Lambda$.

\begin{remark}[cf.~\cite{MR3581326}*{Example~3.5}]\label{rem:CK-crossed}
When a self-similar $k$-graph is defined by an action by graph automorphisms, so that $g|_\mu=g$ for all~$\mu$, then the group $Q_k(G)\rtimes\Z^k$ in the definitions of $\G_{(G,\Lambda)}$ and $\G_{G,\Lambda}$ can be replaced by its subgroup $G\times\Z^k$, where the elements of $G$ are viewed as constant functions $\Z^k\to G$.
Moreover, in this case the group $G$ acts by automorphisms on the groupoid $\G_\Lambda$ and the map $G\ltimes\G_\Lambda\to \G_{G,\Lambda}$, $(g,(y,n,z))\mapsto(g\cdot y,g,n,z)$, is a topological isomorphism of groupoids. Therefore $C^*(\G_{G,\Lambda})\cong\OO_\Lambda\rtimes G$ and $C^*_r(\G_{G,\Lambda})\cong\OO_\Lambda\rtimes_r G$. \ee
\end{remark}

\begin{remark}\label{rem:isotropy}
The isotropy groups of $\G:=\G_{G,\Lambda}$ can be described as follows. Fix $x\in\partial\Lambda$. Denote by $G_x$ the stabilizer of $x$ in $G$. For $n\le m\le d(x)$, we have a homomorphism $G_{\sigma^n(x)}\to G_{\sigma^m(x)}$, $g\mapsto g|_{x(n,m)}$. This way we get an inductive system $(G_{\sigma^n(x)})_{n\le d(x)}$. Denote by $H_x$ its direct limit. Then we have a short exact sequence
$$
1\to H_x\to\Gxx\to\Phi(\Gxx)\to1,
$$
and $\Phi(\Gxx)\subset\Z^k$ consists of the elements $m-n$ such that $m,n\in\Z_+^k$ and there exists $g\in G$ such that $\sigma^m(x)=g\cdot\sigma^n(x)$.\ee
\end{remark}

\begin{remark}
Since the Deaconu--Renault groupoid C$^{*}$-algebra associated to the shift on $X_{\Lambda}$ gives rise to the so-called \emph{higher-rank Toeplitz graph C$^{*}$-algebra} and the Deaconu--Renault group\-oid C$^{*}$-algebra associated to the reduction to $\partial \Lambda$ gives rise to the higher-rank graph C$^{*}$-algebra, one could make an argument that the C$^{*}$-algebra of $\mathcal{G}_{(G, \Lambda)}$ is a candidate for a \emph{self-similar higher-rank Toeplitz graph C$^{*}$-algebra}. \ee
\end{remark}

\subsection{Amenability}
The goal of this subsection is to prove the following theorem.

\begin{thm}\label{thm:EP-amenability}
Assume $(G,\Lambda)$ is a pseudo-free self-similar $k$-graph. Consider the following conditions:
\begin{enumerate}
  \item the Exel--Pardo groupoid $\G_{G,\Lambda}$ is amenable;
  \item the action $G\curvearrowright\partial\Lambda$ is amenable;
  \item the stabilizer of every vertex for the action $G\curvearrowright\Lambda^0$ is amenable.
\end{enumerate}
Then $(3)\Rightarrow(2)\Rightarrow(1)$. If $(G,\Lambda)$ is countable, then $(1)\Rightarrow(2)$. If either $\Lambda$ is row-finite or $k=1$, then $(2)\Rightarrow(3)$, so if in addition $(G,\Lambda)$ is countable, then all three conditions are equivalent.
\end{thm}

It seems reasonable to expect that conditions (1)--(3) are equivalent for all pseudo-free self-similar $k$-graphs. For $k=1$ the implication (3)$\Rightarrow$(1) is essentially known, see~\cite{MR3581326}*{Corollary~10.18 and Remark~10.17}, \cite{MR3725509}*{Remark~3.9}. 
For $k=1$, under the assumption of countability, the equivalence of (1) and (2) follows from~\cite{MS}*{Theorem~2.18}. For row-finite $k$-graphs without sources it is also known that amenability of $G$ implies amenability of $\G_{G,\Lambda}$, see~\cite{MR4283280}*{Proposition~3.10}.

\smallskip

For the proof of the theorem we will need the following property.

\begin{lemma}\label{lem:trans-embedding}
For any pseudo-free self-similar $k$-graph $(G,\Lambda)$, the map $G\ltimes\partial\Lambda\to \G_{G,\Lambda}$, $(g,x)\mapsto (g\cdot x,[g|_{x}],0,x)$, is a topological isomorphism of $G\ltimes\partial\Lambda$ onto a clopen subgroupoid of $\G_{G,\Lambda}$.
\end{lemma}

\bp
It is immediate that the map is a homomorphism of groupoids. It is injective by Lemma~\ref{lem:cor57}. To see that it is a homeomorphism onto an open set, recall that a basis of topology on $G\ltimes\partial\Lambda$ is given by the sets $\{g\}\times (Z_F(\lambda)\cap\partial\Lambda)$. Every such set is mapped bijectively onto the open bisection  $Z(g\cdot\lambda,g|_\lambda,\lambda,F)\cap\G_{G,\Lambda}$ of $\G_{G,\Lambda}$. Moreover, we can write the map $\{g\}\times (Z_F(\lambda)\cap\partial\Lambda)\to\G_{G,\Lambda}$ as the map $(g,x)\mapsto x\in Z_F(\lambda)\cap\partial\Lambda$ followed by the homeomorphism $Z_F(\lambda)\cap\partial\Lambda\to Z(g\cdot\lambda,g|_\lambda,\lambda,F)\cap\G_{G,\Lambda}$ that is the inverse of the source map, which makes it clear that we get a homeomorphism of $\{g\}\times (Z_F(\lambda)\cap\partial\Lambda)$ onto its open image.

It remains to show that the image of $G\ltimes\partial\Lambda$ is closed. Assume a net of elements $(g_i\cdot x_i,[g_i|_{x_i}],0,x_i)$ converges to a point in $\G_{G,\Lambda}$. This point must lie in the kernel of~$\Phi$, that is, to be of the form $(\mu(h\cdot y),\TT_{d(\mu)}([h|_y]),0,\nu y)$ for some $\mu,\nu\in\Lambda$ of the same degree. For all $i$ large enough, we can write $x_i=\nu y_i$ for some $y_i$. Then
$$
(g_i\cdot x_i,[g_i|_{x_i}],0,x_i)=((g_i\cdot\nu)(g_i|_\nu\cdot y_i),\TT_{d(\mu)}([(g_i|_\nu)|_{y_i}]),0,\nu y_i).
$$
Hence we must have $g_i\cdot\nu=\mu$ for all $i$ large enough. By assumption, the above expressions eventually lie in $Z(\mu,h,\nu)$. It follows that for all $i$ large enough we have
$$
g_i|_\nu\cdot y_i=h\cdot y_i\quad\text{and}\quad [(g_i|_\nu)|_{y_i}]=[h|_{y_i}].
$$
By Lemma~\ref{lem:cor57} it follows that $g_i|_\nu=h$. Therefore for all $i,j$ large enough we get
$$
g_i\cdot\nu=\mu=g_j\cdot\nu\quad\text{and}\quad g_i|_\nu=g_j|_\nu.
$$
Applying Lemma~\ref{lem:cor57} again we conclude that eventually the elements $g_i$ are equal to one element $g\in G$. Therefore $(g_i,x_i)\to(g,x)$ in $G\ltimes\partial\Lambda$ and hence $(\mu(h\cdot y),\TT_{d(\mu)}([h|_y]),0,\nu y)$ coincides with the image of $(g,x)$.
\ep

In the second countable case amenability of groupoids passes to closed subgroupoids \cite{MR1799683}*{Proposition~5.1.1}. Therefore the above lemma proves the implication (1)$\Rightarrow$(2) for countable $(G,\Lambda)$.

\bp[Proof of $(2)\Rightarrow(1)$ in Theorem~\ref{thm:EP-amenability}]
The overall strategy of the proof is by now standard, cf. \cite{RW}*{Theorem~5.13}.
Denote by $\HH$ the kernel of $\Phi\colon\G_{G,\Lambda}\to\Z^k$. It is an open subgroupoid of~$\G_{G,\Lambda}$ consisting of points of the form
$
(\mu (g\cdot x) , \mathcal{T}_{d(\mu)}([g|_{x}]), 0, \nu x )
$
with $d(\mu)=d(\nu)$. For every $n\in\Z^k_+$, denote by $\HH^n\subset\HH$ the subset of points of the form $(\mu (g\cdot x) , \mathcal{T}_{d(\mu)}([g|_{x}]), 0, \nu x )$ with $d(\mu)=d(\nu)\le n$. Let also denote by $U_n\subset\partial\Lambda$ the domain of definition of $\sigma^n|_{\partial\Lambda}$, that is, the open subset of boundary paths of degree $\ge n$. One can then easily check the following properties:
\begin{itemize}
  \item every set $\HH^n$ is an open subgroupoid of $\HH$;
  \item $\HH^n\subset\HH^m$ if $n\le m$, and $\bigcup_{n\in\Z^k_+}\HH^n=\HH$;
  \item every set $U_n$ is $\HH$-invariant, hence it is $\HH^m$-invariant for any $m$.
\end{itemize}

By Proposition~\ref{prop:graded-amen}, in order to show that $\G_{G,\Lambda}$ is amenable it suffices to show that~$\HH$ is amenable. By Lemma~\ref{lem:amen-sequence} for this, in turn, it suffices to show that every groupoid~$\HH^n$ is amenable.

Let us first show that each reduced groupoid $\HH^n_{U_n}$ is amenable. This groupoid consists of points of the form $(\mu (g\cdot x) , \mathcal{T}_n([g|_{x}]), 0, \nu x )$ with $d(\mu)=d(\nu)=n$. Let $(f_i)_i$ be an approximate invariant density on $G\ltimes\partial\Lambda$. For every $v\in s(\Lambda^n)\subset\Lambda^0$ fix a path $\mu_v\in\Lambda^nv$. Define a function~$\tilde f_i$ on $\HH^n_{U_n}$ by
$$
\tilde f_i(\mu (g\cdot x) , \mathcal{T}_n([g|_{x}]), 0, \nu x ):=\begin{cases}
                                                                         f_i(g,x), & \mbox{if } \mu=\mu_{r(g\cdot x)}, \\
                                                                         0, & \mbox{otherwise},
                                                                       \end{cases}
$$
where $d(\mu)=d(\nu)=n$. Lemma~\ref{lem:trans-embedding} implies that this function is well-defined and continuous. For every $\nu x\in U_n$, $d(\nu)=n$, we then have
$$
\sum_{\gamma\in\HH^n_{\nu x}}\tilde f_i(\gamma)=\sum_{g\in G} f_i(g,x)=1,
$$
and if $x=h\cdot y$ and $\gamma'=(\nu(h\cdot y),\TT_n([h|_{y}]),0,\lambda y)$, $d(\lambda)=n$, then
$$
\sum_{\gamma\in\HH^n_{\nu x}}|\tilde f_i(\gamma\gamma')-\tilde f_i(\gamma)|=\sum_{g\in G} |f_i(gh,y)-f_i(g,x)|\xrightarrow[i\to\infty]{}0.
$$
This shows that $\HH^n_{U_n}$ is indeed amenable.

Now, for a fixed $n\in\Z^k_+$, consider the algebra of sets generated by $U_m$ for $m\le n$. Let $D_1,\dots,D_l$ be its atoms. By Lemma~\ref{lem:amen-decompose} in order to finish the proof of amenability of $\HH^n$ it suffices to show that the groupoid $\HH^n_{D_i}$ is amenable for every $i$. For a fixed $i$, let $m$ be the largest element dominated by $n$ such that $D_i\subset U_{m}$. Note that it exists, because $U_p\cap U_q=U_{p\vee q}$. Then $\HH^n_{D_i}=\HH^{m}_{D_i}$, hence amenability of $\HH^n_{D_i}$ follows from that of $\HH^m_{U_m}$.
\ep

\bp[Proof of $(3)\Rightarrow(2)$ in Theorem~\ref{thm:EP-amenability}]
Amenability of the stabilizers of vertices is equivalent to amenability of the action of $G$ on the discrete space $\Lambda^0$. The $G$-equivariant map $r\colon\partial\Lambda\to\Lambda^0$ allows us to pull an approximate invariant density on $G\ltimes\Lambda^0$ to that on $G\ltimes\partial\Lambda$, so the action $G\curvearrowright\partial\Lambda$ is amenable.
\ep

It remains to prove that under additional assumptions (2) implies  (3).  For this we will need the following simple fact.

\begin{lemma}\label{lem:inv-state}
Assume a discrete group $\Gamma$ acts on a unital AF-algebra $A=\overline{\bigcup^\infty_{n=1}A_n}$, where $A_1\subset A_2\subset\dots$ are unital finite dimensional C$^*$-subalgebras of $A$ such that each algebra $A_n$ is $\Gamma$-invariant. Then there exists a $\Gamma$-invariant tracial state on $A$.
\end{lemma}

\bp
Denote by $S_n$ the set of $\Gamma$-invariant tracial states on $A_n$. It is nonempty, since it contains, for example, the normalized canonical trace. Denote by $\tilde S_n$ the set of states $\varphi$ on $A$ such that $\varphi|_{A_n}\in S_n$. Then the sets $\tilde S_n$ are nonempty, weakly$^*$ closed and $\tilde S_1\supset\tilde S_2\supset\dots$. It follows that the set $\bigcap^\infty_{n=1}\tilde S_n$, which is exactly the set of $\Gamma$-invariant tracial states, is nonempty.
\ep

\bp[Proof of $(2)\Rightarrow(3)$ in Theorem~\ref{thm:EP-amenability}]
We assume that the action $G\curvearrowright\partial\Lambda$ is amenable and want to deduce from this that the stabilizer of every vertex is amenable when either $\Lambda$ is row-finite or $k=1$.

\smallskip

Assume first that $\Lambda$ is row-finite. Fix a vertex $v\in\Lambda^0$ and consider its stabilizer $\Gamma:=G_v$. We claim that the action $\Gamma\curvearrowright\partial\Lambda$ is amenable. In the second countable case this already follows from the fact that $\Gamma\ltimes\partial\Lambda$ is a closed subgroupoid of $G\ltimes\partial\Lambda$, but the claim is true without any countability assumptions for the same reasons as in the case of groups. Namely, if $(f_i)_i$ is an approximate invariant density on $G\ltimes\partial\Lambda$ and $g_j\in G$ are representatives of $G/\Gamma$, then the functions
$$
\tilde f_i(h,x):=\sum_jf_i(g_jh,x)
$$
form an approximate invariant density on $\Gamma\ltimes\partial\Lambda$.

By~\cite{MR2301938}*{Proposition~4.3}, the set $\Omega:=v\partial\Lambda$ is nonempty. Since the set $vX_\Lambda=Z(v)$ is compact by row-finiteness of $\Lambda$, the set $\Omega=Z(v)\cap\partial\Lambda$ is compact as well. This set is also $\Gamma$-invariant, so we get an amenable action $\Gamma\curvearrowright\Omega$.

For $n\in\N$, put $e(n):=(n,\dots,n)\in\Z^k_+$. It is not difficult to see that row-finiteness of $\Lambda$ implies that the map $X_\Lambda\to\Z^k_+$, $x\mapsto d(x)\wedge e(n)$, is locally constant. Denote by $\Lambda^{\le e(n)}$ the set of all paths in $\Lambda$ of degree $\le e(n)$. The set $v\Lambda^{\le e(n)}$ is finite, and we can define a continuous map
$$
\pi_n\colon \Omega\to v\Lambda^{\le e(n)},\quad \pi_n(x):=x(0,d(x)\wedge e(n)).
$$
Let $A_n\subset C(\Omega)$ be the subalgebra of functions of the form $f\circ\pi_n$, with $f\in C(v\Lambda^{\le e(n)})$. The C$^*$-algebras $A_n$ are unital, finite dimensional and $\Gamma$-invariant, $A_1\subset A_2\subset\dots$, and $\cup_nA_n$ is dense in $C(\Omega)$, since the maps $\pi_n$ separate points of $\Omega$. By Lemma~\ref{lem:inv-state} we conclude that there is a $\Gamma$-invariant state on $C(\Omega)$. By Lemma~\ref{lem:inv-measure-amenable} it follows then that $\Gamma$ is amenable.

\smallskip

Assume now that $k=1$. Fix again a vertex $v\in\Lambda^0$ and consider its stabilizer $\Gamma:=G_v$. If the sets $v\Lambda^n$ are finite for all $n\ge1$, then we effectively deal with the row-finite case, and the same arguments as above show that $\Gamma$ is amenable. Assume therefore that some sets $v\Lambda^n$ are infinite, and let $n\ge1$ be the smallest number with this property. Then the finite set $v\Lambda^{n-1}$ contains a path $\mu$ such that the set $s(\mu)\Lambda^1$ is infinite. As was remarked at the end of Section~\ref{ssec:higher-rank}, then $s(\mu)\in\partial\Lambda$. Amenability of the action $G\curvearrowright\partial\Lambda$ implies then that the stabilizer $G_{s(\mu)}$ of $s(\mu)$ is amenable. It follows that the stabilizer $G_\mu\subset \Gamma\cap G_{s(\mu)}$ of the path $\mu$ is amenable as well. But since the set $v\Lambda^{n-1}$ is finite, the stabilizer $\Gamma_\mu=G_\mu$ of $\mu$ for the action of $\Gamma$ on this set is a finite index subgroup of $\Gamma$, hence $\Gamma$ is also amenable.
\ep

\bigskip

\section{Self-similar directed graphs}\label{sec:1-graphs}
Throughout this section we fix a countable self-similar $1$-graph $(G,E)$. Recall that then $E$ is an ordinary directed graph and that the standing assumption of finite alignment is automatically satisfied for such graphs. In order to be more consistent with the literature on graph algebras, we will slightly change the notation used in the previous section. Namely, we write $E^*$, rather than simply $E$, for the set of finite paths. A path $x\in X_E$ of length $|x|=d(x)\ge1$ is often written as $x_1x_2\dots$, with $x_n\in E^1$, so $x_n=x(n-1,n)$ in the earlier notation. We write $\sigma$ for the shift $\sigma^1\colon X_E\setminus E^0\to X_E$.

\subsection{The quasi-orbit space of a self-similar graph}\label{ssec:quasio}
Consider the Exel--Pardo groupoid $\G:=\G_{G,E}$. In this subsection only the topology on $\partial E=\Gu$ will matter to us, so we do not require $(G,E)$ to be pseudo-free and do not consider any topology on the entire groupoid~$\G$. Our goal is to obtain a description of the quasi-orbit space $(\G\backslash\partial E)^\sim$ that generalizes and refines that given in \cite{CN3}*{Section~4.1}  for the groupoid~$\G_E$ underlying the Cuntz--Krieger algebra~$\OO_E$. We remark that since the $\G$-saturation $r(\G_U)$ of every open subset $U\subset\partial E$ is open, the discussion at the beginning of Section~\ref{ssec:quasi-orbits} still applies and  the quasi-orbit space indeed coincides with the $T_0$-ization of the orbit space $\G\backslash\partial E$, see~\cite{CN3}*{Section~1.4}.

\smallskip

Recall that there is a preorder on $E^{0}$ given by $v \geq w$ if and only if $v E^{*} w \neq \emptyset$. It defines an equivalence relation on $E^{0}$ such that $v \sim w$ if and only if $v \leq w$ and $w \leq v$. 
The group $G$ respects the preorder in the sense that for any $g\in G $ we have $v \geq w$ if and only if $g\cdot v \geq g \cdot w$.

Recall that a vertex $v$ is called \emph{regular} if $0<|vE^{1}|<\infty$, and it is called \emph{singular} otherwise. There are two kinds of singular vertices -- the sources ($|vE^{1}|=0$) and the infinite receivers ($|vE^{1}|=\infty$). We remind that $\partial E$ is the union of the set $E^\infty$ of infinite paths and the set of finite paths $\alpha\in E^*$ such that $s(\alpha)\in E^\sing$.

The notion of maximal tails~\cite{MR1988256} is important for understanding the ideal structure of~$\OO_E$. For self-similar graphs we introduce a similar notion that takes into account the group action.

\begin{defn}\label{def:max-G-tail}
A nonempty $G$-invariant subset $M \subset E^{0}$ is called a \emph{maximal $G$-tail} if the following conditions are satisfied:
\begin{enumerate}
\item[(i)] if $v\in E^{0}$ and $v \geq w$ for some $w\in M$, then $v\in M$;
\item[(ii)] if $v \in M$ is a regular vertex, then $vE^1M\ne\emptyset$;
\item[(iii)] for every $v,w \in M$, there are elements $g\in G$ and $u\in M$ such that $v \geq u$ and $w \geq g\cdot u$.
\end{enumerate}
We denote by $\M(G,E)$  the set of  maximal $G$-tails in $(G,E)$.
\end{defn}

\begin{remark}
The definition of a maximal $G$-tail is identical with the definition of a maximal tail in the case where the $G$-action is trivial. Observe that any self-similar action $G\curvearrowright E$, being restricted to the vertices and edges of $E$, defines a quotient graph $G\backslash E$. Then it is not difficult to see that if $E$ is row-finite and has no sources, so that all vertices are regular, then the quotient map $E^0\to G\backslash E^0$ establishes a one-to-one correspondence between the maximal $G$-tails in $(G,E)$ and maximal tails in $G\backslash E$. 
\ee
\end{remark}

Define for each $x\in \partial E$ a subset $\MT_G(x)$ of $E^{0}$ by
$$
\MT_G(x):= \{ r(y) : y \in  [ x]_{\mathcal{G}} \},
$$
where $[ x]_{\G}:=r(\G_x)$ is the orbit of $x$ under the action of $\mathcal{G}$ on $\partial E$. We prefer to keep the subscript~$\G$ to make it clear that we are talking about $\G$-orbits, not $G$-orbits or $\sigma$-orbits.

The $\G$-orbit of $x$ is the smallest set containing $x$ that is invariant under the action of $G$, the shift~$\sigma$ and its inverse. As a consequence, it can be easily seen that it has the following description.

\begin{lemma}\label{lem:EP-orbit}
For every $x\in\partial E$, its $\G$-orbit $[x]_\G$ consists of all paths of the form $\mu(g\cdot\sigma^n(x))$, with $n\ge0$ such that $x\in\operatorname{dom}(\sigma^n)$, $g\in G$ and $\mu\in E^*(g\cdot s(x_n))$.
\end{lemma}

Observe also that given a point $\mu(g\cdot\sigma^n(x))\in[x]_\G$ such that $x\in\operatorname{dom}(\sigma^{n+k})$ for some $k\ge1$, we can write
\begin{equation}\label{eq:obs-orbit}
\mu (g \cdot \sigma^{n}(x)) = \mu (g \cdot (x_{n+1}\cdots x_{n+k} )) ( g|_{x_{n+1}\cdots x_{n+k}} \cdot \sigma^{n+k}(x)).
\end{equation}

The relevance of maximal $G$-tails for the quasi-orbit structure can be seen from the following lemma.

\begin{lemma}\label{lem:Gtails}
For every $x\in\partial E$,  the set $\MT_G(x)$ is the smallest maximal $G$-tail containing all vertices of $x$. If $\overline{[x]}_\G=\overline{[y]}_\G$ for some $x,y\in\partial E$, then $\MT_G(x)=\MT_G(y)$. If $x,y\in E^\infty$, then conversely, the equality $\MT_G(x)=\MT_G(y)$ implies that $\overline{[x]}_\G=\overline{[y]}_\G$.
\end{lemma}


\begin{proof}
The set $\MT_G(x)$ is obviously $G$-invariant. If $y\in[x]_\G$ and $s(\mu)=r(y)$, then $\mu y\in[x]_\G$ and therefore $r(\mu)\in\MT_G(x)$. This proves property (i) in Definition~\ref{def:max-G-tail}. Assume $v=r(y)$ for some $y\in[x]_\G$ and $0<|vE^{1}|<\infty$. Since $v$ is a regular vertex, the path $y$ must have length at least~$1$, so $y_{1} \in vE^{1} $ and $s(y_{1}) = r(\sigma(y)) \in \MT_G(x)$, proving (ii). For (iii), take $v=r(y)$ and $w = r(z)$ for some $y, z \in  [ x]_{\mathcal{G}}$. We can write $y=\nu (g \cdot \sigma^{n}(x))$ and $z=\mu (h \cdot \sigma^{m}(x))$. By observation~\eqref{eq:obs-orbit} we may assume that $n=m$. Set $u := r(g\cdot \sigma^{n}(x)) \in \MT_G(x)$. Then $v=r(\nu)$ and $u=s(\nu)$, so $v\ge u$. Since $hg^{-1} \cdot u = h g^{-1} \cdot  r(g\cdot \sigma^{n}(x)) = r(h\cdot \sigma^{n}(x))=s(\mu)$, we also get that $w\ge hg^{-1} \cdot u$. This proves that $\MT_G(x)$ is a maximal $G$-tail.

By definition and Lemma~\ref{lem:EP-orbit}, the set $\MT_G(x)$ consists of all vertices~$v$ such that $v\ge g\cdot w$ for some vertex $w$ of $x$. It follows that $\MT_G(x)$ is the smallest maximal $G$-tail containing all vertices of $x$.

Since the map $r\colon\partial E\to E^0$ is continuous, we have $r(y)\in\MT_G(x)$ for all $y\in \overline{[x]}_\G$. Hence $\MT_G(x)=\MT_G(y)$ whenever $\overline{[x]}_\G=\overline{[y]}_\G$.

Assume now that $x,y\in E^\infty$ and $\MT_G(x)=\MT_G(y)$. If we write $x=\alpha x'$ for some $\alpha\in E^*$ and $x'\in E^\infty$, then $r(x')=r(y')$ for some $y'\in[y]_\G$, hence $\alpha y'\in[y]_\G$ and therefore $Z(\alpha)\cap[y]_\G\ne\emptyset$. This implies that $x\in\overline{[y]}_\G$. By symmetry we also have $y\in\overline{[x]}_\G$.
\end{proof}

To get a full description of quasi-orbits we need to understand the image of $\MT_G$ and the orbit closures of finite paths in~$\partial E$.

\smallskip

Consider the canonical cocycle $\Phi\colon\G\to\Z$ and define the set of \emph{$G$-aperiodic boundary paths} by
\begin{equation}\label{eq:AGE}
A_{G,E} := \{ x \in \partial E : \Phi (\Gxx)=0\}.
\end{equation}
In other words, $A_{G,E}$ consists of all boundary paths $x$ such that if $\sigma^n(x)=g\cdot\sigma^m(x)$ for some $g\in G$ and $n,m\in\Z_+$, then $n=m$. When the action of $G$ is trivial, this is the set of aperiodic points for the shift $\sigma$. Since the group $\Phi(\Gxx)$ depends only on the $\G$-orbit of $x$, it is clear that the set $A_{G,E}$ is $\G$-invariant.

Let us look closer at the set $\partial E\setminus A_{G,E}$. So assume we have $\sigma^n(x)=g\cdot\sigma^m(x)$ for some $n\ne m$. By acting by $g^{-1}$ if necessary, we may assume that $n>m$. Then $\sigma^{n-m}(\sigma^m(x))=g\cdot\sigma^m(x)$. The paths $y$ satisfying $\sigma^n(y)=g\cdot y$ for some $n\ge1$ have been analyzed in~\cite{MR3581326}*{Section~14}. Let us recall a few notions and results from there.

By a \emph{$G$-circuit} one means a pair $(g,\gamma)$, where $g\in G$ and $\gamma \in E^{*}$ is a finite path of length $|\gamma|>0$ such that $s(\gamma)=g\cdot r(\gamma)$. The following lemma is a reformulation of Propositions~14.2 and~14.3 in~\cite{MR3581326}.

\begin{lemma}\label{lem:G-circuits}
Given a $G$-circuit $(g,\gamma)$ with $|\gamma|=n$, there exists a unique path $x\in E^\infty$ such that $\sigma^n(x)=g\cdot x$ and $x_1\dots x_n=\gamma$. Conversely, if we are given a path $x$ such that $\sigma^n(x)=g\cdot x$ for some $n\ge1$ and $g\in G$, then $x\in E^\infty$ and the pair $(g,x_1\dots x_n)$ is a $G$-circuit.
\end{lemma}

Let us recall how the path $x$ associated with a $G$-circuit $(g,\gamma)$ is constructed. Define inductively a sequence of finite paths $\gamma^{(n)}\in E^{*}$ and a sequence of group elements $g_{n}\in G$ by setting $\gamma^{(1)}:=\gamma$, $g_{1}:=g$, $\gamma^{(n+1)}:= g_{n} \cdot \gamma^{(n)}$ and $g_{n+1}:=g_{n}|_{\gamma^{(n)}}$. Then $x=\gamma^{(1)}\gamma^{(2)}\dots$.

Since we assumed that the self-similar graph $(G,E)$ is countable and hence there can only be countably many finite paths and $G$-circuits, the above discussion leads to the following conclusion.

\begin{lemma} \label{lem:uncountable}
The set $\partial E\setminus A_{G,E}$ is countable and consists of infinite paths.
\end{lemma}

Given a maximal $G$-tail and an infinite path $x\in E^\infty$ with vertices in $M$, we say that $x$ \emph{has no entry in $M$} if $r(x_n)E^1M=\{x_n\}$ for all $n\ge1$. Note that if $x$ has no entry in $M$, then either all vertices $r(x_n)$ are different or $x=\mu\alpha\alpha\dots$ for some $\mu\in E^*$ and a simple cycle $\alpha=e_1\dots e_p$, i.e., $r(e_{1})=s(e_{p})$ and $r(e_{i})\neq r(e_{j})$ for $i\neq j$.

We are now ready to prove surjectivity of $\MT_G$. In fact, we will prove a more precise result, which will be important for a number of subsequent arguments.

\begin{prop} \label{prop:Gmaxt}
Assume $(G,E)$ is a countable self-similar directed graph and $M\subset E^0$ is a maximal $G$-tail. Then there exists $x\in\partial E$ such that $\MT_G(x)=M$. Moreover, such a path $x$ can be chosen so that either $x\in A_{G,E}$ or $x\in E^{\infty} \setminus A_{G,E}$ is a path with no entry in $M$.
\end{prop}

\begin{proof}
We say that an element $v\in M$ is \emph{minimal} if for any $w\in M$ with $w \leq v$ we have $w \sim v$. We first consider the case where there exists a minimal element $v\in M$. It follows then from properties (iii) and (i) of maximal $G$-tails that $M$ consists of all vertices $w$ such that $w\ge g\cdot v$ for some $g\in G$. In particular, any path $y\in\partial E$ with vertices in $M$ and $v\in\MT_G(y)$ satisfies $\MT_G(y)=M$. We now consider several subcases.

If $vE^{*}v=\{v\}$, then $v$ must be a singular vertex, because if it was regular, the second condition in the definition of a maximal $G$-tail would give us an edge $e$ with $s(e) \in M$ and $v \geq s(e)$, hence $s(e)\sim v$ by minimality of $v$ and we would be able to construct an element of $vE^{*}v$ different from $v$. So if $vE^{*}v=\{v\}$, we set $x:=v$ and remark that $x\in A_{G,E}$.

If $vE^{*}v$ consists of powers of one simple cycle $\alpha$, we set $x:=\alpha\alpha\dots$. The path $x$ has no entry in $M$, since if $r(x_n)E^1M$ contained an edge $e\ne x_n$, then $s(e)\sim v$ by minimality of $v$ and we would be able to construct a path in $vE^*v$ different from a power of $\alpha$.

If $vE^{*}v$ contains a simple cycle $\alpha$ and another, not necessarily simple, cycle $\beta$ that is not a power of $\alpha$, then the collection of all infinite paths obtained by concatenating the cycles $\alpha^{|\beta|}$ and $\beta^{|\alpha|}$ is uncountable, hence it contains an element $x \in A_{G, E}$ by Lemma \ref{lem:uncountable}.

\smallskip

Next we consider the case where there are no minimal elements in $M$. Enumerate the elements of~$M$ as $u_{1}, u_{2}, \dots$. We can inductively construct pairwise inequivalent vertices $v_{1} \geq v_{2}\geq \dots $ in~$M$ such that for each $n$ there is an element $g\in G$ (depending on $n$) satisfying $u_{n} \geq g\cdot v_{n}$. Namely, for $n=1$ we take $v_1:=u_1$. Once $v_1,\dots,v_n$ are constructed, take any vertex $v$ in~$M$ such that $v\le v_n$ and $v\not\sim v_n$. Then we can find $v_{n+1}\in M$ and $g\in G$ such that $v\ge v_{n+1}$ and $u_{n+1}\ge g\cdot v_{n+1}$.

By concatenating paths from $v_{n+1}$ to $v_n$ we can construct an infinite path $y$ passing through all vertices $v_n$. In particular, we have $|v_{n}E^{m}M|\ge1$ for all $n,m\geq 1$. 
Consider two subcases.

\smallskip

\underline{Case 1}. \emph{There exists $n\ge 1$ such that $|v_{n}E^{m}M|=1$ for all $m\geq 1$.}

\smallskip

In this case there is a unique path $x\in E^\infty$ with vertices in $M$ such that $r(x)=v_n$. Since it must pass through all vertices $v_k$ with $k\ge n$, we have $M=\MT_G(x)$. It is also clear that $x$ has no entry in $M$.

\smallskip

\underline{Case 2}. \emph{For each $n\ge1$, we have $|v_{n}E^{m}M|>1$ for some $m\geq 1$ (depending on $n$).}

\smallskip

In this case consider the set $\Omega$ of all infinite paths $y$ with vertices in $M$ satisfying the following property: for every $n\ge1$, there exist $j\ge n$ and $k\ge1$ such $r(y_k)\in G\cdot v_j$. For any such path~$y$ we have $v_n\in\MT_G(y)$ for all $n$, hence $\MT_G(y)=M$. We claim that the set $\Omega$ is uncountable, hence we can find $x\in\Omega\cap A_{G,E}$ by Lemma~\ref{lem:uncountable}.

In order to prove the claim we will show how to embed $\{0,1\}^\N$ into $\Omega$. Observe first that for any $v\in M$ and $n\ge1$ there exist $j\ge n$ and different paths $y,z\in vE^mM$ for some $m\ge1$ such that $r(y_k),r(z_k)\in G\cdot v_j$ for some $k\le m$. Indeed, we can find $j\ge n$ and $g\in G$ such that $v\ge g\cdot v_j$. Take any $\lambda\in v E^*(g\cdot v_j)$. Since $|(g\cdot v_{j})E^{l}M|=|v_{j}E^{l}M|>1$ for some $l$, we can find different paths $\mu,\nu\in (g\cdot v_{j})E^{l}M$. Then $y:=\lambda\mu$ and $z:=\lambda\nu$ have the required properties.

We now start with any vertex $v\in M$ and apply the observation to $n=1$ to get two different paths $y,z\in vE^mM$. Then we apply the observation to $n=2$ and the vertices $s(y)$ and $s(z)$ to get different paths $y',z'\in s(y)E^pM$ and different paths $y'',z''\in s(z)E^qM$. Then we apply the observation to $n=3$ and the sources of four different paths $yy'$, $yz'$, $zy''$, $zz''$. By continuing this procedure we can construct a family of different paths in $\Omega$ indexed by the elements of $\{0,1\}^\N$.  This finishes the proof of the proposition.
\end{proof}

For later use it is convenient to record some cases considered in the above proof as separate statements.

\begin{cor}\label{cor:Gmaxt}
For any $M\in\M(G,E)$, the following properties hold:
\begin{enumerate}
  \item if there is a minimal element $v\in M$ such that $vE^*v=\{v\}$, then $v\in E^\sing$ and $M=\MT_G(v)$;
  \item if there is a minimal element $v\in M$ such that the set $v E^*v$ is neither $\{v\}$ nor $\{\alpha^n\}_{n\ge0}$ for a simple cycle $\alpha$, then $M=\MT_G(x)$ for some $x\in E^\infty\cap A_{G,E}$;
  \item if there are no minimal elements in $M$, then $M=\MT_G(x)$ for some $x\in E^\infty$ such that either $x\in A_{G,E}$ or $x\notin A_{G,E}$ is a path with no entry in $M$.
\end{enumerate}
\end{cor}


\begin{remark}\label{rem:G-tails}
If $M=\MT_G(x)$ and we consider the maximal tail $\MT(x)\subset E^0$ of $x$ (that is, the maximal $H$-tail $\MT_H(x)$ for the trivial group $H$), then since $M$ consists of all vertices~$v$ such that $v\ge g\cdot w$ for some $g\in G$ and a vertex $w$ of $x$, while $\MT(x)$ consists of all vertices~$v$ such that $v\ge w$ for a vertex $w$ of $x$, we can conclude that $M=G\cdot\MT(x)$. Therefore the maximal $G$-tails are simply the $G$-saturations of the maximal tails. \ee
\end{remark}

Our next goal is to understand when we can choose $x\notin A_{G,E}$ with no entry in $M$. We need some preparation.

\begin{lemma}\label{lem:no-entry-paths}
Assume $M$ is a maximal $G$-tail and $x\in E^\infty$ is a path with vertices in $M$ and no entry in $M$. Then $M=\MT_G(x)$ and the orbit $[x]_\G$ is discrete.
\end{lemma}

\bp
Since $\MT_G(x)$ is the smallest maximal $G$-tail containing all vertices of $x$, we have $\MT_G(x)\subset M$. To prove the equality, take any $v\in M$. Then there exist $g\in G$ and $u\in M$ such that $v\ge g\cdot u$ and $r(x_1)\ge u$. Take any path $\mu$ with $s(\mu)=u$ and $r(\mu)=r(x_1)$. Then all vertices of $\mu$ lie in~$M$ and by the assumption that $x$ has no entry in $M$ we conclude that $\mu=x_1\dots x_{|\mu|}$, so that $u=s(\mu)$ is a vertex of~$x$. Hence $v\in\MT_G(x)$. For similar reasons $x$ is the only infinite path with vertices in $M$ that lies in $Z(r(x_1))$, hence $x$ is an isolated point of $[x]_\G$.
\ep

\begin{cor}\label{cor:unique-orbit}
If $y\in E^\infty$ is another path with the same properties as $x$, then $y\in [x]_\G$.
\end{cor}

\bp
Since $\MT_G(y)=M=\MT_G(x)$, by Lemma~\ref{lem:Gtails} we conclude that $\overline{[y]}_\G=\overline{[x]}_\G$. Since $[x]_\G$ is discrete, it follows that $y\in[x]_\G$.
\ep

We denote by $\M_\gamma(G,E)$ the set of maximal $G$-tails $M$ such that there is no path $x\in E^\infty\setminus A_{G,E}$ with vertices in $M$ and no entry in $M$. We let
\begin{equation}\label{eq:LGE}
\LL(G,E):=\M(G,E)\setminus\M_\gamma(G,E).
\end{equation}
For every $L\in \LL(G,E)$ there is therefore a path $x\in E^\infty\setminus A_{G,E}$ with vertices in $L$ and no entry in $L$. By Lemma~\ref{lem:no-entry-paths} and Corollary~\ref{cor:unique-orbit}, then $\MT_G(x)=L$ and the $\G$-orbit of $x$ is uniquely determined by~$L$. Therefore we get a well-defined number $n_L\ge1$ such that
$\Phi(\Gxx)=n_L\Z$ for any such path $x$. For convenience, from now on we fix such a path~$x$ and denote it by~$x_L$.

\begin{defn}\label{def:LGE}
For $L\in \LL(G,E)$, we call the number $n_L\ge1$ defined by $\Phi(\G^{x_L}_{x_L})=n_L\Z$ the \emph{$G$-period} of $L$.
\end{defn}

Since elements of $E^\infty\setminus A_{G,E}$ arise from $G$-circuits, one can determine whether a maximal $G$-tail lies in $\LL(G,E)$ in terms of them, without using infinite paths. Namely, let us say that a $G$-circuit $(g,\gamma)$ with vertices in a maximal $G$-tail $M$ \emph{has no entry in $M$} if $|vE^1M|=1$ for every vertex $v$ of $\gamma$. In other words, letting $n:=|\gamma|$, we require $r(\gamma_k)E^1M=\{\gamma_k\}$ for $1\le k\le n$ and $s(\gamma_n)E^1M=\{g\cdot\gamma_1\}$. We call $|\gamma|$ the length of the $G$-circuit $(g,\gamma)$.

\begin{lemma}[cf.~\cite{MR3581326}*{Proposition~14.5}]\label{lem:circuits-wo-entry}
A maximal $G$-tail $M$ lies in $\LL(G,E)$ if and only if there is a $G$-circuit $(g,\gamma)$ with vertices in $M$ and no entry in $M$. Moreover, if $M\in\LL(G,E)$, then $n_M$ is the length of a shortest such $G$-circuit, and the length of any other such $G$-circuit is a multiple of $n_M$.
\end{lemma}

\bp
If we have a $G$-circuit $(g,\gamma)$ with vertices in $M$ and of length $n$, consider the corresponding path $x\in E^\infty$ such that $\sigma^n(x)=g\cdot x$ and $x_1\dots x_n=\gamma$. The inductive procedure described after Lemma~\ref{lem:G-circuits} implies that all vertices of $x$ lie in $M$. It is also easy to see from this procedure that $(g,\gamma)$ has no entry in $M$ if and only if $x$ has no entry in $M$. Therefore if~$(g,\gamma)$ has no entry in~$M$, then $M\in\LL(G,E)$, $x\in[x_M]_\G$ and $n\in\Phi(\Gxx)=n_M\Z$.

Conversely, assume $M\in\LL(G,E)$. Since $\Phi(\G^{x_M}_{x_M})=n_M\Z$, we can find $m\ge0$ and $g\in G$ such that $\sigma^{m+n_M}(x_M)=g\cdot\sigma^m(x_M)$. Then the path $x:=\sigma^m(x_M)$ still has no entry in $M$ and satisfies $\sigma^{n_M}(x)=g\cdot x$. It follows that by letting $\gamma:=x_1\dots x_{n_M}$ we get a $G$-circuit $(g,\gamma)$ of length~$n_M$ and with no entry in $M$. This proves the lemma.
\ep

\begin{remark}\label{rem:minimal-period}
If $M$ lies in $\LL(G,E)$, then for any $G$-circuit $(g,\gamma)$ in $M$ without an entry the path~$\gamma$ is composed of $G$-circuits of length $n_M$. In order to see this, consider the path $x\in E^\infty$ corresponding to $(g,\gamma)$, so that $\sigma^{|\gamma|}(x)=g\cdot x$. To prove our claim it suffices to show that $\sigma^{n_M}(x)=g'\cdot x$ for some $g'\in G$, as then $\sigma^{n_M+k}(x)=h_k\cdot \sigma^k(x)$ for all $k\ge0$ and some $h_k\in G$ and therefore any part of $x$ of length $n_M$ arises from a $G$-circuit of length $n_M$. Since~$x$ lies on the $\G$-orbit of a path defined by a $G$-circuit of length $n_M$, we can find $k\ge1$ large enough such that $\sigma^{k|\gamma|+n_M}(x)=h\cdot \sigma^{k|\gamma|}(x)$ for some $h\in G$. Since $\sigma^{|\gamma|}(x)=g\cdot x$, we have $\sigma^{k|\gamma|}(x)=g_k\cdot x$ for some $g_k\in G$. It follows that $\sigma^{n_M}(g_k\cdot x)=(hg_k)\cdot x$ and hence $\sigma^{n_M}(x)=g'\cdot x$ for $g':=(g_k|_{x_1\dots x_{n_M}})^{-1}hg_k$.\ee
\end{remark}

\begin{remark}
The notation $\LL(G,E)$ and the name $G$-period are motivated by the fact that when~$G$ is trivial, there is a one-to-one correspondence between the elements of $\LL(G,E)$ and what we called the \emph{primitive loops} in~\cite{CN3}, and then $n_L$ is the length of the corresponding loop. The above discussion implies that as a $G$-equivariant analogue of a primitive a loop one can take an equivalence class of $G$-circuits that have no entry in the maximal $G$-tails they generate, where two $G$-circuits are called equivalent if they generate the same maximal $G$-tail. \ee
\end{remark}

While every element of $\LL(G,E)$ can be realized as $\MT_G(x)$ for an infinite path $x$, this is not in general true for elements of $\M_\gamma(G,E)$. At this point we can at least say the following.

\begin{lemma} \label{lem:Mgamma}
We have $\MT_G(E^{\infty}\cap A_{G,E}) \subset \M_{\gamma}(G,E) \subset \MT_G(A_{G,E})$.
\end{lemma}

\begin{proof}
Assume $x\in E^\infty$ and $L:=\MT_G(x)$ lies in $\LL(G,E)$. Then $\overline{[x]}_\G=\overline{[x_L]}_G$ by Lemma~\ref{lem:Gtails}. Since $[x_L]_\G$ is discrete, it follows that $x\in[x_L]_\G$, hence $x\notin A_{G,E}$. This proves the first inclusion in the formulation.

If a maximal $G$-tail $M$ does not lie in $\MT_G(A_{G,E})$, then by Proposition~\ref{prop:Gmaxt} we have $M=\MT_G(x)$ for some $x\in E^\infty\setminus A_{G,E}$ with no entry in $M$. By definition this means that $M\in\LL(G,E)$. This proves the second inclusion.
\end{proof}

In order to describe the orbit closures of finite boundary paths we will need a $G$-equivariant version of the notion of breaking vertices~\cite{MR1988256}.

\begin{defn}
A singular vertex $v\in E^{\sing}$ is called a \emph{$G$-breaking vertex} if
$$
0 < | \{e\in E^{1} : r(e) = v , s(e) \geq g \cdot v \text{ for some } g\in G\} | < \infty
$$
We denote the set of $G$-breaking vertices by $\BV(G,E)$.
\end{defn}

Note that if $v\in E^\sing$, then the maximal $G$-tail $\MT_G(v)$ consists of all vertices $w$ such that $w\ge g\cdot v$ for some $g$, hence
$$
\{e\in E^1: r(e)=v,\ s(e)\ge g\cdot v \text{ for some } g\in G\}=vE^1\MT_G(v).
$$
Therefore $v$ is a $G$-breaking vertex if and only if $0<|vE^1\MT_G(v)|<\infty$. It is easy to see that the set $\BV(G,E)$ is $G$-invariant.

\smallskip

In order to illustrate the difference between $G$-breaking vertices and usual breaking vertices (that is, defined for the trivial group action), consider the following two examples.

\begin{example}
Consider the graph shown in Figure~\ref{gr:G-breaking1}, where $\infty$ indicates that the arrow represents infinitely many edges from one vertex to the next one. The dashed lines indicate that the graph continues in the same pattern indefinitely.
\begin{figure}[ht]
\centering
\begin{tikzpicture}

\coordinate (z) at (-4,0);
\coordinate (w) at (-2,0);
\coordinate (v) at (0,0);
\coordinate (q) at (2,0);

\fill (z) circle (2pt);
\fill (w) circle (2pt);
\fill (v) circle (2pt);
\fill (q) circle (2pt);

\draw[midarrow] (q) -- (v);
\draw[midarrow] (v) -- (w);
\draw[midarrow] (w) -- (z);

\def\R{6mm}


\path (z) ++(90:\R) coordinate (cz);
\draw[->, shorten >=0pt, shorten <=0pt]
  (cz) ++(90:\R) arc[start angle=90, end angle=450, radius=\R];

\path (w) ++(90:\R) coordinate (cw);
\draw[->, shorten >=0pt, shorten <=0pt]
  (cw) ++(90:\R) arc[start angle=90, end angle=450, radius=\R];

\path (v) ++(90:\R) coordinate (cv);
\draw[->, shorten >=0pt, shorten <=0pt]
  (cv) ++(90:\R) arc[start angle=90, end angle=450, radius=\R];

\path (q) ++(90:\R) coordinate (cq);
\draw[->, shorten >=0pt, shorten <=0pt]
  (cq) ++(90:\R) arc[start angle=90, end angle=450, radius=\R];

\node at (-2.9,0.23) {$\infty$};
\node at (-0.9,0.23) {$\infty$};
\node at (1.1,0.23) {$\infty$};

\draw[dashed] (-6,0) -- (z);
\draw[dashed] (q) -- (4,0);
\end{tikzpicture}
\caption{}\label{gr:G-breaking1}
\end{figure}
Consider the obvious action of $G:=\mathbb{Z}$ on this graph by shifts. Then there are no $G$-breaking vertices, but every vertex is a breaking vertex of the underlying graph. \ee
\end{example}

\begin{example}
Consider the graph shown in Figure~\ref{gr:G-breaking2}, where $\infty$ indicates again that we have infinitely many edges and the dashed lines indicate that the graph continues in the same pattern indefinitely.
\begin{figure}[ht]
\centering
\begin{tikzpicture}

\coordinate (a1) at (-4,-2);
\coordinate (b1) at (-2,-2);
\coordinate (c1) at (0,-2);
\coordinate (d1) at (2,-2);

\coordinate (a2) at (-4,0);
\coordinate (b2) at (-2,0);
\coordinate (c2) at (0,0);
\coordinate (d2) at (2,0);

\coordinate (a3) at (-4,2);
\coordinate (b3) at (-2,2);
\coordinate (c3) at (0,2);
\coordinate (d3) at (2,2);

\foreach \p in {a1,b1,c1,d1,a2,b2,c2,d2,a3,b3,c3,d3}
  \fill (\p) circle (2pt);

\draw[midarrow] (d1) -- (c1);
\draw[midarrow] (c1) -- (b1);
\draw[midarrow] (b1) -- (a1);

\draw[midarrow] (d2) -- (c2);
\draw[midarrow] (c2) -- (b2);
\draw[midarrow] (b2) -- (a2);

\draw[midarrow] (d3) -- (c3);
\draw[midarrow] (c3) -- (b3);
\draw[midarrow] (b3) -- (a3);

\draw[midarrow] (a1) -- (a2);
\draw[midarrow] (a2) -- (a3);

\draw[midarrow] (b1) -- (b2);
\draw[midarrow] (b2) -- (b3);

\draw[midarrow] (c1) -- (c2);
\draw[midarrow] (c2) -- (c3);

\draw[midarrow] (d1) -- (d2);
\draw[midarrow] (d2) -- (d3);

\node at (-3.6,1.0) {$\infty$};
\node at (-3.6,-1.0) {$\infty$};
\node at (-1.6,1.0) {$\infty$};
\node at (-1.6,-1.0) {$\infty$};
\node at (0.4,1.0) {$\infty$};
\node at (0.4,-1.0) {$\infty$};
\node at (2.4,1.0) {$\infty$};
\node at (2.4,-1.0) {$\infty$};

\draw[dashed] (-5.5,-2) -- (a1);
\draw[dashed] (d1) -- (3.5,-2);

\draw[dashed] (-6,0) -- (a2);
\draw[dashed] (d2) -- (4,0);

\draw[dashed] (-5.5,2) -- (a3);
\draw[dashed] (d3) -- (3.5,2);

\draw[dashed] (a3) -- (-4,3.5);
\draw[dashed] (b3) -- (-2,4);
\draw[dashed] (c3) -- (0,4);
\draw[dashed] (d3) -- (2,3.5);

\draw[dashed] (-4,-3.5) -- (a1);
\draw[dashed] (-2,-4) -- (b1);
\draw[dashed] (0,-4) -- (c1);
\draw[dashed] (2,-3.5) -- (d1);
\end{tikzpicture}
\caption{}\label{gr:G-breaking2}
\end{figure}
Let $G:=\mathbb{Z}$ act on the graph by horizontal shifts. Then the graph has no breaking vertices, but every vertex is a $G$-breaking vertex.\ee
\end{example}

It is also not difficult to give examples showing that in general the $G$-breaking vertices of $(G,E)$ are not related to the breaking vertices of the quotient graph $G\backslash E$.

\begin{lemma}\label{lem:Gsing-orbits}
For every singular vertex $v\in E^\sing$ we have one (and only one) of the following possibilities.
\begin{enumerate}
  \item The set $\{e\in E^1: r(e)=v,\ s(e)\ge g\cdot v \text{ for some } g\in G\}$ is empty. Then $\MT_G(v)\in \M_\gamma(G,E)$, and if $\MT_G(v)=\MT_G(x)$ for some $x\in E^\sing\cup E^\infty$, then $x=g \cdot v$ for some $g\in G$.
  \item The set $\{e\in E^1: r(e)=v,\ s(e)\ge g\cdot v \text{ for some } g\in G\}$ is infinite. Then $\MT_G(v)\in \M_\gamma(G,E)$ and $\overline{[v]}_{\mathcal{G}}=\overline{[x]}_{\mathcal{G}}$ for some $x\in E^\infty\cap A_{G,E}$.
  \item The vertex $v$ is a $G$-breaking vertex and $\MT_G(v)\in \LL(G,E)$. Then $\MT_G(v)\ne \MT_G(x)$ for all $x\in E^\infty\cap A_{G,E}$. If $\overline{[v]}_{\mathcal{G}}=\overline{[u]}_{\mathcal{G}}$ for some $u\in E^\sing \cup E^{\infty}$, then $u=g \cdot v$ for some $g\in G$.
  \item The vertex $v$ is a $G$-breaking vertex and $\MT_G(v)\in \M_\gamma(G,E)$. Then $\MT_G(v)=\MT_G(x)$ for some $x\in E^\infty\cap A_{G,E}$. If $\overline{[v]}_{\mathcal{G}}=\overline{[u]}_{\mathcal{G}}$ for some $u\in E^\sing \cup E^{\infty}$, then $u=g \cdot v$ for some $g\in G$.
 \end{enumerate}
\end{lemma}

\begin{proof}
Fix $v\in E^\sing$ and put $M:=\MT_G(v)$. Therefore, as we mentioned above, $v$ is a $G$-breaking vertex if and only if $0<|vE^1M|<\infty$, while the cases~(1) and~(2) mean that $|vE^1M|=0$ and $|vE^1M|=\infty$, resp.

\smallskip

(1) Assume the set $vE^1M$ is empty. It follows then that $(g\cdot v) E^{1}M=\emptyset$ for all $g\in G$. Assume $M\in\LL(G,E)$. Then $r(x_M)\ge g\cdot v$ for some $g$. Since $x_M$ has no entry in $M$, it follows that $g\cdot v=r(x_{M,n})$ for some $n$. Since $x_{M,n}\in (g\cdot v) E^{1}M$, we obtain a contradiction. Thus, $M\in\M_\gamma(G,E)$.

If $M=\MT_G(x)$ for some $x\in E^\sing\cup E^\infty$, then $v\ge g\cdot w$ for a vertex $w$ of $x$. Since $vE^{1}M=\emptyset$, and so the only path in $M$ with range $v$ is $v$ itself, this is possible only when $x=g\cdot v$.

\smallskip

(2) Assume $vE^1M$ is infinite. Consider the set $\Omega$ of paths $y\in E^\infty$ with vertices in~$M$ and $r(y)=v$. We claim that it is uncountable. In order to see why this is true, similarly to the proof of Proposition~\ref{prop:Gmaxt}, Case~2, it suffices to show that for every vertex $w\in M$ there are two different paths $y,z\in wE^mM$ for some $m\ge1$, as then we can inductively construct an embedding of $\{0,1\}^\N$ into $\Omega$. To find $y$ and $z$, take $g\in G$ such that $w\ge g\cdot v$ and choose any $\lambda\in wE^*(g\cdot v)$. As the set $(g\cdot v)E^1M$ is infinite, we can choose two different edges $e_1$, $e_2$ in it. Then $y:=\lambda e_1$ and $z:=\lambda e_2$ have the required properties.

By Lemma~\ref{lem:uncountable} it follows that the set $\Omega\cap A_{G,E}$ is nonempty. Pick any element $x$ in it. To see that $v\in\overline{[x]}_\G$, take a sequence of different edges $e_n\in vE^1M$. For each $n$, we have $s(e_n)\ge g_n\cdot  v$ for some $g_n$. Take any path $\lambda_n\in s(e_n)E^*(g_n\cdot v)$. Then
$e_n\lambda_n(g_n\cdot x)\to v$. To see that $x\in\overline{[v]}_\G$, choose $h_n\in G$ such that $r(x_n)\ge h_n\cdot v$. Take any path $\mu_n\in r(x_n)E^*(h_n\cdot v)$. Then $x_1\dots x_{n-1}\mu_n(h_n\cdot v)\to x$. Therefore $\overline{[v]}_{\mathcal{G}}=\overline{[x]}_{\mathcal{G}}$. By Lemmas~\ref{lem:Gtails} and~\ref{lem:Mgamma} we then get that $M=\MT_G(x)\in\M_\gamma(G,E)$.

\smallskip

(3) Assume $v\in\BV(G,E)$ and $M\in\LL(G,E)$. For every $x\in E^\infty\cap A_{G,E}$, we know by Lemma~\ref{lem:Mgamma} that $\MT_G(x) \in \M_\gamma(G,E)$, hence $\MT_G(x) \neq M$.

Since the set $vE^1M$ is finite, there is no sequence of paths of length $\ge1$ in $M$ converging to~$v$.
This implies that if $u\in E^\sing \cup E^{\infty}$ is such that $\overline{[v]}_{\mathcal{G}}=\overline{[u]}_{\mathcal{G}}$, then $u\in E^\sing$ and $v\in[u]_\G$. Since the only paths of zero length in $[u]_\G$ are the elements of $G\cdot u$, we conclude that $u=g\cdot v$ for some $g$.

\smallskip

(4) Finally, assume $v \in \BV(G,E)$ and $M\in \M_\gamma(G,E)$. Assume first that $v$ is minimal in~$M$. Then $vE^{*}v\neq \{v\}$, because by assumption there is an edge $e\in vE^{1}M$, and since $s(e) \sim v$ by minimality of $v$, we can construct a nontrivial element of $vE^{*}v$. Moreover, we claim that $vE^{*}v$ does not consist of powers of one simple cycle $\alpha$. Suppose it does. Then $(1_G, \alpha)$ is a $G$-circuit without an entry in $M$. Indeed, it is clearly a $G$-circuit.  If there was an edge $e\in r(\alpha_{j})E^{1}M$ for some $j$ different from  $\alpha_j$, then $v \geq s(e)$ and hence $v\sim s(e)$ by minimality of~$v$, so we could find an element $\gamma \in s(e) E^{*} v$. But then $\alpha_{1}\cdots \alpha_{j-1}e\gamma \in vE^{*}v $,  contradicting that $vE^{*}v=\{\alpha^{n}\}_{n\geq 0}$. By Lemma~\ref{lem:circuits-wo-entry} we conclude then that $M\in\LL(G,E)$, which contradicts our assumption $M\in \M_\gamma(G,E)$ and hence proves our claim that $vE^{*}v$ does not consist of powers of one simple cycle. By Corollary~\ref{cor:Gmaxt}(2) we then have $M=\MT_G(x)$ for some $x\in E^\infty\cap A_{G,E}$.

Assume next that the element $v$ is not minimal in $M$. Then the elements $g\cdot v$ are not minimal either. Then there cannot be any minimal elements in $M$, since any vertex $w$ in $M$ satisfies $w\ge g\cdot v$ for some $g$. By Corollary~\ref{cor:Gmaxt}(3) there is then a path $x\in E^\infty$ such that $M=\MT_G(x)$ and either $x\in A_{G,E}$ or $x\notin A_{G,E}$ is a path with no entry in $M$. Since $M\in\M_\gamma(G,E)$ by assumption, we must have $x\in A_{G,E}$.

The last claim in (4) follows by the same argument as in (3), because it only uses that $v$ is a $G$-breaking vertex.
\end{proof}

Define
\begin{equation}\label{eq:eg0}
E^\sing_0(G):=\{v\in E^\sing: vE^1\MT_G(v)=\emptyset\}.
\end{equation}
It is immediate that the set $E^\sing_0(G)$ is $G$-invariant. Let us also denote by $\M_0(G,E)$ the set of maximal $G$-tails $M$ such that there exists a minimal element $v\in M$ satisfying $vE^1M=\emptyset$.

\begin{lemma}\label{lem:M0}
If $M\in\M_0(G,E)$ and $v\in M$ is a minimal element satisfying $vE^1M=\emptyset$, then $v\in E^\sing$ and $M=\MT_G(v)$. As a consequence, $\M_0(G,E)\subset\M_\gamma(G,E)$ and the map $G\backslash E^\sing_0(G)\to\M_0(G,E)$, $G\cdot v\mapsto \MT_G(v)$, is a bijection.
\end{lemma}

\bp
For $v\in M$ the assumption $vE^1M=\emptyset$ implies that $vE^*v=\{v\}$. Hence the first statement of the lemma follows from Corollary~\ref{cor:Gmaxt}(1). The second statement follows then from Lemma~\ref{lem:Gsing-orbits}(1).
\ep

We can now determine whether a maximal $G$-tail can be represented by an infinite, or even an infinite $G$-aperiodic, path. In order to formulate the result define
\begin{equation}\label{eq:M-infinity}
\M_\infty(G,E):=\M_\gamma(G,E)\setminus\M_0(G,E).
\end{equation}

\begin{lemma}\label{lem:M-infinity-paths}
For any maximal $G$-tail $M$, we have:
\begin{enumerate}
  \item a path $x\in E^\infty$ such that $\MT_G(x)=M$ exists if and only if $M\notin\M_0(G,E)$, that is, if and only if $M\in\M_\infty(G,E)\cup\LL(G,E)$;
  \item a path $x\in E^\infty\cap A_{G,E}$ such that $\MT_G(x)=M$ exists if and only if $M\in\M_\infty(G,E)$.
\end{enumerate}
\end{lemma}

\bp
(1) Assume first that $M\in\M_\infty(G,E)\cup\LL(G,E)$. If $M\in\LL(G,E)$, then $M=\MT_G(x_M)$ by Lemma~\ref{lem:no-entry-paths}. Consider the case $M\in\M_\infty(G,E)$. Since the set $A_{G,E}$ contains all finite boundary paths and every such path lies on the orbit of a singular vertex, by Proposition~\ref{prop:Gmaxt} we conclude that if there is no $x\in E^\infty$ satisfying $M=\MT_G(x)$, then $M=\MT_G(v)$ for some $v\in E^\sing$. We now apply Lemma~\ref{lem:Gsing-orbits} to this vertex. The assumptions  $M\notin\M_0(G,E)$ and $M\notin\LL(G,E)$ mean that cases (1) and (3) of that lemma do not apply. But in both remaining cases (2) and (4) we can find $x\in E^\infty\cap A_{G,E}$ such that $\MT_G(v)=\MT_G(x)$, which is a contradiction.

On the other hand, if $M\in\M_0(G,E)$, then, by Lemma~\ref{lem:M0}, $M=\MT_G(v)$ for any minimal element $v\in M$ with $vE^1M=\emptyset$. By Lemma~\ref{lem:Gsing-orbits}(1) there is then no $x\in E^\infty$ such that $M=\MT_G(x)$.

\smallskip

(2) Assume $M\in\M_\infty(G,E)$. By Proposition~\ref{prop:Gmaxt}, if there is no $x\in E^\infty\cap A_{G,E}$ satisfying $M=\MT_G(x)$, then either $M=\MT_G(v)$ for some $v\in E^\sing$ or $M=\MT_G(x)$ for some path $x\in E^\infty\setminus A_{G,E}$ with no entry in $M$. Arguing in the same way as in the proof of (1), we see that the first case is not possible by Lemma~\ref{lem:Gsing-orbits}. The second case is not possible either, since then by definition we have $M\in\LL(G,E)$.

Conversely, assume $M=\MT_G(x)$ for some $x\in E^\infty\cap A_{G,E}$. By part (1) we then have either $M\in\M_\infty(G,E)$ or $M\in\LL(G,E)$. If $M\in\LL(G,E)$, then $M=\MT_G(x_M)$. By Lemma~\ref{lem:Gtails} it follows that $\overline{[x]}_\G=\overline{[x_M]}_\G$. Since the $\G$-orbit of $x_M$ is discrete by Lemma~\ref{lem:no-entry-paths}, we then conclude that $x\in[x_M]_\G$. This contradicts the assumption $x\in A_{G,E}$. Hence $M\in\M_\infty(G,E)$.
\ep

\begin{cor}\label{cor:G-periodic-paths}
If $x\in E^\infty\setminus A_{G,E}$, then either $x\in[x_L]_\G$ for some $L\in\LL(G,E)$, or $\overline{[x]}_\G=\overline{[y]}_\G$ for some $y\in E^\infty\cap A_{G,E}$.
\end{cor}

\bp
Consider $M:=\MT_G(x)$. Then, by Lemma~\ref{lem:M-infinity-paths}(1), we have either $M\in\LL(G,E)$ or $M\in\M_\infty(G,E)$. In the first case $x\in[x_M]_\G$ by Lemmas~\ref{lem:Gtails} and~\ref{lem:no-entry-paths}. In the second case $\overline{[x]}_\G=\overline{[y]}_\G$ for some $y\in E^\infty\cap A_{G,E}$ by Lemma~\ref{lem:M-infinity-paths}(2).
\ep

We are now ready to describe the quasi-orbit space for the action $\G\curvearrowright\partial E$. We stress again that the definition of this space does not involve any topology on $\G$, only the one on $\partial E$, so we do not need assumptions like pseudo-freeness. We denote by $\QQ(x)$ the quasi-orbit of $x$.

\begin{thm}\label{thm:quasio}
Assume $(G,E)$ is a countable self-similar directed graph and consider the corresponding Exel--Pardo groupoid $\G:=\G_{G,E}$. Then there is a unique bijection
$$
\pi\colon \M(G,E)\sqcup(G\backslash\BV(G,E))\to (\G\backslash\partial E)^\sim
$$
satisfying the following properties:
\begin{itemize}
  \item[(i)] if $L\in\LL(G,E)$, then $\pi(L)=\QQ(x_L)$;
  \item[(ii)] if $M\in\M_0(G,E)$, that is, there is a minimal element $v\in M$ such that $vE^1M=\emptyset$, then $v\in E^\sing$ and $\pi(M)=\QQ(v)$;
  \item[(iii)] if $M\in\M_\infty(G,E)=\M_\gamma(G,E)\setminus\M_0(G,E)$, then~$\pi(M)=\QQ(x)$ for any path $x\in E^\infty$ such that $\MT_G(x)=M$; moreover, such a path $x$ can be chosen to lie in $E^\infty\cap A_{G,E}$;
  \item[(iv)] if $v\in\BV(G,E)$, then $\pi(G\cdot v)=\QQ(v)$.
\end{itemize}

The image of the quasi-orbit $\QQ(x)\in(\G\backslash\partial E)^\sim$ of a path $x\in\partial E$ under the inverse map $\pi^{-1}$ is described as follows: if $x\in \partial E\cap E^*$ and $s(x)$ is a $G$-breaking vertex, then $\pi^{-1}(\QQ(x))=G\cdot s(x) \in G\backslash\BV(G,E)$, otherwise $\pi^{-1}(\QQ(x))=\MT_G(x)\in\M(G,E)$.
\end{thm}

\bp
First of all let us see that the map $\pi$ with properties as in the statement is well-defined and unique. We have to show that the specified images in cases (i)--(iii) do not depend on any choices and the claimed choices are actually possible to make. For (i) this follows from the fact the $\G$-orbit of $x_L$ is uniquely determined by $L$ by Corollary~\ref{cor:unique-orbit} (recall the discussion preceding Definition~\ref{def:LGE}). For (ii) this follows from Lemma~\ref{lem:M0}. For (iii)  this follows from Lemmas~\ref{lem:M-infinity-paths}(2) and~\ref{lem:Gtails}.

Next let us show that $\pi$ is surjective. Take a boundary path $x\in\partial E$. If $x\in E^\infty\cap A_{G,E}$, then $\QQ(x)\in\operatorname{Im}\pi$ by case (iii) and Lemma~\ref{lem:M-infinity-paths}(2). If $x\in E^\infty\setminus A_{G,E}$, then $\QQ(x)\in\operatorname{Im}\pi$ by case~(i) and Corollary~\ref{cor:G-periodic-paths}. Since every finite boundary path belongs to the $\G$-orbit of a singular vertex, it remains to consider $x\in E^\sing$. If $x$ is a $G$-breaking vertex, then $\QQ(x)\in\operatorname{Im}\pi$ by case (iv). If $x\in E^\sing_0(G)$, then $\QQ(x)\in\operatorname{Im}\pi$ by case (ii) and Lemma~\ref{lem:M0}. By Lemma~\ref{lem:Gsing-orbits}, the only remaining case corresponds to part (2) of that lemma, according to which there exists $y\in E^\infty\cap A_{G,E}$ such that $\overline{[x]}_\G=\overline{[y]}_\G$. Hence $\QQ(x)\in\operatorname{Im}\pi$.

By Lemma~\ref{lem:Gtails} and cases (3) and (4) of Lemma~\ref{lem:Gsing-orbits}, we can define a map $\vartheta\colon (\G\backslash\partial E)^\sim\to \M(G,E)\sqcup(G\backslash\BV(G,E))$ such that if $x\in \partial E\cap E^*$ and $s(x)$ is a $G$-breaking vertex, then $\vartheta(\QQ(x))=G\cdot s(x) \in G\backslash\BV(G,E)$, while for all other boundary paths we have $\vartheta(\QQ(x))=\MT_G(x)\in\M(G,E)$. Then it is easy to see from the definition of $\pi$ that $\vartheta\circ\pi=\operatorname{id}$. This finishes the proof of the theorem.
\ep


This theorem describes the quasi-orbit space as a set, but it is also not difficult to understand the topology on $(\G\backslash\partial E)^\sim$ in this picture, since $\QQ(x_n)\to\QQ(x)$ if and only if there exist $y_n\in[x_n]_\G$ such that $y_n\to x$. We will say a bit more on this in the next subsection.

\smallskip

From our description of $(\G\backslash\partial E)^\sim$ we easily get a criterion for minimality of the Exel--Pardo groupoid, that is, for triviality of $(\G\backslash\partial E)^\sim$. This result follows also from~\cite{KM}*{Proposition~6.32}.

\begin{cor}\label{cor:minimality}
In the setting of Theorem~\ref{thm:quasio}, the action $\G\curvearrowright\partial E$ is minimal if and only if the only maximal $G$-tail in $(G,E)$ is $E^0$.
\end{cor}

\bp
Since every vertex of $E$ is the range of a boundary path, it is contained in some maximal $G$-tail. Hence $(\G\backslash\partial E)^\sim$ is a single point if and only if the only maximal $G$-tail is $E^0$ and there are no $G$-breaking vertices. It remains to show that if $E^0$ is the only maximal $G$-tail, then there are no $G$-breaking vertices. Assume there does exist a vertex $v\in\BV(G,E)$. Then, on the one hand, the vertex $v$ is singular. On the other hand, we have $0<|vE^1\MT_G(v)|<\infty$. Since $\MT_G(v)=E^0$, we get a contradiction.
\ep

\subsection{The ideal structure of self-similar graph algebras}
We now assume that $(G,E)$ is pseudo-free and want to apply results of Section~\ref{sec:ess-isotropy} to describe $\Prim C^*(\G_{G,E})$ under suitable conditions. First we need to understand when the groupoid has essentially central isotropy.

\begin{prop} \label{prop:esscent}
Assume $(G,E)$ is a countable pseudo-free self-similar directed graph. Consider the corresponding Exel--Pardo groupoid $\G:=\G_{G,E}$ with its standard grading $\Phi\colon\G\to\Z$. Then $\G$ has essentially central isotropy if and only if the following conditions are satisfied:
\begin{itemize}
  \item[(i)] if $v\in \BV(G,E)\cup E^\sing_0(G)$, then the stabilizer of $v$ in~$G$ is trivial;
  \item[(ii)] if $L\in\LL(G,E)$ and $(g,\gamma)$ is a $G$-circuit with vertices in $L$ and no entry in $L$, then the stabilizer of $r(\gamma)$ in $G$ is trivial;
  \item[(iii)] if $M\in\M_\infty(G,E)$ and $v\in M$, then there are no elements $g\ne1_G$ that act trivially on $vE^*M$.
\end{itemize}
\end{prop}

We remind here that a singular vertex $v$ belongs to $\BV(G,E)\cup E^\sing_0(G)$ if and only if the set
$$
vE^1\MT_G(v)=\{e\in E^1: r(e)=v,\ s(e)\ge g\cdot v \text{ for some } g\in G\}
$$
is finite, while $\M_\infty(G,E)$ is exactly the set of maximal $G$-tails $M$ such that $M=\MT_G(x)$ for some $x\in E^\infty\cap A_{G,E}$.

Before we turn to the proof of the proposition, let us comment on condition (ii). As the following lemma shows, it suffices to check this condition for \emph{some} $G$-circuit $(g,\gamma)$. Instead of~$r(\gamma)$ we may also require the stabilizer of some other vertex of $\gamma$, or even the stabilizer of some part of $\gamma$, to be trivial.

\begin{lemma}\label{lem:L-stabilizers}
Assume $L\in\LL(G,E)$ and $(g,\gamma)$ is a $G$-circuit with vertices in $L$ and no entry in~$L$. Let $x\in E^\infty$ be the corresponding path such that $\sigma^{|\gamma|}(x)=g\cdot x$ (see Lemma~\ref{lem:G-circuits}). Then the following conditions are equivalent:
\begin{enumerate}
  \item the stabilizer of some vertex of $x$ in $G$ is trivial;
  \item the stabilizer of every vertex of $x$ in $G$ is trivial;
  \item the stabilizer of $x$ in $G$ is trivial.
\end{enumerate}
As a consequence, these conditions depend only on $L$, not on the choice of $(g,\gamma)$.
\end{lemma}

\bp
Since $x$ has no entry in $L$ by Lemma~\ref{lem:circuits-wo-entry} and its proof, it is easy to see that the stabilizer of $x$ coincides with the stabilizer of $r(x_1)$. It follows that in order to proof equivalence of (1)--(3) it suffices to show that (1) implies (2). Equivalently, we need to show that if the stabilizer of some vertex of $x$ is nontrivial, then the stabilizer of every vertex of $x$ is nontrivial. So assume $h\ne1_G$ stabilizes $r(x_k)$ for some $k\ge1$. Since the path $x_kx_{k+1}\dots$ has no entry in $L$, then $h$ stabilizes this path, hence it stabilizes every vertex $r(x_m)$ with $m\ge k$.

Recall that $x$ is constructed as $\gamma^{(1)}\gamma^{(2)}\dots$ for some $G$-circuits $(g_n,\gamma^{(n)})$ defined inductively from $(g,\gamma)$. Then $r(x_k)$ is a vertex of $\gamma^{(n)}$ for some $n$, while $s(\gamma^{(n)})$ coincides with $r(x_m)$ for some $m\ge k$. Therefore the stabilizer of $s(\gamma^{(n)})$ is nontrivial. Since $s(\gamma^{(n)})=g_n\cdot r(\gamma^{(n)})$, we conclude that the stabilizer of $r(\gamma^{(n)})$ is also nontrivial. Since $r(\gamma^{(n)})=s(\gamma^{(n-1)})$, by repeating this argument we eventually conclude that the stabilizer of $r(\gamma)=r(x_1)$ is nontrivial, hence the stabilizer of $r(x_n)$ is nontrivial for all $n\ge1$.

Finally, take another $G$-circuit $(g',\gamma')$ with vertices in $L$ and no entry in $L$, and let $x'\in E^\infty$ be the corresponding infinite path. Then both~$x$ and~$x'$ lie on the $\G$-orbit of $x_L$ by Corollary~\ref{cor:unique-orbit}, so we have $x'=\mu(g\cdot\sigma^m(x))$ for some $\mu\in E^*$, $g\in G$ and $m\ge0$. Hence $x'$ has some vertices of the form $g\cdot r(x_n)$. Therefore if condition (2) is satisfied for $(g,\gamma)$, then condition (1) is satisfied for $(g',\gamma')$. This implies the last statement of the lemma.
\ep




\bp[Proof of Proposition~\ref{prop:esscent}]
To ease the notation let us write $[x]$ for $[x]_\G$. Let us show first that, given $x\in \partial E$, the map $\Phi$ is injective on $\IsoGx{x}^\circ_x$ if and only if there are no $g\ne1_G$ and $m\ge0$ such that $g$ acts trivially on a neighbourhood of $\sigma^m(x)$ in $\overline{[x]}$.

For the ``if'' direction, assume $\Phi$ is not injective on $\IsoGx{x}^\circ_x$. By the definition of the topology on $\G$ this implies that there is a bisection $Z(\mu,g,\nu,F)$ such that $d(\mu)=d(\nu)$, $g\ne1_G$, $x\in Z_F(\nu)$ and $Z(\mu,g,\nu,F)\cap \G_{\overline{[x]}}\subset \IsoG$. The last condition is equivalent to $\G^y_y\cap Z(\mu,g,\nu,F)\ne\emptyset$ for all $y\in \overline{[x]}\cap Z_F(\nu)$. Then we must have $\mu=\nu$. Put $m:=d(\mu)$ and $v:=r(\sigma^m(x))$. By definition of the set $Z(\mu,g,\nu,F)$ we conclude that $g$ acts trivially on $\overline{[x]}\cap Z_F(v)$.

Conversely, if $g$ and $m$ with properties as above exist, then, since $G\ltimes\partial E$ is an open subgroupoid of $\G$ by Lemma~\ref{lem:trans-embedding}, the element $g$ defines a nontrivial element of $\IsoGx{x}^\circ_{\sigma^m(x)}$ that is killed by $\Phi$. Since $\sigma^m(x)\in[x]$, the map $\Phi$ is not injective on $\IsoGx{x}^\circ_x$ either.

\smallskip

We will now use this criterion to determine when $\G$ has essentially central isotropy. Recall that by Lemma~\ref{lem:iso} it suffices to analyze representatives of the quasi-orbits for the action $\G\curvearrowright \partial E$. By Theorem~\ref{thm:quasio} we have quasi-orbits of four types.

We will treat the vertices in $\BV(G,E)$ and $E^\sing_0(G)$ (cases (ii) and (iv) of Theorem~\ref{thm:quasio}) together. Assume $v$ is such a vertex, so that the set
$F:=vE^1\MT_G(v)$ is finite. Then
$$
vE^*\MT_G(v)\cap Z_F(v)=\{v\}.
$$
This implies that $\overline{[v]}\cap Z_F(v)=\{v\}$. Since $v\notin\operatorname{dom}(\sigma)$, it follows that $\Phi$ is injective on $\IsoGx{v}^\circ_v$ if and only if the stabilizer of $v$ in $G$ is trivial.

Next consider the path $x_L$ for some $L\in\LL(G,E)$. Since $[x_{L}]=[\sigma^{n}(x_L)]$ for any $n\geq0$, we can assume without loss of generality that $x_{L}$ is defined by a $G$-circuit without an entry in $L$. Since the orbit $[x_L]$ is discrete by Lemma~\ref{lem:no-entry-paths}, an element $g\ne 1_G$ acts trivially on a neighbourhood of $\sigma^m(x_{L})$ in $\overline{[x_{L}]}$ if and only if it acts trivially on $\sigma^m(x_{L})$, and since $\sigma^m(x_{L})$ has no entries in $L$, this is equivalent to $g$ fixing $r(\sigma^m(x_{L}))$. It follows that $\Phi$ is injective on $\IsoGx{x_{L}}^\circ_{x_{L}}$ if and only if the stabilizer of every vertex of $x_{L}$ is trivial, which by Lemma~\ref{lem:L-stabilizers} is equivalent to condition~(ii) of the proposition.

Finally, consider a path $x\in E^\infty\cap A_{G,E}$ and let $M:=\MT_G(x)$. Assume $\Phi$ is not injective on $\IsoGx{x}^\circ_x$, so there exist $g\ne1_G$ and $m\ge0$ such that $g$ acts trivially on a neighbourhood of~$\sigma^m(x)$ in  $\overline{[x]}$. Since $\sigma^m(x)$ is an infinite path, we may assume that this neighbourhood has the form $\overline{[x]}\cap Z(\lambda)$. Then the element $h:=g|_\lambda$, which is nontrivial by pseudo-freeness, acts trivially on all paths $y\in\overline{[x]}\cap Z(v)$, where $v:=s(\lambda)$. Since any path in $vE^*M$ can be prolonged to an element of $[x]$, it follows that $h$ acts trivially on $vE^*M$.

Conversely, assume there exist $v\in M$ and $g\ne1_G$ such that $g$ acts trivially on $vE^*M$. It follows that~$g$ acts trivially on $[x]\cap Z(v)$, hence also on $\overline{[x]}\cap Z(v)$. Taking any $y\in[x]$ with $r(y)=v$, we conclude that $\Phi$ is not injective on $\IsoGx{x}^\circ_y$, hence $\Phi$ is not injective on $\IsoGx{x}^\circ_x$ either.
\ep

We are now ready to prove the main result of this section. It will be convenient to fix representatives of quasi-orbits for the action $\G\curvearrowright\partial E$. For $L\in\LL(G,E)$ we have already done this: recall that $x_L\in E^\infty\setminus A_{G,E}$ denotes a fixed path with vertices in $L$ and no entry in $L$. For $M\in\M_0(G,E)$ we take $x_M$ to be any minimal element of $M$ such that $x_ME^1M=\emptyset$. For $M\in \M_\infty(G,E)$ we take $x_M$ to be any path in $E^\infty\cap A_{G,E}$  such that $\MT_G(x)=M$.

Recall from Section~\ref{ssec:parameterization} that  we denote by $\pi_{(x,z)}$, where $z\in\T$ is viewed as a character of $\Z$, the irreducible representation
$\Ind^\G_{\Gxx}(z\circ\Phi|_{\Gxx})$ of $C^*(\G)$. Recall also that for $L\in\LL(G,E)$ we denote by $n_L\ge1$ the $G$-period of $L$, so that $\Phi(\G^{x_L}_{x_L})=n_L\Z$.

\begin{thm}\label{thm:main-rank-one}
Assume $(G,E)$ is a countable pseudo-free self-similar directed graph and consider the corresponding Exel--Pardo groupoid $\G:=\G_{G,E}$. Assume that the stabilizer of every vertex of $E$ in $G$ is amenable and conditions (i)--(iii) of Proposition~\ref{prop:esscent} are satisfied. Then the groupoid $\G$ is amenable and we have a bijection
$$
\M_{\gamma}(G,E)\sqcup (G\backslash\BV(G,E))\sqcup(\LL(G,E)\times\T)\to \Prim C^{*}(\G)
$$
such that $\M_{\gamma}(G,E)\ni M\mapsto\ker\pi_{(x_M,1)}$, $G\backslash\BV(E)\ni G\cdot v\mapsto \ker\pi_{(v,1)}$, $\LL(G,E)\times\T\ni (L,w)\mapsto \ker\pi_{(x_L,z)}$, where $z\in\T$ is any $n_{L}$-th root of $w$.

The topology on  $\Prim C^{*}(\G)$ is described as follows. Consider a sequence of elements $((x_n,z_n))_n$ and an element $(x,z)$, each of the form $(x_M,1)$, $(v,1)$ or $(x_L,z')$. Then we have:
\begin{enumerate}
\item[(i)] if $x=x_M$ ($M\in\M_{\gamma}(G,E)$) or $x=v\in \BV(G,E)$, then $\ker \pi_{(x_n,z_n)} \to \ker \pi_{(x,z)}$ if and only if there exist $y_{n}\in [x_{n}]_\G$ such that $y_{n} \to x$;
\item[(ii)] if $x=x_L$ ($L\in\LL(G,E)$) and $x_{n}\ne x_L$ for all $n$, then $\ker \pi_{(x_n,z_n)} \to \ker \pi_{(x,z)}$ if and only if there exist $y_{n}\in [x_{n}]_\G$ such that $y_{n} \to x$;
\item[(iii)] if $x=x_L$ ($L\in\LL(G,E)$) and $x_{n}=x_L$ for all $n$, then $\ker \pi_{(x_n,z_n)} \to \ker \pi_{(x,z)}$ if and only if  $z_n^{n_{L}}\to z^{n_{L}}$.
\end{enumerate}
\end{thm}

\bp
We write $[y]$ for $[y]_\G$. The groupoid $\G$ is amenable by Theorem~\ref{thm:EP-amenability}. By Proposition~\ref{prop:esscent} it has essentially central isotropy. By Theorem~\ref{thm:essentially-central}, Corollary~\ref{cor:differen-reps} and the discussion following it, the points of the primitive spectrum can then be indexed by the pairs $(x,z)$, where $x$ is a representative of a quasi-orbit for the action $\G\curvearrowright\partial E$ and $z\in\T$, and what matters is only the class of $z$ in $\T/\Phi(\IsoGx{x}^\circ_x)^\perp$.

For our choices of representatives of quasi-orbits the essential isotropy groups are described as follows. If $v\in\BV(G,E)\cup E^\sing_0(G)$, then, as we saw in the proof of Proposition~\ref{prop:esscent}, the vertex~$v$ is an isolated point of $[v]$, so by condition (i) of that proposition we have $\G^v_v=\IsoGx{v}^\circ_v=\{v\}$. If $L\in\LL(G,E)$, then $x_L$ is an isolated point of $[x_L]$, so $\G^{x_L}_{x_L}=\IsoGx{x_L}^\circ_{x_L}$ and, by definition, $\Phi(\G^{x_L}_{x_L})=n_L\Z$. Finally, if $x\in E^\infty\cap A_{G,E}$, then essential centrality and the definition of $A_{G,E}$ mean that $\IsoGx{x}^\circ_x=\{x\}$, while $\Gxx$ can in principle be very large. Therefore the bijection in the statement of the theorem follows from Theorem~\ref{thm:essentially-central} (recall also the discussion preceding Proposition~\ref{prop:central-necessity}) and the description of the quasi-orbit space $(\G\backslash\partial E)^\sim$ given in Theorem~\ref{thm:quasio}.

\smallskip

In order to describe the topology we apply Theorem~\ref{thm:convfg}. Case (i) follows immediately from this theorem, since $\IsoGx{x}^\circ_x=\{x\}$. Case (iii) is also easy, since the orbit $[x_L]$ is discrete and therefore the only choice of points $y_n\in[x_n]=[x_L]$ converging to $x_L$ is $y_n=x_L$.

Consider case (ii). Take a number $k\ge1$ divisible by $n_L$. Then we can find $m\ge0$ and $g\in G$ such that $\sigma^k(\sigma^m(x_L))=g\cdot\sigma^m(x_L)$. Put
$$
\gamma:=x_{L,m+1}\dots x_{L,m+k},\quad\mu:=x_{L,1}\dots x_{L,m+k},\quad \nu:=x_{L,1}\dots x_{L,m}.
$$
Then $(g,\gamma)$ is a $G$-circuit of length $k$ and $\sigma^m(x_L)$ is the corresponding infinite path. Consider the bisection $W_k:=Z(\mu,g,\nu)$, which contains the element $(x_L,\TT_{m+k}([g|_{\sigma^{m}(x_L)}]),k,x_L)\in\IsoGx{x_L}^\circ_{x_L}$. If $y_n\in[x_n]$ and $W_k\cap\G^{y_n}_{y_n}\ne\emptyset$, then $y_n=\mu(g\cdot\sigma^m(y_n))=\nu\sigma^m(y_n)$, hence also $\sigma^k(\sigma^m(y_n))=g\cdot \sigma^m(y_n)$. It follows that $\sigma^m(y_n)$ is the path corresponding to the $G$-circuit $(g,\gamma)$, that is, $\sigma^m(y_n)=\sigma^m(x_L)$, and then $y_n=\nu\sigma^m(y_n)=x_L$. This shows that $[x_n]=[x_L]$, contradicting our assumption that $x_n$ and $x_L$ represent different quasi-orbits.

Therefore for any choice of points $y_n\in[x_n]$ and any $k\ge1$ divisible by $n_L$, we have $W_k\cap\G^{y_n}_{y_n}=\emptyset$, hence also $W_k^{-1}\cap\G^{y_n}_{y_n}=\emptyset$. By Theorem~\ref{thm:convfg} we conclude that $\ker \pi_{(x_n,z_n)} \to \ker \pi_{(x,z)}$ if and only if there exist $y_{n}\in [x_{n}]$ such that $y_{n} \to x$.
\ep

Existence of paths $y_n\in[x_n]$ converging to $x$ in the theorem is equivalent to convergence $\QQ(x_n)\to\QQ(x)$ in $(\G\backslash\partial E)^\sim$. This convergence can be entirely described in terms of maximal $G$-tails and finite paths. The whole list of rules is hardly illuminating and we will not attempt to compile it here, but the main source of difficulties are the vertices $x$ in $\BV(G,E)$ and $E^\sing_0(G)$, while for all other types of quasi-orbits there is the following simple description.

\begin{lemma}\label{lem:convergence}
Assume we are given a sequence $(x_n)_n$ in $\partial E$ and a maximal $G$-tail $M\in\M_\infty(G,E)\cup\LL(G,E)$. Then the following conditions are equivalent:
\begin{enumerate}
  \item there exists a sequence of elements $y_n\in[x_n]_\G$ such that $y_n\to x_M$;
  \item for every vertex $v\in M$, we have $v\in \MT_G(x_n)$ for all $n$ large enough.
\end{enumerate}
\end{lemma}

\bp
(1)$\Rightarrow$(2) Take a vertex $v\in M$. Let $y\in[x_M]_\G$ be such that $r(y)=v$. Then $y$ has the form $\mu(g\cdot\sigma^m(x_M))$. Since $y_n\to x_M$, for all $n$ large enough the elements $\mu(g\cdot\sigma^m(y_n))$ are well-defined and have range $v$, hence $v\in M_G(x_{n})$.

(2)$\Rightarrow$(1) Choose a strictly increasing sequence of natural numbers $m_k$ such that $v_k:=r(x_{M,k})\in M_n$ for all $n\ge m_k$. For $m_k\le n< m_{k+1}$, choose $y_n'\in[x_n]_\G$ with $r(y_n')=v_k$, and put $y_n:=x_{M,1}\dots x_{M,k-1}y_n'$. Then $y_n\to x_M$.
\ep

Note that by Proposition~\ref{prop:central-necessity}, if some of the conditions (i)--(iii) of Proposition~\ref{prop:esscent} are not satisfied, then the ideal structure is certainly more complicated compared to that given by Theorem~\ref{thm:main-rank-one}, and to unravel it one has to analyze more carefully the kernels of the maps $\Phi|_{\Gxx}\colon \Gxx\to\Z$, see Remark~\ref{rem:isotropy}. Simple examples, such as the one below, show that one should expect that these conditions do \emph{not} hold more often than they do, but at the same time they put no restrictions on the individual stabilizers of vertices or infinite paths.

\begin{example}
Consider the finite directed graph $E_{n,m}$ shown in Figure~\ref{gr:Enm}, with $n,m\ge2$.
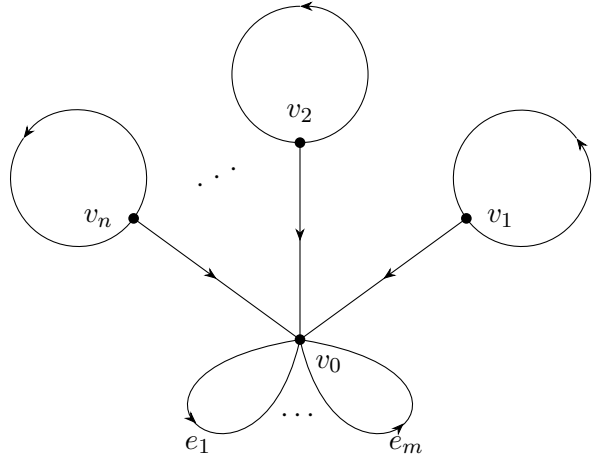
\begin{figure}[ht]
\begin{tikzpicture}

\coordinate (v0) at (0,0);
\coordinate (v1) at (2.2,1.6);
\coordinate (v2) at (0,2.6);
\coordinate (vn) at (-2.2,1.6);

\fill (v0) circle (2pt) node[below right=2pt] {$v_0$};
\fill (v1) circle (2pt) node[right=4pt] {$v_1$};
\fill (v2) circle (2pt) node[above=4pt] {$v_2$};
\fill (vn) circle (2pt) node[left=4pt] {$v_n$};


\draw[midarrow] (v1) -- (v0);
\draw[midarrow] (v2) -- (v0);
\draw[midarrow] (vn) -- (v0);

\def\R{9mm}  

\path (v1) ++(36:\R) coordinate (c1);
\draw[->, shorten >=0pt, shorten <=0pt]
  (c1) ++(36:\R) arc[start angle=36, end angle=396, radius=\R];

\path (v2) ++(90:\R) coordinate (c2);
\draw[->, shorten >=0pt, shorten <=0pt]
  (c2) ++(90:\R) arc[start angle=90, end angle=450, radius=\R];

\path (vn) ++(144:\R) coordinate (cN);
\draw[->, shorten >=0pt, shorten <=0pt]
  (cN) ++(144:\R) arc[start angle=144, end angle=504, radius=\R];

\draw[midarrow]
  (v0)
    .. controls +(-3,-0.5) and +(-0.6,-2.5)
    .. (v0);
\node[below left=4pt] at ($(v0)+(-0.9,-1.0)$) {$e_1$};

\begin{scope}[shift={(v0)}, xscale=-1]
  \draw[midarrow] (0,0) .. controls +(-0.6,-2.5) and +(-3,-0.5) .. (0,0);
\end{scope}
\node[below right=4pt] at ($(v0)+(0.9,-1.0)$) {$e_m$};

\node at (0,-1) {$\cdots$};
\node at (-0.88,2.24) {$\cdot$};
\node at (-1.1,2.14) {$\cdot$};
\node at (-1.32,2.04) {$\cdot$};
\end{tikzpicture}
\caption{Graph $E_{n,m}$}\label{gr:Enm}
\end{figure}
Every graph automorphism of $E_{n,m}$ permutes the vertices $v_1,\dots,v_n$ and the edges $e_1,\dots,e_m$, and is completely determined by these permutations, so $\Aut(E_{n,m})\cong S_n\times S_m$. Fix a subgroup $G\subset \Aut(E_{n,m})$. Then the maximal $G$-tails in $E_{n,m}$ are the sets $\{v_0\}$ and $\{v_0\}\cup O$, where~$O$ is a $G$-orbit in $\{v_1,\dots.v_n\}$. Combining Corollary \ref{cor:Gmaxt}$(2)$ with Lemma \ref{lem:M-infinity-paths}$(2)$ we see that the maximal $G$-tail~$\{v_0\}$ lies in $\M_\infty(G,E)$, while Lemma \ref{lem:circuits-wo-entry} implies that all other ones lie in~$\LL(G,E)$. It follows that conditions (i)--(iii) of Proposition~\ref{prop:esscent} are equivalent to the following two conditions:
\begin{enumerate}
  \item[(a)] the action of $G$ on $\{v_1,\dots,v_n\}$ is free;
  \item[(b)] the action of $G$ on $\{e_1,\dots,e_m\}$ is effective.
\end{enumerate}
Condition (a) already implies that $|G|\le n$ independently of how large $m$ and $\Aut(E_{n,m})$ are. 
At the same time it is easy to embed any finite group into $\Aut(E_{n,m})$ for $n$ and $m$ large enough in such a way that conditions~(a) and~(b) are satisfied.

If conditions (a) and (b) are satisfied, then by Theorem~\ref{thm:main-rank-one} we conclude that $$\Prim C^*(\G_{G,E_{n,m}})=\Prim(\OO_{E_{n,m}}\rtimes G)$$ can be identified with the union of $\{v_0\}$ and the sets $\{O\}\times\T$, where  $O$ runs through the set of $G$-orbits in $\{v_1,\dots.v_n\}$. The topology is described as follows: the point $v_0$ is closed and the closure of $\{O\}\times X$ is $\{v_0\}\cup(\{O\}\times\bar X)$ for any nonempty subset $X\subset\T$, where $\bar X\subset\T$ is the closure of $X$ in the standard topology.

We can also consider similar graphs with an infinite number of vertices $v_i$ and/or edges~$e_j$, but we need to assume that if $n=\infty$ then also $m=\infty$, as otherwise $v_0$ becomes a $G$-breaking vertex for any choice of $G$. Any countable group $G$ can be realized as a subgroup of $\Aut(E_{\infty,\infty})$ satisfying conditions (a) and (b). Moreover, any countable collection of subgroups of $G$ can be realized as stabilizers of some of the edges $e_j$, hence as stabilizers of some infinite paths. In this case in order to apply Theorem~\ref{thm:main-rank-one} we need $G$ to be amenable, and then the topology on the primitive spectrum is described similarly to the case of finite $n$ and $m$.

\smallskip

If conditions (a) and (b) are not satisfied, then it is still possible to describe $\Prim(\OO_{E_{n,m}}\rtimes G)$, at least for groups $G$ of local polynomial growth, using results of~\cite{CN4}. Namely, assume for simplicity that $m=\infty$ if $n=\infty$. Denote by $H\subset G$ the subgroup of elements that leave all edges $e_j$ invariant. For every orbit $O$ of $G$ in $\{v_i\}_i$, fix a vertex in $O$ and denote by $H_O\subset G$ its stabilizer. Then, using~\cite{CN4}*{Theorem~4.2}, one can show with some work that $\Prim(\OO_{E_{n,m}}\rtimes G)$ can be identified with
$$
\Prim C^*(H)\sqcup\bigsqcup_O(\T\times\Prim C^*(H_O)),
$$
and the topology on it is described by the following rules. The set $\Prim C^*(H)$ is closed in $\Prim(\OO_{E_{n,m}}\rtimes G)$, and the relative topology on it coincides with the Jacobson topology. If $\Omega_O\subset \T\times\Prim C^*(H_O)$ are some subsets, at least one of which is nonempty, then the closure of $\bigcup_O\Omega_O$ in $\Prim(\OO_{E_{n,m}}\rtimes G)$ consists of the closures $\bar\Omega_O\subset \T\times\Prim C^*(H_O)$ in the product-topology together with all ideals $J\in\Prim C^*(H)$ such that
\begin{equation}\label{eq:weak-graph-example}
\pi_J\prec\bigoplus_{O}\bigoplus_{g\in G/H_O}\bigoplus_{(z,I)\in\Omega_O}\Ind^H_{H\cap H_O}((\pi_I\circ\Ad g)|_{H\cap H_O}).
\end{equation}

In fact, it can be checked that this description of $\Prim(\OO_{E_{n,m}}\rtimes G)$ remains true for any amenable countable group $G$ such that $H\subset G$ is a subgroup of local polynomial growth. Note also that when $H$ is finite, condition~\eqref{eq:weak-graph-example} simplifies by Frobenius reciprocity to the requirement that the representations $\pi_J|_{H\cap H_O}$ and $(\pi_I\circ\Ad g)|_{H\cap H_O}$ are not disjoint for some orbit~$O$, $g\in G$ and $(z,I)\in\Omega_O$.\ee
\end{example}

Our description of the primitive spectrum gives in particular a criterion for simplicity of the algebras~$C^*(\G_{G,E})$. However, it is easier to obtain such a criterion using more traditional methods, for which one needs only a part of our considerations, cf.~\cite{MR3581326}*{Sections~13--14}, \cite{KM}*{Theorem~9.35}.

\begin{prop}\label{prop:simple-rank-one}
Assume $(G,E)$ is a countable pseudo-free self-similar directed graph and consider the corresponding Exel--Pardo groupoid $\G:=\G_{G,E}$. Assume that the following conditions are satisfied:
\begin{itemize}
  \item[(i)] the only maximal $G$-tail is $E^0$;
  \item[(ii)] there are no $G$-circuits without an entry in $E^0$;
  \item[(iii)] for every vertex $v\in E^0$, there are no elements $g\ne1_G$ that act trivially on $vE^*$.
\end{itemize}
Then $C^*_r(\G)$ is simple. If the stabilizer of every vertex of $E$ in $G$ is amenable, then conversely, simplicity of $C^*_r(\G)$ implies conditions (i)--(iii).
\end{prop}

\bp
By Corollary~\ref{cor:minimality}, condition (i) is equivalent to minimality of~$\G$. It is a minimal condition that one needs for simplicity of $C^*_r(\G)$, so assume from now on that it is satisfied. Recall from the proof of Corollary~\ref{cor:minimality} that then there are no $G$-breaking vertices. Observe also that in this case the set $E^\sing_0(G)$ coincides with the set $E^\sing_0:=\{v\in E^0: vE^1=\emptyset\}$ of sources in $E$. By Lemma~\ref{lem:M0} we conclude that either $\mathcal{M}_{\infty}(G,E)\cup\LL(G,E)=\{E^0\}$ and $E^\sing_0(G)=\emptyset$, or $\mathcal{M}_{\infty}(G,E)=\LL(G,E)=\emptyset$ and $E^\sing_0(G)=G\cdot w$ for a source $w$.

We claim that then conditions (ii) and (iii) are equivalent to effectiveness of~$\G$. Indeed, if these conditions are satisfied, then the above discussion implies that conditions (i)--(iii) of Proposition~\ref{prop:esscent} are satisfied and therefore $\G$ has essentially central isotropy. Since $\LL(G,E)=\emptyset$ by condition (ii), cases (ii) and (iii) of Theorem~\ref{thm:quasio} imply that there is a path $x\in A_{G,E}$ with dense orbit. By the definition of the set $A_{G,E}$ and essential centrality of the isotropy this means that $\Iso_x=\{x\}$. By density of the orbit of $x$ we then have $\IsoG^\circ_{y}=\{y\}$ for all $y\in \Gu$, so~$\G$ is effective.

Conversely, assume that $\G$ is effective. Since $\G$ is minimal by assumption, it then also has essentially central isotropy. Condition (ii) must hold, since the elements $x_{L}$ for $L\in\LL(G,E)$ satisfy $\Iso_{x_{L}} \neq \{x_{L}\}$. If $E^{0} \in \mathcal{M}_{\infty}(G,E)$, then (iii) is clearly satisfied by Proposition~\ref{prop:esscent}. By the first paragraph of the proof, the only other option is that $E^{0}=\MT_{G}(w)$ for some $w\in E^\sing_0(G)$. By Proposition~\ref{prop:esscent} the vertex $w$ has trivial stabilizer in $G$. Take any vertex $v\in E^0$. Then there exists $h\in G$ such that $vE^{*}(h\cdot w)\ne\emptyset$. Since $h\cdot w$ has trivial stabilizer in $G$, we then conclude that there are no elements $g\ne1_G$ that act trivially on $vE^*$, so condition (iii) is satisfied in this case as well. Thus, our claim is proved.

Now, it is well-known that if $\G$ is minimal and effective, then $C^*_r(\G)$ is simple~\cite{Rbook}*{Proposition~II.4.6}, and if $\G$ is amenable, then the converse is true as well~\cite{MR3189105}*{Theorem~5.1} (note also that Proposition~\ref{prop:gauge} can be viewed as a generalization of these facts to graded groupoids). This proves the proposition.
\ep

\bigskip

\section{Self-similar \texorpdfstring{$k$}{k}-graphs}\label{sec:k-graphs}

In this section we fix a countable self-similar row-finite $k$-graph $(G,\Lambda)$ without sources.

\subsection{The primitive spectrum: first take}
Denote by $\G$ the Exel--Pardo groupoid $\G_{G,\Lambda}$. We proceed similarly to the rank one case. The assumptions of row-finiteness and source-freeness make the first several steps even easier compared to the previous section. In particular, the space~$\partial\Lambda$ coincides with the space $\Lambda^\infty$ of infinite paths, that is, paths $x$ such that $d(x)=(\infty,\dots,\infty)$.

\smallskip

As before, given vertices $v,w\in\Lambda^0$, we write $v\ge w$ if $v\Lambda w\ne\emptyset$. The notion of a maximal tail for row-finite $k$-graphs without sources~\cite{MR3189779} has the following $G$-equivariant extension.

\begin{defn}\label{def:Gmaxhigh}
A nonempty $G$-invariant subset $M \subset \Lambda^{0}$ is called a \emph{maximal $G$-tail} if it satisfies the following conditions:
\begin{enumerate}
\item[(i)] if $v\in \Lambda^{0}$ and $v\ge w$ for some $w\in M$, then $v\in M$;
\item[(ii)] if $v \in M$, then $v\Lambda^{e_{i}}M\ne\emptyset$ for all $1\leq i \leq k$, where $e_1,\dots, e_k$ denote the standard generators of $\Z^k_+$;
\item[(iii)] for every $v,w \in M$, there exist $g\in G$ and $u\in M$ such that $v\ge u$ and $w\ge g\cdot u$.
\end{enumerate}
We denote by $\M(G,\Lambda)$ the set of maximal $G$-tails in $(G,\Lambda)$.
\end{defn}

Define for each $x\in \Lambda^\infty$ a subset $\MT_G(x)$ of $\Lambda^{0}$ by
$$
\MT_G(x):= \{ r(y) : y \in  [ x]_{\mathcal{G}} \}.
$$
Similarly to Lemma~\ref{lem:EP-orbit}, the $\G$-orbit $[x]_\G$ consists of all elements of the form $\mu (g\cdot \sigma^{n}(x))$, with $n\in\Z^k_+$, $g\in G$ and $\mu\in\Lambda(g\cdot x(n))$.

The following proposition completely describes the quasi-orbit structure of the action $\G\curvearrowright\Lambda^\infty$ in terms of maximal $G$-tails, cf.~Theorem~\ref{thm:quasio}.

\begin{prop}\label{prop:max-tail}
Let $(G,\Lambda)$ be a countable self-similar row-finite $k$-graph without sources, and consider the corresponding Exel--Pardo groupoid $\G:=\G_{G,\Lambda}$. Then, for every $x\in\Lambda^\infty$, the set~$\MT_G(x)$ is the smallest maximal $G$-tail containing all vertices of $x$. The map $\Lambda^\infty\to \M(G,\Lambda)$, $x\mapsto\MT_G(x)$, is surjective, and $\MT_G(x)=\MT_G(y)$ for $x,y\in \Lambda^{\infty}$ if and only if $\overline{[x]}_{\mathcal{G}}=\overline{[y]}_{\mathcal{G}}$. Therefore we get a bijection between $(\G\backslash\Lambda^\infty)^\sim$ and $\M(G,\Lambda)$.
\end{prop}

\bp
The proof is similar to the rank one case, see Lemma~\ref{lem:Gtails} and Proposition~\ref{prop:Gmaxt}, but is easier, because we do not have analogues of singular vertices and do not care (for now) about $G$-aperiodicity. So let us only briefly explain why the map $\Lambda^\infty\to \M(G,\Lambda)$ is surjective. Take a maximal $G$-tail $M$.

If there is a minimal element $v\in M$, then using condition (ii) in Definition \ref{def:Gmaxhigh} we can find a path $\mu\in v\Lambda v$ such that $d(\mu)_{i}>0$ for all $i=1, \dots,k$. Then $\MT_G(\mu^{\infty})=M$.

If $M$ does not have a minimal element, let us numerate the elements of $M$ as $u_1,u_2,\dots$. We can inductively construct elements $g_{n} \in G$ and $v_{n} \in M$ such that $u_n\ge g_{n}\cdot v_{n}$ and $v_{n} \Lambda^{l(n)} v_{n+1} \neq \emptyset$ for some $l(n)\in \Z^{k}_+$ with $l(n)_{i}>0$ for all $i=1, \dots, k$. Take any $\mu_{n}\in v_{n} \Lambda^{l(n)} v_{n+1}$ and set $x:=\mu_{1}\mu_{2}\ldots \in \Lambda^{\infty}$. Then $\MT_G(x)=M$.
\ep

It is also not difficult to describe the topology on $(\G\backslash\Lambda^\infty)^\sim$. Namely, essentially the same proof as that of Lemma~\ref{lem:convergence} gives the following result.

\begin{lemma}\label{lem:max-tail-convergence}
In the setting of Proposition~\ref{prop:max-tail}, identify $(\G\backslash\Lambda^\infty)^\sim$ with $\M(G,\Lambda)$. Then the topology $\M(G,\Lambda)$ is described as follows: a sequence $(M_n)^\infty_{n=1}$ in $\M(G,\Lambda)$ converges to $M$ if and only if, for every vertex $v\in M$, we have $v\in M_n$ for all $n$ large enough.
\end{lemma}

There is a more or less known description of the lattice of open subsets of $(\G\backslash\Lambda^\infty)^\sim$ (equivalently, the lattice of $\G$-invariant open subsets of $\Lambda^\infty$) that does not explicitly use maximal $G$-tails. Recall that a subset $H\subset\Lambda^0$ is called \emph{hereditary} if whenever $v\ge w$ and $v\in H$, we must have $w\in H$. It is called \emph{saturated} if whenever $s(v\Lambda^{n})\subset H$ for some $n\in \Z_+^{k}$ and $v\in \Lambda^{0}$, we must have $v\in H$.

\begin{prop}[cf.~\citelist{\cite{RSY}*{Theorem~5.2}\cite{MR4283280}*{Theorem~5.5}\cite{MR4887755}*{Section~11}}]
Let $(G,\Lambda)$ be a self-similar row-finite $k$-graph without sources, and consider the corresponding Exel--Pardo groupoid $\G:=\G_{G,\Lambda}$. Then there is a one-to-one correspondence between the $\G$-invariant open subsets of~$\Lambda^\infty$ and the $G$-invariant hereditary and saturated subsets of $\Lambda^0$. Namely, given a $\G$-invariant open subset $\Omega\subset\Lambda^\infty$, we define
$$
H_\Omega:=\{v\in\Lambda^0: v\Lambda^\infty\subset \Omega\},
$$
and given a $G$-invariant hereditary and saturated subset $H\subset\Lambda^0$, we define
$$
\Omega_H:=\{x\in\Lambda^\infty: x(n)\in H\ \text{for some}\ n\in\Z_+^k\}.
$$
Then the maps $\Omega\mapsto H_\Omega$ and $H\mapsto \Omega_H$ are inverse to each other.
\end{prop}

\bp
It is easy to see that if $\Omega\subset\Lambda^\infty$ is  a $\G$-invariant open set, then $H_\Omega$ is $G$-invariant, hereditary and saturated, and if $H\subset\Lambda^0$ is $G$-invariant, hereditary and saturated, then $\Omega_H$ is open and $\G$-invariant. It is also straightforward to show that $\Omega_{(H_\Omega)}=\Omega$ and $H\subset H_{(\Omega_H)}$. Therefore in order to establish the desired bijection we only have to argue that $H_{(\Omega_H)} \subset H$.

Take a vertex $v\in H_{(\Omega_H)}$, so that $v\Lambda^\infty \subset \Omega_H$. Then, for every $x\in v\Lambda^\infty$, there is $n_{x} \in \Z^k_+$ such that $x(n_x)\in H$. By compactness of $v\Lambda^\infty$ we can find $x_{1}, \dots , x_{l}\in v\Lambda^\infty$ such that $v\Lambda^\infty =\bigcup_{i=1}^{l} x_i(0, n_{x_{i}})\Lambda^\infty$.
Put $n:=n_{x_1}\vee\dots\vee n_{x_l}$. Using that $H$ is hereditary, we conclude that $s(v\Lambda^n)\subset H$, and since $H$ is saturated, we get $v\in H$. Thus, $H_{(\Omega_H)} \subset H$.
\ep

\begin{remark}
If $(G,\Lambda)$ is pseudo-free, we conclude that the lattice of dynamical ideals in $C^*_r(\G_{G,\Lambda})$ is isomorphic to the lattice of $G$-invariant hereditary and saturated subsets of $\Lambda^0$. If, in addition, the $G$-stabilizers of the vertices are amenable, so that $\G_{G,\Lambda}$ is amenable by Theorem~\ref{thm:EP-amenability}, then the dynamical ideals coincide with the diagonal-invariant ones by Proposition~\ref{prop:diagonal}, and therefore this isomorphism generalizes~\cite{MR4283280}*{Theorem 5.5}.\footnote{Note that although there are no explicit amenability assumptions in~\cite{MR4283280}*{Theorem 5.5}, the proof of this theorem relies on \cite{MR4283280}*{Theorem 4.1}, which, in turn, is proved under the assumption that $G$ is amenable.} This shows in particular that the assumption of gauge-invariance in~\cite{MR4283280} is redundant.\ee
\end{remark}

As we saw already in the rank once case, an appropriate equivariant version of graph periodicity plays an important role in the analysis of the primitive spectrum. Under certain assumptions the necessary definitions pertaining to periodicity of self-similar $k$-graphs have been given by Li and Yang in~\cite{LY19}, see Definition~3.1 and Section~4.1 there. As we will see, the assumptions in~\cite{LY19} can be relaxed to cover the case we are interested in.

\begin{defn}[cf.~\cite{LY19}]
A \emph{cycline triple} $(\mu, g, \nu)$ in $(G, \Lambda)$ consists of paths $\mu, \nu \in \Lambda$ and an element $g\in G$ such that $g\cdot s(\nu)=s(\mu)$ and $\mu(g\cdot x)=\nu x$ for all $x\in s(\nu)\Lambda^\infty$. Assuming that~$\Lambda^0$ is a maximal $G$-tail, the \emph{$G$-periodicity group} of $\Lambda$ is defined as
$$
\PerG(\Lambda):=\{ d(\mu)-d(\nu) : (\mu, g, \nu) \text{ is a cycline triple in } (G, \Lambda)\}\subset\Z^k.
$$
\end{defn}

\begin{remark}\label{rem:57}
By erasing the same initial part of $\mu$ and $\nu$ we can always assume that a cycline triple $(\mu,g,\nu)$ satisfies $d(\mu)\wedge d(\nu)=0$. In the rank one case such cycline triples are closely related to $G$-circuits without an entry in $\Lambda^0$. Namely, if $(\mu,g,v)$ is a cycline triple for some $\mu\in\Lambda$ of nonzero length and $v\in\Lambda^0$, then $(g,\mu)$ is a $G$-circuit. Then, by Lemma~\ref{lem:G-circuits}, the equation $\mu(g\cdot x)=x$ in $x\in v\Lambda^\infty$ has only one solution, which then necessarily has no entry in~$\Lambda^0$. It follows that the $G$-circuit $(g,\mu)$ has no entry. Conversely, if $(g,\mu)$ is a $G$-circuit without an entry, then $(\mu,g,g^{-1}\cdot s(\mu))$ is a cycline triple. Similarly, given a path $\nu\in\Lambda$ of nonzero length, $(g\cdot s(\nu),g,\nu)$ is a cycline triple if and only if $(g^{-1},\nu)$ is a $G$-circuit without an entry. It follows that if $\Lambda^0$ is a maximal $G$-tail and there are no $G$-circuits without an entry in~$\Lambda^0$, then $\PerG(\Lambda)=0$, while if there are such $G$-circuits, then, by Lemma~\ref{lem:circuits-wo-entry}, we have $\PerG(\Lambda)=n_{\Lambda^0}\Z$, where $n_{\Lambda^0}$ is the $G$-period of $\Lambda^0$ introduced in Definition~\ref{def:LGE}. \ee
\end{remark}

\begin{prop}\label{prop:perG}
Let $(G,\Lambda)$ be a countable self-similar row-finite $k$-graph without sources such that $\Lambda^0$ is a maximal $G$-tail. Consider the corresponding Exel--Pardo groupoid $\G:=\G_{G,\Lambda}$ with its standard grading $\Phi\colon\G\to\Z^k$. Then, for any $x\in\Lambda^\infty$ such that $\MT_G(x)=\Lambda^0$, we have $\Phi(\Iso_x)=\PerG(\Lambda)$. In particular, $\PerG(\Lambda)$ is a subgroup of $\Z^k$.
\end{prop}

\bp
Assume $(\mu, g, \nu)$ is a cycline triple in $(G, \Lambda)$. Take an element $y\in [x]_\G$ with $r(y)=s(\nu)$. Then
$$
(\mu( g\cdot y) , \mathcal{T}_{d(\mu)}([g|_{y}]), d(\mu)-d(\nu), \nu y)  \in \Iso_{\nu y},
$$
and we obtain that $ d(\mu)-d(\nu) \in \Phi(\Iso_{\nu y})$. Since $\Phi(\Iso_{\nu y})= \Phi(\Iso_{x})$, this proves that $\PerG(\Lambda) \subset \Phi(\Iso_{x})$.
On the other hand, any element of $\Phi(\Iso_{x})$ has the form $d(\mu)-d(\nu)$ such that $Z(\mu, g, \nu)\cap\G \subset \Iso$ for some~$g \in G$. By definition this implies that $(\mu, g, \nu)$ is a cycline triple, proving that $\Phi(\Iso_{x}) \subset \PerG(\Lambda)$. Therefore $\Phi(\Iso_{x})=\PerG(\Lambda)$. 
\ep

Note that $\G$ might be non-Hausdorff, but this does not affect the above proof.

\begin{remark}\label{rem:perg-x}
Let us give yet another interpretation of $\PerG(\Lambda)$, which will be useful later. Assume again that $x\in\Lambda^\infty$ is such that $\MT_G(x)=\Lambda^0$. Consider the set $P_x$ of elements $m-n$ such that $m,n\in\Z^k_+$ and $\sigma^m(y)=g\cdot \sigma^n(y)$ for some $g\in G$ and all $y\in\Lambda^\infty$ in a neighbourhood of $x$. We claim that $\PerG(\Lambda)=P_x$.

Observe first that $P_x$ does not change if we replace $x$ by another element on the same $\G$-orbit. Indeed, it suffices to check that $P_x\subset P_{\lambda x}$ for $\lambda\in\Lambda r(x)$, $P_{h\cdot x}\subset P_x$ for $h\in G$, and $P_x\subset P_{\sigma^l(x)}$ for $l\in\Z^k_+$. The first inclusion is immediate, since if $\sigma^m(y)=g\cdot \sigma^n(y)$, then $\sigma^{m+d(\lambda)}(\lambda y)=g\cdot \sigma^{n+d(\lambda)}(\lambda y)$. The second inclusion follows from the fact that an identity of the form $\sigma^m(h\cdot y)=g\cdot \sigma^n(h\cdot y)$ implies that
$$
\sigma^m(y)=\big((h|_{y(0,m)})^{-1}g(h|_{y(0,n)})\big)\cdot \sigma^n(y).
$$
Similarly, the third inclusion follows from the fact that an identity of the form $\sigma^m(y)=g\cdot \sigma^n(y)$ implies that
$$
\sigma^m(\sigma^l(y))=\sigma^l(g\cdot\sigma^n(y))=(g|_{y(n,n+l)})\cdot\sigma^n(\sigma^l(y)).
$$

Now, assume $(\mu, g, \nu)$ is a cycline triple in $(G, \Lambda)$.
Then $\sigma^{d(\mu)}(y)=g\cdot \sigma^{d(\nu)}(y)$ for all $y\in \nu\Lambda^\infty$.
Taking any element $y\in [x]_\G\cap \nu\Lambda^\infty$, we conclude that $d(\mu)-d(\nu) \in P_y=P_x$. Thus, $\PerG(\Lambda)\subset P_x$. For the opposite inclusion, assume $\sigma^m(y)=g\cdot \sigma^n(y)$ for some $g\in G$ and all $y\in\Lambda^\infty$ in a neighbourhood $U$ of $x$. Choose $\lambda\in\Lambda$ such that $N:=d(\lambda)\ge m\vee n$ and $\lambda\Lambda^\infty\subset U$. Then for all $z\in s(\lambda)\Lambda^\infty$ we have
$\sigma^m(\lambda z)=\lambda(m,N)z$ and
$$
g\cdot \sigma^n(\lambda z)=g\cdot (\lambda(n,N)z)=\big(g\cdot (\lambda(n,N))\big)(g|_{\lambda(n,N)}\cdot z).
$$
Since $\sigma^m(\lambda z)=g\cdot \sigma^n(\lambda z)$, we see that $(g\cdot (\lambda(n,N)),g|_{\lambda(n,N)},\lambda(m,N))$ is a cycline triple and therefore
$(N-n)-(N-m)\in\PerG(\Lambda)$. Hence $P_x\subset \PerG(\Lambda)$.

Let us also note that it is not difficult to see directly that $P_x$ is a group, so we get a proof of the fact that $\PerG(\Lambda)$ is a group that does not use the topology on $\G$. Indeed, it is obvious that $P_x$ contains $0$ and is closed under inversion. In order to see that it is closed under addition, assume $\sigma^m(y)=g\cdot \sigma^n(y)$ and $\sigma^p(y)=h\cdot \sigma^q(y)$ for all $y\in\Lambda^\infty$ sufficiently close to $x$. Then, for all such~$y$, we have
$$
\sigma^{p+m}(y)=\sigma^p(g\cdot\sigma^n(y))=(g|_{y(n,n+p)})\cdot\sigma^n(h\cdot\sigma^q(y))
=\big((g|_{y(n,n+p)})(h|_{y(q,q+n)})\big)\cdot\sigma^{q+n}(y),
$$
hence $(p+m)-(q+n)\in P_x$. \ee
\end{remark}

In general, when $\Lambda^0$ is not necessarily a maximal $G$-tail, for any $M\in\M(G,\Lambda)$ we can consider the self-similar $k$-graph $(G,\Lambda M)$ and conclude that $\PerG(\Lambda M)\subset\Z^k$ is a group. The Exel--Pardo groupoid of $(G,\Lambda M)$ has the following description.

\begin{lemma}\label{lem:identres}
Let $(G,\Lambda)$ be a countable self-similar row-finite $k$-graph without sources with the Exel--Pardo groupoid $\G:=\G_{G,\Lambda}$. Let $M\in\M(G,\Lambda)$. Then, for any $x\in\Lambda^\infty$ such that $\MT_G(x)=M$, the Exel--Pardo groupoid of $(G,\Lambda M)$ can be identified with $\G_{\overline{[x]}}$, where $[x]:=[x]_\G$. If $(G,\Lambda)$ is pseudo-free, then this identification is a homeomorphism.
\end{lemma}

\bp
In order to prove the lemma it suffices to notice that a path $y\in \Lambda^{\infty}$ satisfies $y\in \overline{[x]}$ if and only if $\MT_G(y)\subset\MT_G(x)$, that is, if and only if $y\in(M\Lambda)^\infty$. This can be easily checked directly or by invoking Lemma~\ref{lem:max-tail-convergence}.
\ep

For $M\in\M(G,\Lambda)$ we conclude from Proposition \ref{prop:perG} and Lemma \ref{lem:identres} that for any $x\in\Lambda^\infty$ such that $\MT_G(x)=M$ we have $\Phi(\IsoGx{x}^\circ_x)=\PerG(\Lambda M)$. By Remark~\ref{rem:perg-x} we see also that $\PerG(\Lambda M)$ coincides with the set of differences $m-n$ (with $m,n\in\Z^k_+$) for which there exists $g\in G$ such that $\sigma^m(y)=g\cdot \sigma^n(y)$ for all $y\in\overline{[x]}_\G$ in some neighbourhood of~$x$.

\begin{prop} \label{prop:esscentriso}
Let $(G,\Lambda)$ be a countable pseudo-free self-similar row-finite $k$-graph without sources. Then the Exel--Pardo groupoid $\G:=\G_{G,\Lambda}$, with its standard grading $\Phi\colon \G\to\Z^{k}$, has essentially central isotropy if and only if the following condition is satisfied:
\begin{equation}\label{itcentral2}
\text{for all }M\in\M(G,\Lambda)\text{ and }v\in M,\text{ there are no }g\ne1_G\text{ that act trivially on }v\Lambda M.
\end{equation}
\end{prop}


\bp
The proof is similar to that of Proposition~\ref{prop:esscent}. Briefly, one shows first that $\Phi$ is injective on $\IsoGx{x}^\circ_x$ if and only if there are no $g\ne1_G$ and $m\in\Z^k_+$ such that $g$ acts trivially on a neighbourhood of $\sigma^m(x)$ in $\overline{[x]}$. Then one checks that the last condition is equivalent to~\eqref{itcentral2}, see the last two paragraphs of the proof of Proposition~\ref{prop:esscent}.
\ep

We are now ready to get our first description of the primitive spectrum.

\begin{thm}\label{thm:main-higher-rank1}
Let $(G,\Lambda)$ be a countable pseudo-free self-similar row-finite $k$-graph without sources, and consider the corresponding Exel--Pardo groupoid $\G:=\G_{G,\Lambda}$. Assume that the stabilizer of every vertex of $\Lambda$ in $G$ is amenable and condition~\eqref{itcentral2} is satisfied. Then the groupoid $\G$ is amenable and $\Prim C^*(\G)$ can be identified with the set of pairs $(M,\chi)$, where $M\in\M(G,\Lambda)$ and $\chi\in\PerG(\Lambda M)\dach$. Namely, the primitive ideal corresponding to $(M,\chi)$ is $\ker\pi_{(x,z)}$, where $x\in\Lambda^\infty$ is any path with $\MT_G(x)=M$ and $z\in\T^k$ is any character such that $z|_{\PerG(\Lambda M)}=\chi$.

Under this identification the topology on $\Prim C^*(\G)$ is described as follows. Assume we are given $(M,\chi)$ and $(M_n,\chi_n)$ ($n\in\N$) as above. Fix $x,x_n\in\Lambda^\infty$ and $z,z(n)\in\T^k$ such that $\MT_G(x)=M$, $\MT_G(x_n)=M_n$, $z|_{\PerG(\Lambda M)}=\chi$ and $z(n)|_{\PerG(\Lambda M_n)}=\chi_n$. For every $l\in\PerG(\Lambda M)$, choose $(p,g,q)\in\Z^k_+\times G\times\Z^k_+$ such that $p-q=l$ and $\sigma^p(y)=g\cdot \sigma^q(y)$ for all $y\in\overline{[x]}_\G$ in some neighbourhood of~$x$, and denote by $\Sigma$ the set of triples $(p,g,q)$ we thus get.
Then $(M_n,\chi_n)\to(M,\chi)$ if and only if there exist $y_n\in[x_n]_\G$ such that $y_n\to x$ and $z(n)\to z$ along the sets
$$
T_n:=\{p-q: (p,g,q)\in\Sigma,\ p-q\in\PerG(\Lambda M_n),\ \sigma^p(y_n)=g\cdot \sigma^q(y_n)\}.
$$
\end{thm}

\bp
The groupoid $\G$ is amenable by Theorem~\ref{thm:EP-amenability}, and by Proposition~\ref{prop:esscentriso} it has essentially central isotropy. The claimed parameterization of the primitive ideals follows then from Theorem~\ref{thm:essentially-central} (recall also Corollary~\ref{cor:differen-reps} and the discussion following it) and the description of the essential isotropy groups provided by Proposition~\ref{prop:perG}.

In order to describe the topology on $\Prim C^*(\G)$ we want to apply Theorem~\ref{thm:convfg}. Let us write~$[y]$ for~$[y]_\G$. Assume we are given $(M_n,\chi_n)$ and $(M,\chi)$ and choose $x$, $x_n$, $z$, $z(n)$ and $\Sigma$ as in the statement of the theorem. For every $l\in\PerG(\Lambda M)$, let us define an open bisection~$W_l$ of~$\G$ as follows. Let $(p,g,q)\in\Sigma$ be the triple corresponding to $l$. Define $\mu:=x(0,p)$ and $\nu:=x(0,q)$, and put $W_l:=Z(\mu,g,\nu)\cap\G$. Then, for any $y\in\Lambda^\infty$, we have $W_l\cap\G^y_y\ne\emptyset$ if and only if $y\in\mu\Lambda^\infty\cap\nu\Lambda^\infty$ and $\sigma^p(y)=g\cdot \sigma^q(y)$. It follows that once we are given a sequence of paths $y_n\in[x_n]$ converging to $x$, then $z(n)\to z$ along the sets~$T_n$ in the formulation of theorem if and only if $z(n)\to z$ along the sets
$$
\tilde T_n:=\{l\in\PerG(\Lambda M): l\in\PerG(\Lambda M_n),\ W_l\cap\G^{y_n}_{y_n}\ne\emptyset\}.
$$
On the other hand, consider the sets $R_n$ and $S_n$ from  Theorem~\ref{thm:convfg} defined for our choice of bisections, so
$$
R_n=\{l\in\PerG(\Lambda M): W_l\cap\IsoGx{x_n}^\circ_{y_n}\ne\emptyset\},\quad
S_n=\{l\in\PerG(\Lambda M): W_l\cap\G^{y_n}_{y_n}\ne\emptyset\}.
$$
If $W_l\cap\IsoGx{x_n}^\circ_{y_n}\ne\emptyset$ for some $n$, then $l\in\PerG(\Lambda M_n)$ by Proposition~\ref{prop:perG}. Therefore for any choice of $y_n\in[x_n]$ we have $R_n\subset\tilde T_n\subset S_n$. By Theorem~\ref{thm:convfg} we conclude that $(M_n,\chi_n)\to(M,\chi)$ if and only if there exist $y_n\in[x_n]$ such that $y_n\to x$ and $z(n)\to z$ along the sets~$\tilde T_n$, equivalently, along the sets~$T_n$.
\ep

Similarly to the rank one case, Proposition~\ref{prop:central-necessity} implies that without condition~\eqref{itcentral2} the primitive ideal space is certainly more complicated and cannot be described without taking into account contributions of the $G$-stabilizers of vertices into the isotropy groups.

\smallskip

The simplest situation where the above theorem quickly leads to a complete description of the primitive spectrum is given in the following corollary. When $G$ acts trivially on the vertices of $\Lambda$, it recovers~\cite{MR4283280}*{Theorem~7.1}.

\begin{cor}
In the setting of Theorem~\ref{thm:main-higher-rank1} assume in addition that the only maximal $G$-tail in $(G,\Lambda)$ is $\Lambda^0$. Then $\Prim C^*(\G_{G,\Lambda})$ is homeomorphic to $\PerG(\Lambda)\dach$.
\end{cor}

It is not difficult to see that in the rank one case Theorem~\ref{thm:main-higher-rank1} is equivalent to Theorem~\ref{thm:main-rank-one} (for row-finite graphs without sources). But for $k\ge2$ it is not equally satisfactory, since it is not obvious how to formulate it only in terms of finite paths and maximal $G$-tails. The issue is that for $k=1$ the equation $\sigma^p(y)=g\cdot\sigma^q(y)$ in $y$ has a discrete set of solutions for $p\ne q$, but this is no longer true for $k\ge2$ and, as a result, to define the sets $T_n$ one seemingly needs to know the entire infinite paths $y_n$. In order to get a stronger result, we need a better understanding of $G$-periodicity.
But before we turn to this, let us finish this subsection with a criterion for simplicity.

\begin{prop}[cf.~\cite{MR4283280}*{Theorem~4.3(iii)}]
Let $(G,\Lambda)$ be a countable pseudo-free self-similar row-finite $k$-graph without sources, and consider the corresponding Exel--Pardo groupoid $\G:=\G_{G,\Lambda}$.
Assume that the following conditions are satisfied:
\begin{itemize}
  \item[(i)] the only maximal $G$-tail is $\Lambda^0$;
  \item[(ii)] the $G$-periodicity group $\PerG(\Lambda)$ is trivial;
  \item[(iii)] for every vertex $v\in \Lambda^0$, there are no elements $g\ne1_G$ that act trivially on $v\Lambda$.
\end{itemize}
Then $C^*_r(\G)$ is simple. If the stabilizer of every vertex of $\Lambda$ in $G$ is amenable, then conversely, simplicity of $C^*_r(\G)$ implies conditions (i)--(iii).
\end{prop}

\bp
The proof is similar to that of Proposition~\ref{prop:simple-rank-one}: by Proposition~\ref{prop:max-tail}, condition (i) is equivalent to minimality of the groupoid $\G$, and once this condition is satisfied, conditions (ii) and (iii) are equivalent to effectiveness of the groupoid $\G$ by Propositions~\ref{prop:perG} and~\ref{prop:esscentriso}.
\ep

\subsection{The primitive spectrum in terms of cycline triples and maximal \texorpdfstring{$G$}{G}-tails}
We will need a $G$-equivariant version of the sets $H_{\operatorname{Per}}$ introduced by Carlsen et al.~in \cite{MR3150171}*{Section~4}.

\begin{defn}
Given a maximal $G$-tail $M \in \mathcal{M}(G,\Lambda)$, denote by $M_{\PerG}$ the set of vertices $v\in M$ such that for every $\mu \in v\Lambda M$ and $m\in \mathbb{Z}_{+}^{k}$ with $d(\mu)-m\in \PerG(\Lambda M)$ there exist a path $\nu \in \Lambda ^{m}M$ and an element $g\in G$ such that $(\mu, g, \nu)$ is a cycline triple in $\Lambda M$.
\end{defn}

\begin{remark}
If $\PerG(\Lambda M)$ is the trivial group, then the cycline triples $(\mu, 1_{G}, \mu)$ for $\mu \in v\Lambda M$ guarantee that $M_{\PerG}=M$. Consider now the rank one case.
If $M\in\M_\gamma(G,\Lambda)=\M_\infty(G,\Lambda)$, then $\PerG(\Lambda M)=0$ by Remark~\ref{rem:57} and therefore $M_{\PerG}=M$. We claim that if $M\in\LL(G,\Lambda)$, then $M_{\PerG}$ consists of all vertices of $G$-circuits without an entry in~$M$. If $v\in M_{\PerG}$, then, since $\PerG(\Lambda M)=n_M\Z$ by Remark~\ref{rem:57}, there is a cycline triple $(v, g, \nu)$ with $d(\nu)=n_{M}$, and hence $(g^{-1}, \nu)$ is a $G$-circuit without an entry by the same remark. Conversely, assume $v$ is a vertex of a $G$-circuit without an entry in $M$, $\mu \in v\Lambda M$ and $d(\mu)-m\in \PerG(\Lambda M)=n_M\Z$. Let $x\in(\Lambda M)^\infty$ be the unique path with range $v$. Then $\mu=x_1\dots x_{|\mu|}$. By Remark~\ref{rem:minimal-period} we have $\sigma^{n_M}(x)=g\cdot x$ for some $g\in G$. Since $d(\mu)-m$ is divisible by $n_M$, it follows that we can find $h\in G$ such that $\sigma^{d(\mu)}(x)=h\cdot\sigma^m(x)$. Since $x$ has no entry in $M$, this means that $(\mu,h,x_1\dots x_m)$ is a cycline triple in $\Lambda M$, which proves our claim.\ee
\end{remark}

Clearly, to establish general properties of the sets $M_{\PerG}$ it suffices to consider $M=\Lambda^0$ when $\Lambda^0$ is itself a maximal $G$-tail. The following result extends the main part of \cite{MR3150171}*{Theorem~4.2} to self-similar $k$-graphs.

\begin{prop}\label{prop:mperg}
Assume $(G,\Lambda)$ is a countable self-similar row-finite $k$-graph without sources such that~$\Lambda^0$ is a maximal $G$-tail. Then $\Lambda^0_{\PerG}$ is a nonempty $G$-invariant hereditary subset of~$\Lambda^0$.
\end{prop}

The proof is largely the same as in \cite{MR3150171}, but since the $G$-action makes some arguments slightly trickier, we provide full details.

\smallskip

For $v\in\Lambda^0$, denote by $\Sigma_{v}$ the set of pairs $(p,q) \in \mathbb{Z}_{+}^{k}\times \mathbb{Z}_{+}^{k}$ such that for every $\mu \in v\Lambda^{p}$  there exist $g\in G$ and $\nu \in \Lambda^{q}$ such that $(\mu, g, \nu)$ is a cycline triple in $\Lambda$. Note that by definition if $(p,q)\in\Sigma_v$, then $p-q\in\PerG(\Lambda)$. We define $\Sigma := \bigcup_{v\in \Lambda^0} \Sigma_{v}$.

\begin{lemma}\label{lem:monoids}
The following properties hold:
\begin{enumerate}
\item $\Sigma_{v}=\Sigma_{h\cdot v}$ for all $v\in\Lambda^0$ and $h\in G$;
\item if $\lambda\in\Lambda$, $(p,q) \in \mathbb{Z}_{+}^{k}\times \mathbb{Z}_{+}^{k}$ and $(p+d(\lambda), q+d(\lambda))\in \Sigma_{r(\lambda)}$, then $(p,q)\in \Sigma_{s(\lambda)}$;
\item if $(p,q)\in \Sigma_{v}$, then $(p+n, q+n) \in \Sigma_{v}$ for all $n\in \mathbb{Z}_{+}^{k}$;
\item $ \Sigma_{r(\lambda)} \subset \Sigma_{s(\lambda)}$ for all $\lambda \in \Lambda$;
\item if $(p,q)\in\Sigma$, $(p',q')\in\Z^k_+\times\Z^k_+$ and $p-q=p'-q'$, then $(p',q')\in\Sigma$;
\item the sets $\Sigma$ and $\Sigma_{v}$ are submonoids of $\Z^{2k}_+$;
\item if $(p,q), (r,s) \in \Sigma$ satisfy $(p,q) \geq  (r,s)$, then $(p-r,q-s) \in \Sigma$;
\item the monoid $\Sigma$ is finitely generated.
\end{enumerate}
\end{lemma}

\begin{proof}
(1) It suffices to show that $\Sigma_{v}\subset\Sigma_{h\cdot v}$. Take $(p,q) \in \Sigma_{v}$ and $\mu \in (h\cdot v)\Lambda^{p} $. Since $h^{-1} \cdot \mu \in v\Lambda^{p} $, there are $\nu \in \Lambda^{q} $ and $g\in G$ such that $( h^{-1} \cdot \mu, g, \nu)$ is a cycline triple in $\Lambda$. We claim that $( \mu, h|_{h^{-1} \cdot \mu} g (h|_{\nu})^{-1}, h\cdot \nu)$ is a cycline triple in $\Lambda$, which implies that $(p,q)\in \Sigma_{h\cdot v}$. In order to prove the claim, take $x\in s(h\cdot\nu)\Lambda^\infty$. Since $s(h\cdot \nu)=h|_{\nu} \cdot s(\nu)$, we then get that $(h|_{\nu})^{-1} \cdot x \in s(\nu)\Lambda^\infty$, and hence $ (h^{-1} \cdot \mu) \big(g\cdot( (h|_{\nu})^{-1} \cdot x)\big)= \nu ((h|_{\nu})^{-1} \cdot x)$. Applying $h$ on both sides we obtain $\mu \big((h|_{h^{-1} \cdot \mu}g  (h|_{\nu})^{-1}) \cdot x\big)= (h\cdot \nu)  x$, proving that $( \mu, h|_{h^{-1} \cdot \mu} g (h|_{\nu})^{-1}, h\cdot \nu)$ is indeed a cycline triple in $\Lambda $.

\smallskip

(2) Take a path $\mu \in s(\lambda)\Lambda^{p} $. Since $(d(\lambda \mu), q+d(\lambda))\in \Sigma_{r(\Lambda)}$ by assumption, there exist  $\nu \in \Lambda ^{q+d(\lambda)}$ and $g\in G$ such that $(\lambda \mu, g, \nu)$ is a cycline triple in $\Lambda $. Then $(\mu, g, \sigma^{d(\lambda)}(\nu))$ is also a cycline triple in $\Lambda $, and we conclude that $(p,q)\in \Sigma_{s(\lambda)}$.

\smallskip

(3) Take a path $\lambda \in v\Lambda^{p+n} $. Then by assumption there are $g\in G$ and $\mu \in \Lambda^{q}$ such that $(\lambda(0,p), g, \mu)$ is a cycline triple. For $z\in \lambda(p)\Lambda ^{\infty}$ this implies that $\lambda(0,p) z = \mu (g^{-1}\cdot  z)$, so for $y\in s(\lambda)\Lambda^{\infty}$ we get
$$
\lambda y=\mu \big(g^{-1} \cdot (\lambda(p,p+n) y)\big) = \mu \big(g^{-1} \cdot (\lambda(p,p+n))\big) (g^{-1}|_{\lambda(p,p+n)} \cdot y),
$$
and letting $h:=(g^{-1}|_{\lambda(p,p+n)})^{-1}$ and $\nu:=\mu \big(g^{-1} \cdot (\lambda(p,p+n))\big)$ we conclude that $(\lambda , h, \nu)$ is a cycline triple in $\Lambda $. This proves that $(p+n, q+n) \in \Sigma_{v}$.

\smallskip

(4) This follows by combining (2) and (3).

\smallskip

(5) Let us show first that if $(p,q)\in\Sigma$ and $0\le n\le p\wedge q$, then $(p-n,q-n)\in\Sigma$. Take $v\in\Lambda^0$ such that $(p,q)\in\Sigma_v$ and choose any path $\lambda\in v\Lambda^n$. Then by (2) we conclude that $(p-n,q-n)\in\Sigma_{s(\lambda)}$.

Now consider the general case, so assume that $(p,q)\in\Sigma$, $(p',q')\in\Z^k_+\times\Z^k_+$ and $p-q=p'-q'$. Let $n:=p'-p$ and define $n_+:=n\vee 0$ and $n_-:=-(n\wedge 0)$. Then
$$
(p',q')=(p,q)+(n_+,n_+)-(n_-,n_-)\in\Sigma
$$
by (3) and the particular case we started with.

\smallskip

(6) We will  first show that $\Sigma_{v}$ is a monoid for every $v\in \Lambda^{0}$. Take $(p,q), (m,n) \in \Sigma_{v}$ and a path $\mu \in v\Lambda^{p+m} $. By (4) we have $(m,n) \in \Sigma_{s(\mu(0,p))}$, so there exist $g\in G$ and $\nu \in \Lambda^{n}$ such that $(\mu(p, p+m), g, \nu)$ is a cycline triple. Similarly, there are $h\in G$ and $\lambda \in \Lambda^{q}$ such that $(\mu(0,p) , h, \lambda)$ is a cycline triple. For any $x\in s(\mu)\Lambda^{\infty}$ we then get that
\begin{multline*}
\mu x =\mu(0,p) \mu(p, p+m) x= \lambda \big( h^{-1}\cdot  (\mu(p, p+m) x)\big)
 \\
=\lambda \big( h^{-1}\cdot  (\nu (g^{-1} \cdot x))\big)=\lambda (h^{-1}\cdot  \nu) \big((h^{-1}|_{\nu} g^{-1}) \cdot x\big),
\end{multline*}
and hence $(\mu, (h^{-1}|_{\nu} g^{-1})^{-1}, \lambda (h^{-1}\cdot  \nu) )$ is a cycline triple, proving that $(p+m, q+n) \in \Sigma_{v}$. Since we also obviously have $(0,0) \in \Sigma_{v}$, we conclude that $\Sigma_{v}$ is a monoid.

Combining (1) and (4) we see that if $v\ge g\cdot u$, then $\Sigma_v\subset\Sigma_u$. Since $\Lambda^0$ is a maximal $G$-tail, it follows that $\Sigma$ is a monoid as well.

\smallskip

(7) Choose a vertex $v\in \Lambda^{0}$ such that $(p,q), (r,s) \in \Sigma_{v}$. Fix a path $\mu \in v\Lambda^{r}$. Our aim is to show that $(p-r,q-s)\in\Sigma_{s(\mu)}$. By assumption there are $g\in G$ and a $\nu \in \Lambda^{s}$ such that $(\mu, g, \nu)$ is a cycline triple. Take now any $\lambda \in s(\mu)\Lambda^{p-r}$. Then $\mu\lambda \in v\Lambda^{p}$, so there are $h\in G$ and $\eta \in \Lambda^{q}$ such that $(\mu\lambda, h, \eta)$ is a cycline triple in $\Lambda $. We now have for all $x\in s(\lambda)\Lambda ^{\infty}$ that
$$
\eta (h^{-1} \cdot x) =\mu \lambda x = \nu (g^{-1} \cdot (\lambda x)),
$$
hence $g^{-1} \cdot (\lambda x) = \sigma^{d(\nu)}(\eta) (h^{-1} \cdot x)$, and applying $g$ on both sides we obtain
$$
\lambda  x = (g\cdot\sigma^{d(\nu)}(\eta)) ((g|_{\sigma^{d(\nu)}(\eta)}h^{-1}) \cdot x).
$$
It follows that $(\lambda, h (g|_{\sigma^{d(\nu)}(\eta)})^{-1}, g\cdot\sigma^{d(\nu)}(\eta) )$ is a cycline triple in $\Lambda $. Hence $(p-r,q-s) \in \Sigma_{s(\mu)}$.

\smallskip

(8) This follows now by the same argument as in \cite{MR3150171}*{Proposition 4.4}: the set of minimal elements in $\Sigma\setminus\{0\}\subset\Z^{2k}_+$ for the order $\leq$ is finite by Dickson's lemma \cite{MR1694173}*{Theorem 5.1}, and then (7) and a simple induction argument show that this set generates $\Sigma$ as a monoid.
\end{proof}

\bp[Proof of Proposition~\ref{prop:mperg}]
Since the monoid $\Sigma$ is finitely generated and $\Sigma_v\subset\Sigma_u$ when $v\ge g\cdot u$, there exists a vertex $v$ such that $\Sigma=\Sigma_v$. We claim that then $v\in\Lambda^0_{\PerG}$. By the definition of~$\Sigma_v$ we need to show that if $p,q \in \mathbb{Z}_{+}^{k}$ and $p-q \in \PerG(\Lambda)$, then $(p, q) \in \Sigma_v$.

Since $q-p\in\PerG(\Lambda)$, there exists a cycline triple $(\mu, g, \nu)$ in $\Lambda$ such that $q-p=d(\mu)-d(\nu)$. We want to show that $(d(\nu),d(\mu))\in\Sigma_{s(\nu)}$. Take any path $\lambda \in s(\nu)\Lambda^{d(\nu)}$. Then, for any $x\in s(\lambda)\Lambda^{\infty}$, we have $\mu (g\cdot  (\lambda x)) = \nu \lambda x$, and hence $\sigma^{d(\nu)}(\mu (g\cdot  \lambda)) (g|_{\lambda} \cdot x) = \lambda x$, showing that $(\lambda , (g|_{\lambda})^{-1},  \sigma^{d(\nu)}(\mu (g\cdot  \lambda)) )$ is a cycline triple. Hence $(d(\nu),d(\mu))\in \Sigma_{s(\nu)}$. By Lemma \ref{lem:monoids}(5) it follows that $(p,q)\in\Sigma=\Sigma_v$. This proves our claim and shows that the set $\Lambda^0_{\PerG}$ is nonempty.

As a byproduct we have shown that $\Sigma=\{(p,q)\in\Z^k_+\times\Z^k_+: p-q\in\PerG(\Lambda)\}$. This implies that, conversely, if $v\in\Lambda^0_{\PerG}$, then $\Sigma=\Sigma_v$. Statements (1) and (4) of Lemma \ref{lem:monoids} show then that the set $\Lambda^0_{\PerG}$ is $G$-invariant and hereditary.
\ep

The importance of the sets $M_{\PerG}$ to us lies in the following lemma.

\begin{lemma}\label{lem:stable}
Assume $(G,\Lambda)$ is a countable pseudo-free self-similar row-finite $k$-graph without sources satisfying condition~\eqref{itcentral2}, and consider the corresponding Exel--Pardo groupoid $\G:=\G_{G,\Lambda}$. Assume $M\in\M(G,\Lambda)$, $x\in(\Lambda M)^\infty$ and $p,q,m\in\Z^k_+$ are such that $\MT_G(x)=M$, $p-q\in\PerG(\Lambda M)$ and $x(m)\in M_{\PerG}$. Then, for any $g\in G$ such that $g\cdot (x(q))=x(p)$, we have
\begin{equation}\label{eq:isointer}
Z(x(0,p),g,x(0,q))\cap\IsoGx{x}^\circ_x\ne\emptyset
\end{equation}
if and only if $x(p,p+m)=g\cdot(x(q,q+m))$ and $(x(m,m+p),g|_{x(q,q+m)},x(m,m+q))$ is a cycline triple in $\Lambda M$.
\end{lemma}

The point of this lemma is that although we knew from before that condition~\eqref{eq:isointer} is stable under small perturbations of $x$, we now have a quantitative version of this stability: all that matters is the part $x(0,m+p\vee q)$ of $x$. Note also that since $M_{\PerG}\subset M$ is a $G$-invariant hereditary (relatively to the $k$-graph $\Lambda M$) subset, for any $x\in(\Lambda M)^\infty$ with $\MT_G(x)=M$ there exists $n\in\Z^k_+$ such that $x(m)\in M_{\PerG}$ for all $m\ge n$.

\bp[Proof of Lemma~\ref{lem:stable}]
Without loss of generality we may assume that $M=\Lambda^0$. Then condition~\eqref{eq:isointer} is satisfied if and only if $\sigma^p(y)=g\cdot\sigma^q(y)$ for all $y\in\Lambda^\infty$ in a neighbourhood of $x$. Since $\sigma^p(y)=x(p,p+m)\sigma^{p+m}(y)$ and $\sigma^q(y)=x(q,q+m)\sigma^{q+m}(y)$ for all $y$ sufficiently close to $x$, the latter property can be written as a combination of two conditions: $x(p,p+m)=g\cdot(x(q,q+m))$ and $\sigma^{p+m}(y)=g|_{x(q,q+m)}\cdot\sigma^{q+m}(y)$ for all $y$ in a neighbourhood of $x$.

Since $x(m)\in \Lambda^0_{\PerG}$ and $p-q\in\PerG(\Lambda)$, there is a cycline triple $(x(m,m+p),h,\nu)$ in $\Lambda$ with $d(\nu)=q$. Then $\sigma^m(x)=x(m,m+p)\sigma^{m+p}(x)=\nu(h^{-1}\cdot\sigma^{m+p}(x))$, hence $\nu=x(m,m+q)$. We also have $\sigma^{p+m}(y)=h\cdot\sigma^{q+m}(y)$ for all $y$ close to $x$. Since condition~\eqref{itcentral2} is equivalent to essential centrality of $\G$, it follows that we have $\sigma^{p+m}(y)=g|_{x(q,q+m)}\cdot\sigma^{q+m}(y)$ for all $y$ in a neighbourhood of $x$ if and only if $g|_{x(q,q+m)}=h$, that is, if and only if $(x(m,m+p),g|_{x(q,q+m)},x(m,m+q))$ is a cycline triple in $\Lambda$.
\ep

The second part of the proof of the lemma shows that, in fact, there is at most one $g\in G$ satisfying the conditions of the lemma. Essentially the same argument shows the following: if $M\in \mathcal{M}(G,\Lambda)$ and $v\in M_{\PerG}$, then, for any $\mu \in v\Lambda^p M$ and $q\in\mathbb{Z}_{+}^{k}$ such that $p-q\in  \PerG(\Lambda M)$, there exist unique $g\in G$ and $\nu \in \Lambda^{q} M $ such that $(\mu, g, \nu)$ is a cycline triple in $\Lambda M$. We denote the element $g$ by  $g_{M,\mu, q}$.

\begin{thm}\label{thm:main-higher-rank2}
Assume we are in the setting of Theorem~\ref{thm:main-higher-rank1} and identify $\Prim C^*(\G)$ with the sets of pairs $(M,\chi)$ as explained there. Then the topology on $\Prim C^*(\G)$ is described as follows.

Assume we are given $(M, \chi)$ and $(M_{n}, \chi_{n})$ ($n\in \mathbb{N}$). For every $l\in \PerG(\Lambda M)$, fix $p(l), q(l) \in \mathbb{Z}_{+}^{k}$ such that $l=p(l)-q(l)$. Fix also a vertex $v\in M_{\PerG}$. Then $(M_{n}, \chi_{n}) \to (M, \chi)$ if and only if, for every $\varepsilon>0$, every finite set $F \subset \PerG(\Lambda M)$ and every finite path $\lambda \in v\Lambda M$ with $d(\lambda) \geq p(l)$ for all $l\in F$, we can find $n_{0} \in \mathbb{N}$ such that for each $n \geq n_{0}$ the following property is satisfied: there is a path $\mu \in \Lambda M_{n}$ such that $d(\mu) \geq d(\lambda)$, $\mu(0, d(\lambda))=\lambda$, $\mu(m)\in (M_{n})_{\PerG}$ for some $m \leq d(\mu)$ and, for every $l\in F \cap \PerG(\Lambda M_{n})$, we have $p(l) \vee q(l) \leq d(\mu)-m$ and one of the following properties holds, with $g_l:=g_{M, \lambda(0, p(l)), q(l)}$:
\begin{enumerate}
  \item $g_{l}\cdot (\mu(q(l), q(l)+m))\neq \mu(p(l), p(l)+m)$;
  \item $(\mu(m, m+p(l)), g_{l}|_{\mu(q(l), q(l)+m)}, \mu(m, q(l)+m))$ is not a cycline triple in $\Lambda M_{n}$, equivalently,
  $  g_{l}|_{\mu(q(l), q(l)+m)}\ne g_{M_n,\mu(m, m+p(l)),q(l)}$;
  \item $|\chi(l)-\chi_{n}(l) | < \varepsilon$.
\end{enumerate}
\end{thm}

\begin{proof}
As usual, we write $[y]$ for $[y]_\G$. We start with the ``only if'' direction, so assume $(M_{n}, \chi_{n}) \to (M, \chi)$. Let $\eps$, $F$ and $\lambda$ be as in the statement of the theorem. Put $N:=d(\lambda)\vee\bigvee_{l\in F}q(l)$. We can find $x, x_{n}\in \Lambda^{\infty}$ such that $\MT_G(x)=M$, $x(0,d(\lambda))=\lambda$ and $\MT_G(x_{n})=M_{n}$.  Consider the bisections $W_l:=Z(x(0,p(l)),g_l,x(0,q(l)))$ for $l\in F$. Then by condition (2) in Theorem~\ref{thm:convfg} we can find $n_0\in\N$ such that for each $n \geq n_{0}$ there is $y\in[x_n]\cap Z(x(0,N))$ satisfying the following property: for every $l\in F$, either $W_{l} \cap \IsoGx{x_n}^\circ_y = \emptyset$ or $|\chi(l)-\chi_{n}(l) | < \varepsilon$. Working with fixed $n\ge n_0$ and $y$, let us look closer at the condition $W_{l} \cap \IsoGx{x_n}^\circ_y = \emptyset$.

Let $m\in\Z^k_+$ be such that $y(m)\in(M_n)_{\PerG}$. Put $\mu:=y(0,m+N)$. If $l\in F\cap\PerG(\Lambda M_n)$, then by Lemma~\ref{lem:stable} the condition $W_{l} \cap \IsoGx{x_n}^\circ_y = \emptyset$ is satisfied if and only if either (1) or (2) is true. Therefore, for every $l\in F\cap\PerG(\Lambda M_n)$, one of the conditions (1)--(3) is satisfied.

\smallskip

Conversely, assume we can find paths $\mu$ as in the statement of theorem. Let $z,z(n)\in\T^k$ be any characters extending $\chi$ and $\chi_n$. Choose a path $x\in v\Lambda^\infty$ such that $\MT_G(x)=M$. For $l\in\PerG(\Lambda M)$, consider the bisections
$$
W_l:=Z(x(0,p(l)),g_{M,x(0,p(l)),q(l)},x(0,q(l))).
$$
Taking $\eps>0$, a finite set $F\subset\PerG(\Lambda M)$ and $\lambda:=x(0,N)$ for some $N\ge \bigvee_{l\in F}p(l)$, by assumption we can find $n_0\in\N$ such that for each $n\ge n_0$ there is a path $\mu\in\Lambda M_n$ with properties as in the statement of the theorem. With $n\ge n_0$ fixed, choose any path $y\in\mu\Lambda^\infty$ such that $\MT_G(y)=M_n$. Since $W_l \cap \IsoGx{y}^\circ_y = \emptyset$ for $l\notin\PerG(\Lambda M_n)$, conditions (1)--(3) and Lemma~\ref{lem:stable} imply that for every $l\in F$ we have either $W_{l} \cap \IsoGx{y}^\circ_y = \emptyset$, or $l\in\PerG(\Lambda M_n)$ and $|\chi(l)-\chi_{n}(l) | < \varepsilon$.

Using such paths $y$, a standard diagonal argument allows us to construct a sequence of paths~$y_n$ such that $\MT_G(y_n)=M_n$, $y_n\to x$ and $z(n)\to z$ along the sets $\{l\in\PerG(\Lambda M): W_{l} \cap \IsoGx{y_n}^\circ_{y_n}\ne\emptyset \}$. By Theorem~\ref{thm:convfg} we conclude that $(M_n,\chi_n)\to(M,\chi)$.
\ep

\bigskip

\end{document}